\documentclass[11pt,a4paper]{article}
\usepackage{amsmath}
\usepackage{amssymb}
\usepackage{amsfonts}
\usepackage{amsthm}
\usepackage{relsize}
\usepackage{setspace}
\usepackage{geometry}
\usepackage{enumerate}
\usepackage{url}
\usepackage{xspace}
\usepackage{mathrsfs}
\usepackage{tocloft}

\usepackage{hyperref}
\usepackage{authblk}
\usepackage[usenames, dvipsnames]{color}

\usepackage{makeidx}
\makeindex

\newtheorem{thm}{Theorem}[section]
\newtheorem{lem}[thm]{Lemma}
\newtheorem{prop}[thm]{Proposition}
\newtheorem{cor}[thm]{Corollary}

\newtheorem{thmintro}{Theorem}

\theoremstyle{definition}
\newtheorem{defn}[thm]{Definition}

\newtheorem{que}[thm]{Question}
\newtheorem{rem}[thm]{Remark}

\newtheorem{ex}[thm]{Example}

\newcommand{\defbold}{\textbf}

\newcommand{\inv}{^{-1}}

\newcommand{\VA}{{[\mathrm{A}]}}

\newcommand{\QZ}{\mathrm{QZ}}
\newcommand{\QC}{\mathrm{QC}}

\newcommand{\CC}{\mathrm{C}}
\newcommand{\QQ}{\mathbf{Q}}

\newcommand{\N}{\mathrm{N}}
\newcommand{\R}{\mathrm{R}}
\newcommand{\Z}{\mathrm{Z}}

\newcommand{\tdlc}{t.d.l.c.\@\xspace}

\newcommand{\ldlat}{\mathcal{LD}}

\newcommand{\rist}{\mathrm{rist}}

\newcommand{\lnorm}{\mathcal{LN}}

\newcommand{\lcent}{\mathcal{LC}}

\newcommand{\Aut}{\mathrm{Aut}}

\newcommand{\Comm}{\mathrm{Comm}}

\newcommand{\Out}{\mathrm{Out}}

\newcommand{\LN}{\mathrm{LN}}

\newcommand{\bN}{\mathbf{N}}

\newcommand{\bQ}{\mathbf{Q}}

\newcommand{\bZ}{\mathbf{Z}}

\newcommand{\mcA}{\mathcal{A}}
\newcommand{\mcB}{\mathcal{B}}
\newcommand{\mcC}{\mathcal{C}}

\newcommand{\mcI}{\mathcal{I}}

\newcommand{\mcL}{\mathcal{L}}
\newcommand{\mcM}{\mathcal{M}}
\newcommand{\mcN}{\mathcal{N}}

\newcommand{\mcP}{\mathcal{P}}

\newcommand{\mcT}{\mathcal{T}}

\newcommand{\mfp}{\mathfrak{p}}

\newcommand{\mfS}{\mathfrak{S}}
\newcommand{\mfX}{\mathfrak{X}}
\newcommand{\mfY}{\mathfrak{Y}}

\begin{document}

\title{Locally normal subgroups  \\ of totally disconnected groups. \\ Part I: General theory}


\author[1]{Pierre-Emmanuel Caprace\thanks{F.R.S.-FNRS research associate, supported in part by the ERC (grant \#278469)}}
\author[2]{Colin D. Reid\thanks{Supported in part by ARC Discovery Project DP120100996}}
\author[2]{George A. Willis\thanks{Supported in part by ARC Discovery Project DP0984342}}

\affil[1]{Universit\'e catholique de Louvain, IRMP, Chemin du Cyclotron 2, bte L7.01.02, 1348 Louvain-la-Neuve, Belgique}
\affil[2]{School of Mathematical and Physical Sciences, University of Newcastle, Callaghan, NSW 2308, Australia}

\date{April 22, 2015}

\maketitle

\begin{abstract}
Let $G$ be a totally disconnected, locally compact group.  A closed subgroup of $G$ is \textbf{locally normal} if its normaliser is open in $G$.  We begin an investigation of the structure of the family of closed locally normal subgroups of $G$.  Modulo commensurability, this family forms a modular lattice $\lnorm(G)$, called the \textbf{structure lattice} of $G$.  We show that $G$ admits a canonical maximal quotient $H$ for which the quasi-centre and the abelian locally normal subgroups are trivial.  In this situation $\lnorm(H)$ has a canonical subset called the \textbf{centraliser lattice}, forming a Boolean algebra whose elements correspond to centralisers of locally normal subgroups.  If $H$ is second-countable and acts faithfully on its centraliser lattice, we show that the topology of $H$ is determined by its algebraic structure (and thus invariant by every abstract group automomorphism), and also that the action on the Stone space of the centraliser lattice is universal for a class of actions on profinite spaces. Most of the material is developed in the more general framework of Hecke pairs.
\end{abstract}

\tableofcontents

\newpage

\section{Introduction}

{
The aim of the present article is to establish foundations for a study of general non-discrete totally disconnected locally compact (\tdlc)\index{tdlc@\tdlc} groups in terms of their local structure. We define a canonical lattice in terms of this local structure that, in the case of a $p$-adic Lie group, is equivalent to the lattice of ideals in the Lie algebra. Inspired by earlier work of J.~Wilson~\cite{WilNH} on just-infinite groups and by Barnea--Ershov--Weigel~\cite{BEW} on abstract commensurators of profinite groups, we develop properties of this lattice. This article, which concerns general \tdlc groups, is the first in a series. The second paper in the series \cite{CRW-Part2} uses the tools developed here to study the more special class of compactly generated \tdlc groups with few normal subgroups and, in particular, those that are topologically simple.  
}

Most of the definitions and results in this first part apply more generally to a \defbold{Hecke pair} $(G,U)$\index{Hecke pair}, that is a group $G$ with a commensurated subgroup $U$, in a way that depends only on the (abstract) commensurability class of $U$ as a subgroup of $G$.  Although our principal interest is in \tdlc groups, we will state and prove results in this more general framework in situations where it is unnecessary to use the topology of the group.  For example, in the case $G = U$, most of the results we obtain will be valid for residually finite groups in general, with profinite groups as a special case.

\addtocontents{toc}{\protect\setcounter{tocdepth}{1}}
\subsection{Background}

Since the solution of Hilbert's fifth problem by Gleason, Montgomery--Zippin and Yamabe, connected locally compact groups are known to have a far-reaching structure theory, ensuring notably that they are all projective limits of connected Lie groups. This deep result extends the powerful machinery of Lie theory to all connected locally compact groups, thereby yielding  decomposition theorems and classification results.

In contrast, the solution of Hilbert's fifth problem yields no information about \defbold{totally disconnected}\index{totally disconnected locally compact group} locally compact groups, that is, those in which the connected component of each element is a singleton. Understanding this case would go a long way towards an understanding of general locally compact groups because each general group is an extension of a totally disconnected group by a connected one. While a rather rich and deep theory was developed for some special classes of \tdlc groups (profinite groups, reductive groups over local fields, $p$-adic analytic groups, \dots), for a long time the only general result known was van Dantzig's theorem, which asserts that compact open subgroups exist and form a basis of identity neighbourhoods. Since the class of \tdlc groups includes all discrete groups, it seemed during this time that no meaningful structure theory could be developed beyond the connected or almost connected case. 

Three developments over the last twenty years have changed the situation. First, new general results were obtained with the aid of the \defbold{scale function}\index{scale function},  introduced by Willis in \cite{Willis94}. The scale function on the  \tdlc group~$G$ measures the effect on compact open subgroups of conjugation by elements of $G$; it takes positive-integer values and is non-trivial only when $G$ is not discrete. The scale function and associated ideas elicit significant new features of general non-discrete \tdlc groups such as {tidy} subgroups, contraction groups, parabolic subgroups, {\it etc\/}. and have resolved several open questions.  

A complementary approach was initiated by Barnea--Ershov--Weigel \cite{BEW}, based on the following observation: since any two compact open subgroups of a \tdlc group $G$ are commensurate, $G$ admits a natural homomorphism to the group of abstract  commensurators of its compact open subgroups. By definition, abstract commensurators can be recovered from any compact open subgroup and are thus purely local objects depending only on arbitrarily small identity neighbourhoods of $G$. This observation thus produces a \textbf{local approach}\index{local approach} to \tdlc groups. 

The third development is that techniques for isolating the non-discrete case from the discrete one have been found. At the global level, Caprace--Monod \cite{CM} identified simple pieces of a \tdlc group $G$ under the hypothesis that $G$ has no discrete quotients while, at the local level, Burger--Mozes \cite{BurgerMozes} introduced, in analogy with the kernel of the adjoint representation of a Lie group, the \defbold{quasi-centre}\index{quasi-centre} of $G$, which contains in particular any discrete normal subgroup of $G$. The former of the two motivates the focus of our next paper~\cite{CRW-Part2} on simple groups, while the latter was important in the approach introduced in~\cite{BEW} and is developed further in Sections~\ref{sec:QZmain} and~\ref{radsec} below as part of the foundation for a general theory of the local structure of \tdlc groups. 

\subsection{Locally normal subgroups and their centralisers}

The central concept in the general local theory is that of a \defbold{locally normal}\index{locally normal subgroup} subgroup of a \tdlc group, that is, a subgroup whose normaliser has finite index in an open subgroup.  A major part of \cite{BEW} focuses on \tdlc groups in which every open compact subgroup is (hereditarily) just infinite; this corresponds exactly to the condition that every non-trivial compact locally normal subgroup is open.  Away from this case, the arrangement of (closed) locally normal subgroups becomes more complex.

The intersection of two locally normal subgroups is locally normal. Moreover, if $H$ and $K$ are  closed locally normal subgroups of $G$, then they admit  respective compact open subgroups $H' < H$ and $K' < K$ which normalise each other, and such that the product $H'  K'$, which is also a compact subgroup, is locally normal. It is thus natural to consider the collection of all  compact locally normal subgroups modulo the equivalence relation given by \defbold{local equivalence}\index{local equivalence}: two subgroups $H$ and $K$ are \defbold{locally equivalent} if their intersection $H \cap K$ is open in both $H$ and $K$. In this way, the collection of all equivalence classes of  compact locally normal subgroups of $G$, which is denoted by $\lnorm(G)$\index{LN(G)@$\lnorm(G)$}, endowed with the partial ordering induced by the inclusion of  subgroups, is easily seen to form a modular lattice.  We call it the \defbold{structure lattice}\index{structure lattice}\index{lattice!structure} of $G$. There are two elements that we regard as trivial (assuming $G$ is non-discrete): the minimum element, which is the class of the trivial subgroup, and the maximum element, which is the class of compact open subgroups. The symbols $0$ and $\infty$ will be used to denote the minimum and maximum elements respectively of any bounded lattice, including $\lnorm(G)$. A non-trivial subset of a bounded lattice is one that contains elements other than $0$ and $\infty$. 

Its structure lattice is a \emph{local} invariant of a \tdlc group $G$, since it is determined by any open compact subgroup, which admits a natural action of the topological automorphism group of $G$.  {In case $G$ is a $p$-adic Lie group\index{p-adic Lie group@$p$-adic Lie group}, the structure lattice $\lnorm(G)$ is canonically isomorphic to the lattice of ideals in the $\QQ_p$-Lie algebra of $G$ (see Proposition~\ref{prop:p-adicLie} below). Hence the structure lattice may be viewed as a group theoretic analogue for \emph{arbitrary} \tdlc groups of the lattice of ideals in the $p$-adic case. Although it is difficult to determine the structure lattice of a particular \tdlc group $G$ explicitly, the action of $G$ on $\lnorm(G)$ provides an important tool for elucidating the overall structure of $G$. Moreover, under the right conditions we obtain two natural $G$-invariant subsets of $\lnorm(G)$ that are Boolean algebras when equipped with the induced order and, when non-trivial, constitute valuable additional tools.  The first is the \defbold{local decomposition lattice}\index{local decomposition lattice}\index{lattice!local decomposition} $\ldlat(G)$\index{LD(G)@$\ldlat(G)$}, which consists of all direct factors of open compact subgroups of $G$ up to equivalence.  The second is the \defbold{centraliser lattice}\index{centraliser lattice}\index{lattice!centraliser} $\lcent(G)$\index{LC(G)@$\lcent(G)$}: an element of $\lcent(G)$ is given by the equivalence class of the centraliser $\CC_U(K)$\index{CU(K)@$\CC_U(K)$} for some open compact subgroup $U$ of $G$ and closed normal subgroup $K$ of $U$.  In general, the structure lattice depends on the topology of the open compact subgroups of $G$, but the centraliser and decomposition lattices depend only on the Hecke pair structure $(G,U)$, where $U$ is any open compact subgroup of $G$.  

Following Burger--Mozes \cite{BurgerMozes}, we define the \defbold{quasi-centre}\index{quasi-centre} $\QZ(G)$\index{QZ(G)@$\QZ(G)$} of a locally compact group $G$ to be the collection of all elements whose centraliser is open (see Lemma~\ref{lem:QZ} below for its basic properties).

\begin{thmintro}\label{intro:boolean}
Let $G$ be a \tdlc group.
\begin{enumerate}[(i)]

\item \emph{(See Theorem~\ref{dfaclat})} Suppose that $G$ has trivial quasi-centre. Then $\ldlat(G)$ is a Boolean algebra.

\item \emph{(See Theorem~\ref{centlat})} Suppose in addition that $G$ has no non-trivial abelian  locally normal subgroups.  Then $\lcent(G)$ is a Boolean algebra and $\ldlat(G)$ is a subalgebra of $\lcent(G)$.

\end{enumerate}
\end{thmintro}

This opens up a new perspective: using the Stone representation theorem, a group action on a Boolean algebra $\mcA$ can equivalently be regarded as a group action by homeomorphisms on a topological space $\mfS(\mcA)$, called the \defbold{Stone space}\index{Stone space} of $\mcA$, that is \defbold{profinite}\index{profinite space}, in other words, compact and totally disconnected. An important property of the $G$-actions on the Stone spaces of $\ldlat(G)$ and $\lcent(G)$ is encapsulated in the following definition. Let $G$ be a locally compact group and  $\mfX$ be a profinite space on which $G$ acts faithfully by homeomorphisms. The $G$-action is called \textbf{weakly decomposable}\index{weakly decomposable action}\index{action!weakly decomposable} if for every non-empty clopen subset $\alpha \subsetneq \mfX$, the stabiliser of $\alpha$ in $G$ is open, and the pointwise stabiliser of $\alpha$ is non-trivial.  Say that $G$ is \defbold{faithful weakly decomposable}\index{faithful weakly decomposable} if it admits a faithful weakly decomposable action on some profinite space. 

\begin{thmintro}[See Theorem~\ref{latdecomp}]\label{intro:weakdecomp}
Let $G$ be a \tdlc group such that $\QZ(G)=1$.
\begin{enumerate}[(i)]
\item Suppose that the only abelian locally normal subgroup of $G$ is trivial, and suppose that $G$ acts faithfully on $\lcent(G)$.  Then the induced action of $G$ on $\mfS(\lcent(G))$ is a faithful weakly decomposable action.
\item Suppose $\mfX$ is a profinite space with a faithful weakly decomposable action of $G$.  Then the only abelian locally normal subgroup of $G$ is trivial and there exists a $G$-equivariant continuous surjective map from $\mfS(\lcent(G))$ to $\mfX$.
\end{enumerate}
\end{thmintro}

In particular, if $G$ is faithful weakly decomposable, then 
$\mfS(\lcent(G))$ is the unique maximal compact $G$-space for which the action is faithful and weakly decomposable.

\subsection{Radical theory}

The hypotheses of Theorem~\ref{intro:boolean} are satisfied by all compactly generated, topologically simple \tdlc groups: for (i) this is Theorem 4.8 of \cite{BEW}, and for (ii) this will be proved in \cite[Theorem~A]{CRW-Part2}. We will show, by developing a suitable `radical theory', that an \emph{arbitrary} \tdlc group has a canonical maximal quotient which satisfies the hypotheses of Theorem~\ref{intro:boolean}(i) or (ii). 

\begin{thmintro}\label{intro:radical}
Let  $G$ be a \tdlc group.

Then $G$ has closed characteristic subgroups $\QZ^\infty(G) \leq \R_\VA(G)  \leq G$\index{QZinfty(G)@$\QZ^\infty(G)$}\index{RA(G)@$\R_\VA(G)$} with the following properties:

\begin{enumerate}[(i)]
\item \emph{(See Theorem~\ref{qzinf})} $\QZ^\infty(G)$ is the unique smallest closed normal subgroup $N$ of $G$ such that $G/N$ has trivial quasi-centre. 

\item \emph{(See Theorem~\ref{radtdlc})} $\R_\VA(G) $ is the unique smallest closed normal subgroup $M$ of $G$ such that $G/M$ has trivial quasi-centre and no non-trivial abelian locally normal subgroups. 
\end{enumerate}

\end{thmintro}

In case $G$ is a $p$-adic Lie group\index{p-adic Lie group@$p$-adic Lie group}, $\R_\VA(G) $ is the unique largest closed normal subgroup whose Lie algebra is the soluble radical of $\mathrm{Lie}(G)$; see Proposition~\ref{prop:Lieradical} below.

Both $\QZ^\infty(G)$ and $\R_\VA(G)$ are well behaved with respect to open subgroups, in the sense that $\QZ^\infty(H) = \QZ^\infty(G) \cap H$ and $\R_\VA(H) = \R_\VA(G) \cap H$ for any open subgroup $H$ of $G$ (see Propositions \ref{qzinfstab} and \ref{radtdlcstab}).  Thus if the canonical quotients provided by Theorem~\ref{intro:radical} turn out to be trivial, we have additional local information about the structure of $G$.  For instance it is easily seen that if $\R_\VA(G) = G$, then the open compact subgroups of $G$ are Fitting-regular in the sense of \cite{ReiF}.

\subsection{Group topologies}

We now specialise to the case of first-countable \tdlc groups and give some results on the relationship between the commensurated compact subgroups and the compact locally normal subgroups. For instance (Lemma~\ref{profcomm}), a closed subgroup $K$ of a first-countable profinite group $U$ is commensurated by $U$ if and only if $U$ normalises an open subgroup of $K$. Hence, any compact commensurated subgroup of a first-countable \tdlc group $G$ is commensurate with a locally normal subgroup, whose commensurability class is thus a $G$-fixed point in $\lnorm(G)$. As a consequence, the commensurated compact locally normal subgroups have an interpretation in terms of group topologies.

\begin{thmintro}[See \S~\ref{subsec:refine}]\label{intro:refine}
Let $G$ be a first-countable \tdlc group.  Then there is a natural bijection between the elements of $\lnorm(G)$ that are fixed by the action of $G$, and refinements of the topology of $G$ that are locally compact and compatible with the group structure.
\end{thmintro}

When $G$ is faithful weakly decomposable, we obtain a topological rigidity result of a similar kind to those for semisimple Lie groups over real and $p$-adic fields given in \cite{Kra} and \cite{McC}.

\begin{thmintro}[See Theorem~\ref{refpropcent}]\label{intro:rigid}Let $G$ be a second-countable \tdlc group with trivial quasi-centre, and suppose that $G$ is faithful weakly decomposable.  Then the topology of $G$ is the coarsest group topology on $G$ for which it is a Baire space. In particular, it is preserved by every automorphism of $G$ as an abstract group.  Consequently, the topology of $G$ is the unique $\sigma$-compact locally compact group topology on $G$, and the set of all locally compact group topologies of $G$ is in natural bijection with the set of fixed points of $G$ acting on $\lnorm(G)$ by conjugation.
\end{thmintro}

\subsection*{Acknowledgements}
We express our gratitude to John Wilson for numerous comments which helped in improving the presentation of this paper. We thank Corina Ciobotaru, David Hume, Thierry Stulemeijer and Phillip Wesolek for their careful reading of an earlier draft of the manuscript.

\addtocontents{toc}{\protect\setcounter{tocdepth}{2}}
\section{Locally normal subgroups}

\subsection{Definition and first properties}\label{sec:BasicDefs}

Let $G$ be a \tdlc group, which is assumed to be Hausdorff by convention.  Then van Dantzig's theorem \cite[Theorem~II.7.7]{Hewitt&Ross} asserts that the open compact subgroups of $G$ form a base of identity neighbourhoods. Denote this set by $\mcB(G)$\index{B(G)@$\mcB(G)$}. Any two elements of $\mcB(G)$ are commensurate, as an elementary topological argument shows.  Hence, letting $U$ be any open compact subgroup of $G$, the pair $(G,U)$ satisfies the following condition.  

A \defbold{Hecke pair (of groups)}\index{Hecke pair|textbf} $(G,U)$ is a group $G$ together with a commensurated\index{commensurated subgroup} subgroup $U$ of $G$, that is, $U \cap gUg\inv$ has finite index in $U$ for all $g \in G$.

Most of our arguments will be made in terms of general Hecke pairs, as the results then apply directly to residually finite groups and not just to their profinite completions. When specialising to the case that $U$ is an open compact subgroup of the \tdlc group $G$, we will say that $(G,U)$ is a \textbf{\tdlc}\index{Hecke pair!\tdlc}\index{tdlc@\tdlc!Hecke pair} Hecke pair.


Given a Hecke pair $(G,U)$, a \defbold{locally normal}\index{locally normal subgroup!of a Hecke pair} subgroup of $(G,U)$ is a subgroup $K$ of $G$ such that the normaliser $\N_U(K)$\index{NU(K)@$\N_U(K)$} has finite index in $U$.  A subgroup $K$ of $G$ is \defbold{bounded}\index{bounded subgroup!of a Hecke pair} if $K\cap U$ is commensurate with $K$. When $(G, U)$ is \tdlc, subgroups that are bounded and closed are precisely the compact ones. 

The most elementary properties of locally normal subgroups collected in the following will frequently be used without comment. 

\begin{lem}\label{lem:LN:basic}
Let $(G,U)$ be a Hecke pair and $K, L$ be subgroups of $G$. 
\begin{enumerate}[(i)]
\item If $K$ and $L$ are locally normal, so is $K \cap L$. 

\item If $(G,U)$ is \tdlc and $K$ is compact and locally normal, so is every open subgroup of $K$.

\item If $K$ is locally normal, so is the centraliser $\CC_V(K)$ for any $V\leq G$ commensurate with $U$.

\item If $K$ is bounded, then there is $V\leq G$ commensurate with $U$ that contains $K$. If $K$ is also locally normal, then $V$ may be chosen to normalise $K$. When $(G, U)$ is \tdlc and $K$ is compact and locally normal, then we can take $V \in \mcB(G)$.

\end{enumerate}
\end{lem}

\begin{proof}
(i) is clear since $K \cap L$ is normalised by $\N_U(K) \cap \N_U(L)$. 

(ii) An open subgroup of $K$ may be written as $K \cap U$ for some $U \in \mcB(G)$.  Since any element of $\mcB(G)$ is locally normal, (ii) follows from part (i).

(iii) Let $W = U \cap V$.  If $K$ is locally normal, then $\CC_V(K)$ is a subgroup of $V$, so that $\N_W(K) \le \N_W(\CC_V(K))$ and we have 
$$ |U:\N_U(\CC_V(K))| \le |U:W| |W:\N_{W}(\CC_V(K))| \le |U:W| |W:\N_W(K)| < \infty;$$
hence $\CC_V(K)$ is locally normal.

(iv) If $K$ is bounded, then $W = \bigcap_{k\in K} kUk^{-1}$ is commensurate with $U$ and normalised by $K$, so that $V = KW$ is a subgroup commensurate with $U$ and containing $K$. If $K$ locally normal, then $W = \N_U(K)$ is commensurate with $U$ and $V = KW$ is a subgroup commensurate with $U$ and containing $K$ as a normal subgroup.  When $(G, U)$ is \tdlc and $K$ is compact and locally normal, then $W$ open and compact with either definition, and so $V$ is an open compact subgroup of $G$.
\end{proof}

Notice that the converse to Lemma~\ref{lem:LN:basic}(ii) does not hold:  namely, the fact that  $K$ is compact and contains a locally normal subgroup as an open subgroup does not generally imply that $K$ itself is locally normal, since otherwise all finite subgroups of $G$ would be locally normal and hence contained in the quasi-centre $\QZ(G)$. 

\subsection{Local equivalence and the structure lattice}

Let $(G,U)$ be a Hecke pair and let $H$ be a subgroup of $G$. Write $[H]$ for the class of all subgroups $K$ of $G$ such that $K \cap U$ is commensurate with $H \cap U$ and  say that $H$ and $K$ are \defbold{locally equivalent}\index{local equivalence|textbf} if $[H]=[K]$. Although $H$ might not be bounded, and such will be the case in later arguments, it is the bounded subgroups of $[H]$ that are of interest. There is a natural partial ordering on local equivalence classes, given by $[H] \ge [K]$ if $H \cap K \cap U$ has finite index in $K \cap U$. Note that, when $(G,U)$ is \tdlc and $H$ is closed, the closed subgroups in $[H]$ are those that have an open subgroup in common with $H$.

A subgroup $H$ in $G$ is locally equivalent to the trivial subgroup if and only if $H \cap U$ is finite; in the \tdlc case, this occurs if and only if $H$ is discrete. We shall later be interested in subgroups $H$ of \tdlc groups $G$ such that $H \cap U =1$ for all $U \in \mcB(G)$. Since every compact (and hence every finite) subgroup of $G$ is contained in some open compact subgroup, the latter property holds if and only if $H$ is discrete and torsion-free.

For a given a Hecke pair $(G,U)$, let $\lnorm(G,U)$ denote the set of local equivalence classes of subgroups of $G$ that have locally normal representatives.  Unless otherwise stated, when choosing a representative of $\alpha \in \lnorm(G,U)$, we only take representatives that are bounded and locally normal. Then the partially ordered set $\lnorm(G,U)$ is a lattice. The greatest lower bound and least upper bound of $\alpha, \beta \in \lnorm(G,U)$ are respectively 
\[\alpha \wedge \beta = [H \cap K]\hbox{ and } \alpha \vee \beta = [HK],\]
where representatives $H \in \alpha$ and $K \in \beta$ are chosen such that $H,K \unlhd V$ for some $V\leq U$ of finite index. 
That $[H \cap K]$ and $[HK]$ are independent of how $H$, $K$ and $V$ are chosen and that they are respectively the greatest lower bound and least upper bound of $\alpha$ and $\beta$ is easily verified. We call $(\lnorm(G,U),\leq)$\index{LN(G,U)@$\lnorm(G,U)$} the \defbold{abstract structure lattice}\index{abstract structure lattice}\index{structure lattice!abstract} of $(G,U)$.

If $(G,U)$ is a \tdlc Hecke pair, write $\lnorm(G)$ for the set of elements of $\lnorm(G,U)$ that have compact locally normal representatives.  Then $\lnorm(G)$ is a sublattice of $\lnorm(G,U)$ that does not depend on the choice of $U$ so long as it is compact and open.  We define $\lnorm(G)$ with the induced operations to be the \defbold{structure lattice} of $G$.

It is clear that $\lnorm(G,U)$ is determined entirely by the group structure of $U$, so effectively $\lnorm(G,U) = \lnorm(U,U)$.  We retain the notation $\lnorm(G,U)$ in order to distinguish the abstract (Hecke pair) structure lattice from the topological structure lattice $\lnorm(G)$, which is more useful in the study of \tdlc groups.

\begin{ex} \

\begin{enumerate}[(a)]
\item Let $T$ be a regular tree of finite degree $d$, with $d \ge 3$, and let $G = \Aut(T)$.  Then $G$ is a \tdlc group, with topology generated by the pointwise stabilisers of finite subtrees of $T$.  Let $F$ be a finite subtree of $T$, let $U$ be the pointwise stabiliser of $F$, and let $T_1,\dots,T_n$ be the connected components of $T - F$.  Then $U$ decomposes as a direct product
\[ U = K_1 \times K_2 \times \dots \times K_n\]
where $K_i$ is the pointwise stabiliser in $G$ of all vertices outside $T_i$.  Evidently $K_i$ is compact and normal in $U$, so it is locally normal in $G$, and since each tree $T_i$ is infinite and $G$ induces an infinite group of automorphisms on it, each group $K_i$ is infinite.  Thus the given decomposition of $U$ provides $2^n$ distinct elements of the structure lattice of $G$.  By varying the finite tree $F$, we see that $\lnorm(G)$ is infinite in this case.

\item Let $G = \bQ^n_p$ for $n \ge 2$.  Then the every subgroup of $G$ is locally normal, since $G$ is abelian.  The compact subgroups of $G$ are all isomorphic to $\bZ^m_p$ where $0 \le m \le n$.  It is easily seen that two compact subgroups of $G$ are commensurate if and only if they span the same subspace of $G$ regarded as a $\bQ_p$-vector space, and conversely every $\bQ_p$-subspace of $G$ is spanned by a compact subgroup.  Consequently, the structure lattice of $G$ is naturally isomorphic to the lattice of $\bQ_p$-subspaces of $G$.  In particular, $\lnorm(G)$ is uncountable.

\item At the other extreme, let $G$ be any \tdlc group with a hereditarily just infinite open subgroup $U$, for instance $G = \bQ_p$ or $G = \mathrm{PSL}_n(\bQ_p)$.  Then given a closed locally normal subgroup $K$, the intersection $K \cap U$ is normal in an open subgroup of $U$, and hence $K$ is either finite or open.  Hence $\lnorm(G)$ is the two-element lattice.
\end{enumerate}
\end{ex}

When $G$ is a \tdlc group, the lattice structure of $\lnorm(G)$ is preserved by the topological automorphism group $\Aut(G)$ of $G$.  More generally, the structure of $\lnorm(G,U)$ is preserved by those automorphisms of $G$ that preserve the commensurability class of $U$.  In particular, $G$ acts on $\lnorm(G)$ (when defined) and $\lnorm(G,U)$ by conjugation.  If $U$ is infinite, there are two `trivial' elements of $\lnorm(G,U)$ and of $\lnorm(G)$, namely the class of the trivial group, denoted by $0$, and the class of the subgroups containing a subgroup of $U$ of finite index, denoted by $\infty$.  

The use of the term \emph{structure lattice} is borrowed from J.~Wilson (\cite{WilNH}), whose ideas developed in the context of just infinite groups were an important source of inspiration for the present work.  The structure lattice is a purely `local' invariant of a \tdlc group, that is, it is completely determined by any open compact subgroup.  It is also built out of lattices of normal subgroups of open compact subgroups in a natural sense.  Similar observations apply more generally to the abstract structure lattice.

\begin{lem}Let $(G,U)$ be a Hecke pair.
\begin{enumerate}[(i)]
\item Let $X$ be a finite subset of $\lnorm(G,U)$, and let $\mcL$ be the sublattice generated by $X$.  Then there is a subgroup $V$ of $U$ of finite index and a choice of representatives $Y$ for the elements of $X$, so that every element of $Y$ is a normal subgroup of $V$, and $\mcL$ is isomorphic to the lattice of subgroups of $V$ generated by $Y$ modulo commensuration.  If $\mcL \subseteq \lnorm(G)$, then we can take $V \in \mcB(G)$.

\item Let $I$ be a lattice identity that is satisfied by the normal subgroup lattice of $V$ (modulo commensuration), for every finite index subgroup $V$ of $U$.  Then $I$ is satisfied by $\lnorm(G,U)$.  In particular, $\lnorm(G,U)$ is a modular lattice.

\end{enumerate}
\end{lem}

\begin{proof}(i) Let $X = \{[K_1],\dots,[K_n]\}$, let $V = \bigcap \N_U(K_i)$, and let $L_i = V \cap K_i$ for $i \in \{1,\dots,n\}$.  Then $Y = \{L_1,\dots,L_n\}$ consists of normal subgroups of $V$, and the meet and join operations in $\mcL$ can evidently be given by $[K_i] \wedge [K_j] = [L_i \cap L_j]$ and $[K_i] \vee [K_j] = [L_iL_j] = [\langle L_i,L_j\rangle]$.  Thus the map $L \mapsto [L]$ gives a surjective map from the lattice generated by $Y$ to $\mcL$, and two subgroups have the same image if and only if they are commensurate.  If $(G,U)$ is \tdlc and $\mcL \subseteq \lnorm(G)$, then we can choose $K_1,\dots,K_n$ to be closed, so $V$ is closed, and hence $V \in \mcB(G)$.

(ii) This is immediate from (i), recalling that the lattice of normal subgroups in any group is modular (see \cite[Theorem~8.3]{JacobsonBAI}).
\end{proof}

If $G$ is a \tdlc group and $U$ is an open compact subgroup, we see that $\lnorm(G)$ is a sublattice of $\lnorm(G,U)$.  (In particular, $\lnorm(G)$ inherits the modular property from $\lnorm(G,U)$.)  In general, elements of $\lnorm(G)$ may have non-closed representatives, since a finite index subgroup of a profinite group is not necessarily closed.  However, we can recognise the elements of $\lnorm(G)$ inside $\lnorm(G,U)$ by the existence of closed representatives (without reference to normalisers).

\begin{prop}\label{prop:closed_rep}Let $(G,U)$ be a \tdlc Hecke pair and let $\alpha \in \lnorm(G,U)$.  Then $\alpha \in \lnorm(G)$, that is, $\alpha$ has a compact locally normal representative, if and only if there is closed subgroup $H$ of $G$ such that $[H] = \alpha$.
\end{prop}

\begin{proof}
If $\alpha \in \lnorm(G)$, then certainly $\alpha$ has a closed representative.  It remains to show the converse.  Suppose $H$ is a closed subgroup of $G$ such that $[H] = \alpha$.  Since $\alpha \in \lnorm(G,U)$, there is some bounded locally normal subgroup $K$ of $G$ such that $[K] = \alpha$.  We see that $H \cap U$ is closed and $[H \cap U] = \alpha$, so we may assume $H \le U$.  Now $H \cap K$ has finite index in $K$, so the image of $K$ in $U/H$ is a finite set of cosets, say $k_1 H, \dots, k_n H$. Since $H$ is closed, there is a sufficiently small open normal subgroup $V$ of $U$ such that  the cosets $k_1 H, \dots, k_n H$ remain pairwise distinct modulo $V$. It then follows that
$$ H \cap K \cap V = K \cap V.$$
Note that $K \cap V$ is locally normal and $[K \cap V] = \alpha$.  So by replacing $K$ by $K \cap V$, we may assume $K \le H$.  Since $H$ is a closed subgroup of $G$, we in fact have $\overline{K} \le H$.  Now $K$ has finite index in $\overline{K}$, since $K$ is commensurate with $H$, so $[\overline{K}] = \alpha$; moreover, the normaliser of $\overline{K}$ in $G$ is a closed subgroup containing $\N_G(K)$, so the normaliser of $\overline{K}$ is open.  Hence $\overline{K}$ is the desired compact locally normal representative of $\alpha$.\end{proof}

{We have seen that the structure lattice $\lnorm(G)$ of a \tdlc group $G$ is a modular lattice that can be recovered from any identity neighbourhood in $G$. In the special case where $G$ is a $p$-adic Lie group, the structure lattice can be described in terms of the $\QQ_p$-Lie algebra of $G$, as follows. \index{p-adic Lie group@$p$-adic Lie group}

\begin{prop}\label{prop:p-adicLie}
Let $G$ be a $p$-adic Lie group. Then the structure lattice $\lnorm(G)$ is $G$-equivariantly isomorphic to the lattice of ideals in the $\QQ_p$-Lie algebra $\mathrm{Lie}(G)$.
\end{prop}

Indeed, that proposition readily follows from  the following known properties of the Lie correspondence in $p$-adic Lie groups. 

\begin{prop}\label{prop:p-adicLieCorrespondence}
Let $G$ be a $p$-adic Lie group, and $\mathfrak g= \mathrm{Lie}(G)$ denote its $\QQ_p$-Lie algebra.
\begin{enumerate}[(i)]
\item For any two closed subgroups $H_1, H_2 \leq G$, we have $[H_1] = [H_2]$ if and only if $\mathrm{Lie}(H_1) = \mathrm{Lie}(H_2)$.

\item If $H\leq G$ is closed with $\N_G(H)$ open, then $\mathrm{Lie}(H)$ is an ideal. 

\item If $\mathfrak h$ is an ideal of $\mathfrak g$, then there exists a compact locally normal subgroup $H$ with $\mathrm{Lie}(H) = \mathfrak h$. 
\end{enumerate}
\end{prop}

\begin{proof}
We first  recall that every closed subgroup of $G$ is a Lie group, whose Lie algebra  is a subalgebra of $\mathfrak g$, see  \cite[Ch.~III, \S8, Th.~2]{Bbki}. 
 
We next observe that every vector subspace of $\mathfrak g$ is closed: this is easily seen by induction on dimension, using the fact that every proper open subgroup of $\QQ_p$ is compact. Therefore, every Lie subalgebra of $\mathfrak g$ is the Lie algebra of a closed subgroup of $G$ by  \cite[Ch.~III, \S7, Th.~2]{Bbki}. 

This being recorded, assertion (i) now follows from  \cite[Ch.~III, \S7, Th.~2]{Bbki}, while assertions (ii) and (iii) follow from \cite[Ch.~III, \S7, Prop.~2]{Bbki}. 
\end{proof}
}

\section{Quasi-centralisers and their stability properties}
\label{sec:QZmain}

\subsection{Motivation}
\label{sec:QCmotiv}

Let $(G,U)$ be a Hecke pair.  Our principal aim in the present section and Sections \ref{sec:ldlat} and \ref{sec:lcent} is to obtain two subsets of the structure lattice of $(G,U)$, namely the \defbold{local decomposition lattice}\index{local decomposition lattice!of a Hecke pair} $\ldlat(G,U)$\index{LD(G,U)@$\ldlat(G, U)$} and the \defbold{centraliser lattice}\index{centraliser lattice!of a Hecke pair} $\lcent(G,U)$\index{LC(G,U)@$\lcent(G, U)$} that are Boolean algebras, provided that $(G,U)$ satisfies some unavoidable hypotheses.  The local decomposition lattice will have as representatives the direct factors of finite index subgroups of $U$; more generally, one can define the local decomposition lattice $\ldlat(G,U;K)$\index{LD(G,U;K)@$\ldlat(G,U;K)$} relative to a locally normal subgroup $K$, with representatives given by those direct factors of finite index subgroups of $K$ that are locally normal in $G$.  The centraliser lattice will have as representatives the centralisers of all locally normal subgroups.  Here are some observations that will lead us to the required hypotheses.

\begin{enumerate}
\item
Given a direct factor $K$ of a group $U$, observe that there may be many different subgroups $L$ such that $K \times L = U$: consider for instance the direct factorisations of $U = \bZ^2_p$ as $\bZ_p \times \bZ_p$.  However, if the centre of $U$ is trivial, then $K$ has a unique complement, namely $L = \CC_U(K)$.

\item
\emph{A priori}, the map
$$\lnorm(G,U) \rightarrow \lnorm(G,U); \; [K] \mapsto [\CC_G(K)]$$
is not well-defined: given a locally normal subgroup $K$ and a finite index subgroup $V$ of $U$, it could be the case that $\CC_G(K \cap V)$ lies in a local equivalence class that is strictly above the local equivalence class of $\CC_G(K)$.  So instead, we will define a subgroup $\QC_G([K])$ that is the union of all centralisers of representatives of $[K]$, and then define $[K]^\bot$ to be $[\QC_G([K])]$.  We will see that under suitable conditions, every representative of $[K]$ actually has the same centraliser, so the map $[K] \mapsto [\CC_G(K)]$ is in fact well-defined.

\item
There is a natural approach to obtaining a Boolean algebra from a bounded lattice $\mcL$, provided that $\mcL$ is \emph{pseudocomplemented}.  A \defbold{pseudocomplementation}\index{pseudocomplementation}\index{perp@$\bot$} is a function $\bot: \mcL \rightarrow \mcL$ such that:

$(*)$ For all $\alpha,\beta \in \mcL$, we have $\alpha \wedge \beta = 0$ if and only if $\beta \le \alpha^\bot$.

If a pseudocomplementation exists, it is necessarily unique.  Moreover, the set $\mcA = \{\alpha^\bot \mid \alpha \in \mcL\}$ forms a Boolean algebra, with the restriction of $\bot$ serving as the complementation map on $\mcA$ (in particular, $\bot^3 = \bot$).  We will obtain hypotheses such that $\lnorm(G,U)$ is pseudocomplemented, and the pseudocomplementation will be given by $[K] \mapsto [\CC_G(K)]$.
\end{enumerate}

These observations lead naturally to the definitions of the \emph{quasi-centre} (see \ref{sec:QZ}), \emph{quasi-centralisers} (see \ref{sec:QC}) and \emph{C-stability} (see \ref{sec:Cstable}).  The purpose of the present section (Section~\ref{sec:QZmain}) is to perform the technical groundwork needed to verify that the definitions have the properties we desire of them.

Since the lattices we are describing are `local' invariants of the Hecke pair $(G,U)$, for the proofs, there will be no loss of generality in assuming $G=U$, and we may freely replace $G$ or $U$ by a finite index subgroup when it is convenient to do so.

\subsection{The quasi-centre}\label{sec:QZ}

\index{quasi-centre} As observed in \S\ref{sec:QCmotiv}, to ensure unique complementation in the local decomposition lattice of the Hecke pair $(G,U)$, we need the finite index subgroups of $U$ (or at least all sufficiently small finite index subgroups of $U$) to have trivial centre.  When $(G, U)$ is a \tdlc Hecke pair, this becomes the question as to whether the \textbf{quasi-centre} of $G$, defined as the set of all elements whose centraliser is open, is discrete. The notion of quasi-centre is due to Burger--Mozes \cite{BurgerMozes}; its basic properties are as follows. 

\begin{lem}\label{lem:QZ}
Let $G$ be a locally compact group. Then:
\begin{enumerate}[(i)]
\item $\QZ(G)$ is a (not necessarily closed) characteristic subgroup of $G$. 

\item $\QZ(O) \leq \QZ(G)$ for any open subgroup $O \leq G$. 

\item For any closed normal subgroup $N$ of $G$, the image of $\QZ(G)$ under the projection to $G/N$ is contained in $\QZ(G/N)$. 

\item $\QZ(G)$ contains every discrete subgroup of $G$ whose normaliser in $G$ is open. In particular, all discrete normal subgroups and all finite locally normal subgroups are contained in $\QZ(G)$. 

\item For any discrete normal subgroup $N$ of $G$, we have $\QZ(G)/N = \QZ(G/N)$. In particular, if $\QZ(G)$ is discrete, then $\QZ(G/\QZ(G)) =1$. 
\end{enumerate}
\end{lem}

\begin{proof}
(i) For all $g, h \in G$, we have  $\CC_G(gh) \geq  \CC_G(g) \cap \CC_G(h)$. This implies that $gh \in \QZ(G)$ as soon as  $g, h \in \QZ(G)$. Part (i) follows. 

(ii) Let $O \leq G$ be open. Since $\CC_G(g) \geq \CC_O(g)$ for any $g \in G$, so that $\CC_G(g)$ is open as soon as $\CC_O(g)$ is so. Part (ii) follows. 

(iii) Clear since the projection $G \to G/N$ is an open map. 

(iv) Let $H$ be a discrete subgroup of $G$ such that $\N_G(H)$ is open and let $V$ be an open compact subgroup of $\N_G(H)$. Consider $h\in H$. Since the $V$-conjugacy class of $h$ is compact, thus finite, $h$ is centralised by an open subgroup of $V$ and so $h \in QZ(G)$ as claimed. 

(v) Let $N \leq G$ be a discrete normal subgroup. We have $N \leq \QZ(G)$ by (iv) and  $\QZ(G)/N \leq \QZ(G/N)$ by (iii). Let now $g \in G$ such that $gN \in \QZ(G/N)$. Thus there is $U \in \mcB(G)$ such that $[g, U] \subset N$. Since $N$ is discrete, there exists also $V  \in \mcB(G)$  with $V \cap N =1$. Since the commutator map $x \mapsto [g, x]$ is continuous, we may, upon replacing $U$ by a sufficiently small open subgroup, assume that $[g, U] \subset V$. We deduce that $[g, U]=1$, whence $g \in \QZ(G)$. This confirms that  $\QZ(G)/N \geq \QZ(G/N)$, and the required assertion follows.
\end{proof}

Given a general Hecke pair $(G,U)$, we define the \defbold{quasi-centre of $G$ with respect to $U$}\index{quasi-centre!of a Hecke pair}, denoted by $\QZ(G, U)$, to be the set of all elements of $G$ that have finitely many $U$-conjugates.  When $(G,U)$ is a \tdlc pair, this is equal to the quasi-centre of $G$ as defined previously and we will adopt the notation $\QZ(G)$ for the quasi-centre of $G$ with respect to $U$ in the general case as well when there is no ambiguity about the choice of $U$. Subgroups $H$ of $G$ are understood to have the Hecke pair structure $(H,H \cap U)$, so $\QZ(H)$ is the set of all elements of $H$ that have finitely many $(H \cap U)$-conjugates.  Note that $\QZ(H)$ is a normal subgroup of $H$, and agrees with the usual definition of the quasi-centre for topological groups in the case that $(G,U)$ is \tdlc and $H$ is closed.  When $H$ is bounded with respect to $(G,U)$, observe that $\QZ(H)$ is just the union of all finite conjugacy classes of $H$.

Note that a version of Lemma~\ref{lem:QZ}(iv) also holds for Hecke pairs $(G,U)$:

\begin{lem}\label{lem:QZ:discrete}Let $(G,U)$ be a Hecke pair and let $H \le G$.  Suppose that there is a finite index subgroup $V$ of $U$ such that $V \le \N_G(H)$ and $H \cap V = 1$.  Then $H \le \QZ(G)$.\end{lem}

\begin{proof}Let $x \in H$.  Then there is a finite index subgroup $W$ of $V$ such that $xWx\inv \le V$, since $V$ is commensurated by $G$.  Given $w \in W$, we have $[x,w] \in H$, since $H$ is normalised by $W$, and also $[x,w] \in xWx\inv W \subseteq V$.  Thus $[x,w] \in H \cap V$, so $[x,w] = 1$.  Hence $x \in \CC_G(W) \le \QZ(G)$.\end{proof}

The quasi-centre has a useful characterisation in the case that it is locally equivalent to the trivial group.

\begin{cor}\label{cor:QZ:discrete}Let $(G,U)$ be a Hecke pair.  Suppose that $\QZ(G) \in [1]$, that is, there exists a finite index subgroup $V$ of $U$ such that $\QZ(G) \cap V = 1$.  Then $\QZ(G)$ is the unique largest locally normal representative of $[1]$.  In particular, if $\QZ(G)=1$, then the trivial group is the only locally normal representative of $[1]$ in $G$.\end{cor}

\begin{proof}Follows immediately from Lemma~\ref{lem:QZ:discrete}.\end{proof}

Observe that if $V$ is a finite index subgroup of $U$, then $\QZ(V)=\QZ(G) \cap V$.  So in taking the local approach, to prove results about Hecke pairs $(G,U)$ with $\QZ(G) \in [1]$, we are usually free to assume that $\QZ(G)=1$.

\subsection{Quasi-centralisers}\label{sec:QC}

Generalising the quasi-centre, we introduce a notion of the centraliser of a subgroup that depends only on its local equivalence class. 

\begin{defn}\label{def:QC}
Let $(G,U)$ be a Hecke pair, let $H$ and $K$ be subgroups of $G$ and let $\mcC$ be the class of subgroups of $K \cap U$ of finite index.  The \defbold{quasi-centraliser}\index{quasi-centraliser}\index{QCH(K)@$\QC_H(K)$} of $K$ in $H$ is given by
\[ \QC_H(K) : = \bigcup_{L \in \mcC} \CC_H(L).\]
\end{defn}

Clearly the quasi-centraliser $\QC_H(K)$ contains the centraliser $\CC_H(K)$.  Moreover, the quasi-centraliser is a subgroup of $G$, since $\CC_H(L_1)\CC_H(L_2) \subseteq \CC_H(L_1 \cap L_2)$. Similarly we have $\QC_H(K) = \QC_H(L)$ for any finite index subgroup $L$ of $K \cap U$, so that $\QC_H(K)$ depends only on the local equivalence class $\alpha = [K]$ and we can write $\QC_H(\alpha)$  without ambiguity.  Some caution is required in saying that two subgroups `quasi-commute', as the relation \emph{A is contained in the quasi-centraliser of B} is not symmetric.  We see that $\QZ(H) = \QC_H(H)$.

We emphasise that the quasi-centraliser of a subgroup need not be closed, even in the context of profinite groups. For example, if $G = \prod_{\mathbf N} S_3$, then $\QC_G(G)$ is the dense subgroup $\coprod_{\mathbf N} S_3$.  However, it is the case that $\QC_G(H) = \QC_G(\overline{H})$ in this context.

\begin{lem}\label{lem:quasicent:dense}
Let $(G,U)$ be a Hecke pair, such that $G$ is a topological group and $U$ is open in $G$.  Let $H$ be a subgroup of $G$.  Then
$$\QC_G(H) = \QC_G(\overline{H}).$$
\end{lem}

\begin{proof} It may be assumed that ${H} \le U$ and hence, since $U$ is an open subgroup and therefore closed,  that $\overline{H} \le U$. Hence, to compute the quasi-centralisers, we may work with finite index subgroups of $H$ and $\overline{H}$. 

Let $K\leq \overline{H}$ have finite index. Then $K \cap H$ has finite index in $H$ and $\CC_G(K \cap H) \ge \CC_G(K)$.  Hence $\QC_G(H) \ge \QC_G(\overline{H})$.  On the other hand, if $L$ has finite index in $H$, then $\overline{L}$ has finite index in $\overline{H}$ and $\CC_G(L) = \CC_G(\overline{L})$.  Hence $\QC_G(H) \le \QC_G(\overline{H})$
\end{proof}

In the context of \tdlc groups, the definition of the quasi-centraliser can be reformulated as follows.

\begin{lem}\label{lem:QC:tdlc}
Let $(G, U)$ be a \tdlc Hecke pair and $H, K$ be subgroups of $G$. Then we have
\[ \QC_H(K) : = \bigcup_{V  \in \mcB(G)} \CC_H(K \cap V).\]
\end{lem}
\begin{proof}
In view of Lemma~\ref{lem:quasicent:dense}, we have $\QC_H(K) = \QC_H(\overline K)$. Moreover, for any $V \in \mcB(G)$ we have $\overline K \cap V = \overline{K \cap V}$ because $V$ is open, which yields  $\CC_H(K \cap V) = \CC_H(\overline{K \cap V}) = \CC_H(\overline K \cap V)$. We may therefore assume that $K$ is closed. 

For any $V \in \mcB(G)$, we have $\CC_H(K \cap V) \geq \CC_H(K \cap U \cap V)$, so that $\bigcup_{V  \in \mcB(G)} \CC_H(K \cap V) = \bigcup_{V  \in \mcB(G)} \CC_H(K \cap U \cap V)$. Given a finite index subgroup $L \leq K \cap U$, we have $\CC_H(L) = \CC_H(\overline L)$. Moreover, since $K$ is closed and since $\overline L$ is an open subgroup  of $K \cap U$ there exists $V \in \mcB(G)$ with $K \cap U \cap V = \overline L$. Denoting by $\mcC$ be the class of subgroups of $K \cap U$ of finite index, we infer that 
$$\QC_H(K) = \bigcup_{L \in \mcC} \CC_H(L) = \bigcup_{V  \in \mcB(G)} \CC_H(K \cap U \cap V)$$
as required.
\end{proof}

The following observation is borrowed from \cite[Prop.~2.6]{BEW}.

\begin{lem}\label{bewlem} 
Let $G$ be a group, let $U$ and $V$ be subgroups of $G$, and let $W \le U \cap V$.  Suppose $\theta: U \rightarrow V$ is an (abstract) isomorphism such that $\theta(x) = x$ for all $x \in W$, and suppose also that $\CC_G(W \cap g^{-1}Wg)=1$ for all $g \in U$.  Then $U = V$ and $\theta = \mathrm{id}_U$.
\end{lem}

\begin{proof}
Let $g \in U$ and let $x \in W \cap g^{-1}Wg$.  Then
\[ gxg^{-1} = \theta(gxg^{-1}) = \theta(g)\theta(x)\theta(g^{-1}) = \theta(g)x\theta(g^{-1}),\]
so $g^{-1}\theta(g)$ centralises $x$. This shows that  $g^{-1}\theta(g) \in \CC_G(W \cap g^{-1}Wg) = 1$.  Thus $g = \theta(g)$.
\end{proof}

In the following lemmas, we obtain some consequences of the hypothesis that a given subgroup $H$ of $G$ has trivial quasi-centre.

\begin{lem}\label{bewcor}
Let $(G,U)$ be a Hecke pair and $H$ be a subgroup of $G$ such that $\QZ(H)=1$.  Then:
\begin{enumerate}[(i)]
\item $\QC_G(H) \cap \N_G(H) = \CC_G(H)$;
\item Every finite locally normal subgroup of $G$ has trivial intersection with $H$.
\end{enumerate}
\end{lem}

\begin{proof} For part~(i), let $t \in \QC_G(H) \cap \N_G(H)$.  Then conjugation by $t$ induces an automorphism, $\theta_t$, of $H$ that fixes pointwise some finite index subgroup, $K$, of $H \cap U$. Since $K$ is commensurated by $H$ and $\QZ(H, H\cap U)=1$, we have $\CC_H(K \cap g\inv Kg) = 1$ for all $g \in H$ and so, by Lemma~\ref{bewlem},  $\theta_t = \mathrm{id}_H$, that is, $t \in \CC_G(H)$. The reverse inclusion is clear.

For part~(ii), let $K$ be a finite locally normal subgroup of $G$.  Then there is a subgroup $V$ of $G$ that is commensurate to $U$ such that $K \le V \le \N_G(K)$.  Let $H' = H \cap V$ and let $K' = H \cap K$.  Then $K'$ is finite and normal in $H'$, and therefore contained in the quasi-centre of $H'$. However, $H'$ has trivial quasi-centre, and so $K'$ is trivial.
\end{proof}

\begin{lem}\label{finint}
Let $(G,U)$ be a Hecke pair and let $H$ and $K$ be subgroups of $G$ such that $\QZ(H)=1$. Suppose that $|H \cap K|$, $|H:\N_H(K)|$ and $|K:\N_K(H)|$ are all finite.  Then $\N_K(H) = \CC_K(H)$.\end{lem}

\begin{proof} 
Set $A = \N_H(K)$ and $B = \N_K(H)$, so that $A$ and $B$ normalise each other. We have $A \cap B \leq H \cap K$, so that $A \cap B$ is finite by hypothesis. Since $A \cap B$ is normal in $A$, we infer that $A \cap B$ is centralised by a finite index subgroup of $A$, hence by a finite index subgroup of $H$ since $|H : A|$ is finite by hypothesis. Therefore $A \cap B \leq \QZ(H) = 1$. Since   $A$ and $B$ normalise each other, it follows that they commute. Hence $B$ centralises a finite index subgroup of $H$, which implies that $B \leq \QC_G(H)$. Since $B$ also normalises $H$, we deduce that $B \leq \CC_G(H)$  by Lemma~\ref{bewcor}(i). 
\end{proof}

\begin{lem}\label{lnorm:trivialqz}
Let $(G,U)$ be a Hecke pair and $H$ be a subgroup of $G$ such that $\N_U(H)$ has finite index in $U$.

\begin{enumerate}[(i)]
\item Suppose $(G,U)$ is \tdlc (as defined in \S\ref{sec:BasicDefs} above), $H$ is closed and $\QZ(H)$ is discrete and torsion-free.  Then ${\QC_G(H)}$ is closed and $H \cap {\QC_G(H)}$ is discrete and torsion-free. In particular, $\QC_G(H)$ cannot be dense in $G$ unless $H$ is discrete and torsion-free. 

\item Let $V$ be a subgroup of $G$ that is commensurate with $U$.  If $\QZ(H) \cap V=1$, then  $\QC_G(H)$ is locally equivalent to $\CC_G(H \cap V)$.  In particular, we have 
$$
\QC_G(\QC_G(H)) = \QC_G(\CC_G(H \cap V)) =  \QC_G(\CC_V(H \cap V)).
$$
\end{enumerate}
\end{lem}

\begin{proof}(i) Let $V$ be any open compact subgroup of $G$.  Then $\QZ(H \cap V)=1$ because $\QZ(H)$ is discrete and torsion-free and $\QZ(H \cap V)\leq \QZ(H)$. Therefore, by Lemma~\ref{bewcor}(i)  and because $\QC_G(H \cap V) = \QC_G(H)$, 
\begin{equation}
\label{eq:lnorm:trivialqz}
\QC_G(H) \cap  \N_G(H \cap V) = \CC_G(H \cap V).
\end{equation}
In particular, ${\QC_G(H)}  \cap (H \cap V) \le \Z(H \cap V) = 1$.  Hence $H \cap {\QC_G(H)}$ intersects trivially any open compact subgroup of $G$ and is therefore a discrete and torsion-free. That ${\QC_G(H)}$ is closed follows from~\eqref{eq:lnorm:trivialqz} by noting that $\N_G(H \cap V)$ is open because it contains $\N_U(H)\cap V$ and that $\CC_G(H \cap V)$ is closed.

(ii) Let $W = \N_V(H)$.  Since $\QZ(H) \cap V = 1$, it follows that $\QZ(H \cap V) = 1$, and Lemma~\ref{bewcor}(i) implies that 
$$
\QC_G(H \cap V) \cap W \leq \QC_G(H \cap V) \cap \N_G(H \cap V) = \CC_G(H \cap V).
$$ 
On the other hand, the inclusion $\CC_G(H \cap V)\cap W \leq \QC_G(H \cap V) \cap W$ holds because $ \CC_G(H \cap V) \leq \QC_G(H\cap V)$. Since $\QC_G(H\cap V) = \QC_G(H)$, we infer that $\QC_G(H) \cap W  = \CC_G(H \cap V) \cap W$ and hence that $[\QC_G(H)] =[\CC_G(H \cap V)]$. 
\end{proof}

\subsection{C-stability}\label{sec:Cstable}

Recall that one of our aims is to use the map $K \mapsto \QC_G(K)$ to define a pseudocomplementation on $\lnorm(G,U)$, or at least on a subset of $\lnorm(G,U)$.  Observe that any pseudocomplementation $\bot$ on a lattice $\mcL$ satisfies the following identity:
\begin{equation}\label{identity:Cstable:bot}
 \alpha^\bot \wedge \alpha^{\bot^2} = 0 \quad \forall \alpha \in \mcL.
 \end{equation}

The following property applied to locally normal subgroups, closely related to the identity (\ref{identity:Cstable:bot}), will prove to be the key criterion for our purposes.

\begin{defn}
\label{defn:Cstable}
Let $(G,U)$ be a Hecke pair and let $H$ be a subgroup of $G$.  Say $H$ is \defbold{C-stable}\index{C-stable} in $G$ if
$$R(H):= \QC_G(H) \cap \QC_G(\CC_G(H)) \in [1].$$\index{R(H)@$R(H)$}
\end{defn}
 
We shall mostly be interested in the case that $H$ is contained in a subgroup commensurate with $U$. The following shows that, in that case, C-stability could have equivalently been defined to mean that $\QC_V(H\CC_G(H)) = 1$ for some finite index subgroup $V$ of $U$.  In the \tdlc case, this is equivalent to requiring that $\QC_G(H\CC_G(H))$ be discrete.
 
 \begin{lem}\label{cstabequiv}
 Let $(G,U)$ be a Hecke pair and let $H$ be a bounded subgroup of $G$.  Then
 \[R(H)= \QC_G(H) \cap \QC_G(\CC_G(H)) = \QC_G(H\CC_G(H)). \]
 \end{lem}
 
 \begin{proof}
Let $x \in \QC_G(H\CC_G(H))$, so $x$ centralises some finite index subgroup $K$ of $H\CC_G(H) \cap U$.  Then $H \cap K$ has finite index in $H \cap U$ and $\CC_G(H) \cap K$ has finite index in $\CC_G(H) \cap U$ and $x$ centralises both subgroups. Hence $x \in \QC_G(H)$ and $x \in \QC_G(\CC_G(H))$.
 
Conversely, consider $x \in  \QC_G(H) \cap \QC_G(\CC_G(H))$ and let $V$ be a subgroup commensurate to $U$ such that $H \le V$.  Then there are subgroups $L$ of finite index in $H \cap U$ and $M$ of finite index in $\CC_{U \cap V}(H)$ such that 
$$x \in \CC_G(L) \cap \CC_G(M).$$
Hence $x$ centralises the group $LM$. Since $L$ has finite index in $H$ and $M$ has finite index in $\CC_V(H)$, we thus have that $x$ centralises a finite index subgroup of $H\CC_V(H)$.  Finally, the condition that $H \le V$ implies that  $H\CC_V(H) = H\CC_G(H) \cap V$, and the latter group is commensurate with $H\CC_G(H) \cap U$.  Hence $x$ centralises a finite index subgroup of $H\CC_G(H) \cap U$ as required.
 \end{proof}
 
 We also note the role of the quasi-centre of $(G,U)$ as a whole in C-stability.  Given any subgroup $H$ of $G$, then
 $$\QC_G(H) \cap \QC_G(\CC_G(H)) \ge \QZ(G),$$
and equality holds in the case $H = G$.  So if $\QZ(G) \in [1]$, then $G$ is C-stable in itself, and moreover, every subgroup of $G$ commensurate to $U$ is C-stable in $G$.  If instead $\QZ(G) \not\in [1]$, then $G$ does not have any C-stable subgroups.  So in proving properties of C-stable subgroups of a Hecke pair $(G,U)$, we may assume $\QZ(G) \in [1]$.

In particular, the following is an immediate consequence of Corollary~\ref{cor:QZ:discrete}:

\begin{cor}\label{cor:Cstable:QZ}
Let $(G,U)$ be a Hecke pair and let $H$ be a C-stable locally normal subgroup of $G$.  Then
$$R(H) = \QC_G(H) \cap \QC_G(\CC_G(H)) = \QZ(G).$$
\end{cor}

The next lemma records the most basic properties of C-stable subgroups. 

\begin{lem}\label{lem:CstableBasic}
Let $(G,U)$ be a Hecke pair and let $H$ be a C-stable subgroup of $G$.
\begin{enumerate}[(i)]
\item Let $K$ be a subgroup of $H$ that is commensurate with $H \cap U$.  Then $K$ is C-stable in $G$. 

\item $\QZ(H) \in [1]$; indeed, $\QZ(H) \leq \QC_G(H) \cap \QC_G(\CC_G(H))$.

\item If $H \cap U$ is finite, then $H \le \QC_G(H) \cap \QC_G(\CC_G(H))$. 
\end{enumerate}
\end{lem}

\begin{proof}Let $V$ be a subgroup of $G$ that is commensurate with $U$.

(i) Since $K$ is commensurate with $H \cap U$, we have $\QC_G(K) = \QC_G(H)$. Since $K\leq H$, we have $\CC_G(K) \geq \CC_G(H)$, so that $\QC_G(\CC_G(K)) \leq \QC_G(\CC_G(H))$. Therefore
\[ \QC_G(K) \cap \QC_G(\CC_G(K)) \leq \QC_G(H) \cap \QC_G(\CC_G(H)) \in [1].\]

(ii) We have $\QZ(H) = \QC_H(H) \leq \QC_G(H)$. Moreover, any $g \in \QZ(H)$ belongs to $H$ and therefore commutes with $\CC_G(H)$, so that $\QZ(H) \leq \CC_G(\CC_G(H)) \leq \QC_G(\CC_G(H))$. Therefore 
$$\QZ(H) \leq \QC_G(H) \cap \QC_G(\CC_G(H)) \in [1].$$

(iii) If $H \cap U$ is finite, then $\QZ(H)=H$, so $H \le \QC_G(H)$.  Clearly also $H$ is contained in $\QC_G(\CC_G(H))$. 
\end{proof}

For each integer $n >0$, we define inductively 
$$
\CC^{n+1}_H(K) = \CC_H(\CC^n_H(K)) 
\hspace{.5cm}\text{and} \hspace{.5cm}
\QC^{n+1}_H(\alpha) = \QC_H(\QC^n_H(\alpha)),
$$
where $H$ and $K$ are subgroups of $G$ and $\alpha$ is either a subgroup, or a local equivalence class of subgroups.  For convenience we also adopt the convention $\CC^0_H(K) = K = \QC^0_H(K)$.

 The following is elementary, and will be used frequently without comment.

\begin{lem}\label{elcent}
Let $G$ be a group and let $H$ be a subgroup of $G$.  Then
$$
\CC^2_G(H) \ge H
\hspace{.5cm}\text{and} \hspace{.5cm}
\CC^3_G(H) = \CC_G(H).  \qed
$$
\end{lem}

The following lemma gives a useful relationship between $n$-th centralisers and $n$-th quasi-centralisers in the context of C-stable subgroups.

\begin{lem}\label{goodqc}
Let $(G,U)$ be a Hecke pair, let $H$ be a subgroup of $G$. As before, we set $R(H) = \QC_G(H) \cap \QC_G(\CC_G(H))$.  Let $n \ge 0$ and let $V$ be a subgroup of $G$ commensurate with $U$.

\begin{enumerate}[(i)]

\item If $L$ is a subgroup of $G$ that contains a finite index subgroup of $U$, then $R(\CC^{n+1}_L(H)) \le R(\CC^n_L(H))$.  In particular, if $H$ is C-stable, then so is $\CC^n_L(H)$.

\item Suppose that $H$ is C-stable and locally normal in $G$.  Let $V$ be a subgroup of $G$ commensurate with $U$ such that $\QZ(G) \cap V = 1$.  Then we have 
$$\QC_V(\QC^n_G(H)) = \CC^{n+1}_V(H \cap V).$$
In particular, $\QC_V(H) = \CC_V(H \cap V)$.
\end{enumerate}
\end{lem}

\begin{proof}
(i) Fix $n \ge 0$, let $H' = \CC^n_L(H)$ and let $K = \CC^{n+1}_L(H) = \CC_L(H')$.  On the one hand  $H' \cap L$ is locally equivalent to $H'$ and contained in $\CC_G(K)$, so ${\QC_G(\CC_G(K)) \le \QC_G(H')}$; on the other hand we also have $\QC_G(K) = \QC_G(\CC_G(H'))$.  Hence
\[ \QC_G(K) \cap \QC_G(\CC_G(K)) \le \QC_G(H') \cap \QC_G(\CC_G(H')),\]
in other words, $R(\CC^{n+1}_L(H)) \le R(\CC^n_L(H))$.  Thus if $H$ is C-stable, then so are the groups $\CC^n_L(H)$ for $n \ge 0$.

(ii) Without loss of generality $H \le V$, since neither side of the desired equation is affected by replacing $H$ by $H \cap V$, and $H \cap V$ retains the properties of being C-stable and locally normal in $G$. Thus $\QZ(H)=1$ by Lemma~\ref{lem:CstableBasic}(ii) and Corollary~\ref{cor:Cstable:QZ}. 

First suppose that $n=0$: we must show that $\QC_V(H)$ commutes with $H$. Consider $x \in \QC_V(H)$, so that $x \in \CC_V(K)$ for some finite index subgroup $K$ of $H \cap U$.  Then $K$ has finite index in $H$ and, replacing $K$ by its core in $H$, we may assume that $K$ is normal in $H$.  

Moreover, $K$ is C-stable by Lemma~\ref{lem:CstableBasic}(i), and it follows from part (i) and Lemma~\ref{lem:CstableBasic}(ii) that $\QZ(\CC_V(K))=1$.

Now $\CC_V(K) \cap H= 1$ since $\QZ(H)=1$.  Also, $H$ normalises $\CC_V(K)$ since $H \leq V$ and since $H$ normalises $K$. Furthermore $H$ is locally normal, so that $H$ is normalised by some finite index subgroup $W$ of $U \cap V$ such that $W$ is normal in $V$. Hence $H$ is normalised by $\CC_W(K)$. Thus 
$$[H, \CC_W(K)] \leq H \cap \CC_W(K) \leq H \cap \CC_V(K) =1,$$ 
so that $H$ centralises $\CC_W(K)$.  As $\CC_W(K)$ is a subgroup of $\CC_V(K)$ of finite index, we see that $H$ centralises $\CC_V(K)$ by Lemma~\ref{bewcor}(i), and in particular $H$ centralises $x$. This concludes the proof for $n=0$.

For the general case, note first that $\CC^n_V(H)$ is C-stable by part (i), has trivial quasi-centre by Lemma~\ref{lem:CstableBasic}(ii) and is  also locally normal by Lemma~\ref{lem:LN:basic}(iii). Therefore, by repeated applications of Lemma~\ref{lnorm:trivialqz}(ii), recalling that $H\le V$, we have the equation
\[ \QC^{n+1}_V(H) = \QC_V(\QC^n_V(H)) =  \QC_V(\CC^n_V(H)).\]
The desired conclusion follows, since by the base case $n=0$,  we have
 \[\QC_V(\CC^n_V(H)) = \CC_V(\CC^n_V(H)). \qedhere \]
\end{proof}

We are now able to define a partial map $\alpha \mapsto [\QC_G(\alpha)]$ on the abstract structure lattice.

\begin{cor}\label{goodqc:lnorm}
Let $(G,U)$ be a Hecke pair and let $K$ be a C-stable locally normal subgroup of $G$.  Suppose $\QZ(G) \cap U =1$.
Then for every finite index subgroup $L$ of $K \cap U$, 
we have
$$
\QC_U(K) = \CC_U(L).
$$ 
In particular we have $[\CC_G(K \cap U)] = [\CC_U(K \cap U)]= [\QC_G(K)]$, so that the local equivalence class $[\CC_U(K)]$ depends only on $[K]$. 
\end{cor}

\begin{proof}
By Lemma~\ref{goodqc}(ii), we have $\QC_U(K) = \CC_U(K \cap U)$.

Let now $L$ be a finite index subgroup of $K \cap U$. By definition of the quasi-centraliser, we have $\QC_G(K) = \QC_G(K \cap U) =   \QC_G(L) \geq \CC_G(L)$. Since $L \leq K \cap U$, we also have $\CC_G(L) \geq \CC_G(K \cap U)$. Intersecting with $U$ and combining this with the observation above, we obtain
$$\QC_U(K) = \QC_U(L) \geq \CC_U(L)  \geq \CC_U(K \cap U) = \QC_U(K).$$
Therefore all those subgroups coincide, which yields the desired assertion. 
\end{proof}

\subsection{Local C-stability}\label{qstabsec}

The previous section motivates the following definition.  

\begin{defn}\label{botdef}
A Hecke pair $(G,U)$ is called \defbold{locally C-stable}\index{locally C-stable}\index{C-stable!locally} if all locally normal subgroups of $G$ are C-stable in $G$.
\end{defn}

We can now establish the following theorem, which characterises locally C-stable Hecke pairs under the assumption that $U$ is residually finite; in particular, we obtain a characterisation of those \tdlc Hecke pairs that are locally C-stable.  We also obtain several stability properties of quasi-centralisers; to avoid unnecessary complications, we state these under the assumption that $G$ has trivial quasi-centre. In view of Lemma~\ref{lem:QZ}(v) this hypothesis is harmless: indeed if $G$ is locally $C$-stable, then $\QZ(G)$ is discrete, so that  $G/\QZ(G)$ is locally C-stable and has trivial quasi-centre.

\begin{thm}\label{thm:locallyCstable}
Let $(G,U)$ be a Hecke pair such that $U$ is residually finite.

Then $G$ is locally C-stable if and only if $\QZ(G) \in [1]$ and every bounded abelian  locally normal subgroup of $G$ is contained in $\QZ(G)$ (see \S\ref{sec:BasicDefs} for the definition of the term bounded).

Moreover, if $G$ is locally C-stable and $\QZ(G)=1$, then every locally normal subgroup $H$ of $G$ has the following properties:
\begin{enumerate}[(i)]
\item $\QZ(H) = 1.$

\item $\QC_G(H) = \CC_G(H).$ In particular $\CC_G(H)$ depends only on $[H]$.
\label{thm:locallyCstableii}

\item Given a locally normal subgroup $K$ of $G$, then $H$ and $K$ commute if and only if $H \cap K = 1$.
\end{enumerate}

\end{thm}

The proof requires some properties of subnormal subgroups of residually finite groups, given in the following lemma.

\begin{lem}\label{lem:finitesubnormal}Let $G$ be a group and let $K$ be a normal subgroup of $G$.
\begin{enumerate}[(i)]
\item If $M$ is an abelian minimal normal subgroup of $K$, then the normal closure of $M$ in $G$ is abelian.
\item If $M$ is a finite perfect normal subgroup of $K$ and $G$ is residually finite, then $|G:\CC_G(M)|$ is finite.
\end{enumerate}
\end{lem}

\begin{proof}(i) Let $g \in G$.  Since $M$ is a minimal normal subgroup of $K$, then either $M = gMg\inv$, in which case $[M,gMg\inv] = 1$ since $M$ is abelian, or $M \cap gMg\inv = 1$, in which case $[M,gMg\inv] \le M \cap gMg\inv = 1$ since $M$ and $gMg\inv$ are normal subgroups of $K$.  In either case, we see that $M$ commutes with $gMg\inv$, so every $G$-conjugate of $M$ commutes with every $G$-conjugate of $M$.  Hence the normal closure of $M$ in $G$ is abelian.

(ii) Since $M$ is finite and $G$ is residually finite, there is a normal subgroup $H$ of $G$ of finite index such that the quotient map $G \rightarrow G/H$ restricts to an injective map $M \rightarrow G/H$.  In particular, $M \cap H = 1$.  Now consider $[H,K]$: this is a subgroup of $H \cap K$ that is normal in $G$.  Thus $[H,K]$ and $M$ are subgroups of $G$ that normalise each other and have trivial intersection, so $[H,K,M] = 1$, where $[H, K, M] = [[H,K], M]$\index{[H, K, M]@$[H, K, M]$}.  In particular, $[H,M,M] = 1$, since $M$ is a subgroup of $K$; equivalently $[M,H,M] = 1$.  By the Three Subgroups Lemma\index{Three Subgroups Lemma} (see \cite[Theorem~10.3.5]{HallBook}), we have $[M,M,H] = 1$.  But $M$ is perfect, so $M = [M,M]$, and hence $[M,H] = [M,M,H] = 1$, demonstrating that $H \le \CC_G(M)$.\end{proof}

\begin{proof}[Proof of Theorem~\ref{thm:locallyCstable}]
Suppose that $G$ has an abelian locally normal subgroup $L$ such that $L \not\le \QZ(G)$.  Then $L \not\in [1]$ by Lemma~\ref{lem:QZ:discrete}, and moreover
$$L \le \QC_G(L) \cap \QC_G(\CC_G(L)),$$
so $L$ is not C-stable in $G$.

From now on, we may suppose that $\QZ(G) \in [1]$ and every non-trivial abelian bounded locally normal subgroup $A$ of $G$ is contained in $\QZ(G)$; in particular, every such $A$ is finite.  If $(G,U)$ is \tdlc, it is enough to consider compact locally normal subgroups, since the closure of an abelian locally normal subgroup is an abelian locally normal subgroup.

We now claim that for every bounded locally normal subgroup $L$ of $G$, then $\QZ(L) \le \QZ(G)$. If $L$ is finite, then $L \le [1]$ by the fact that $U$ is residually finite, so $L \le \QZ(G)$ by Lemma~\ref{lem:QZ:discrete}.  Given a bounded locally normal subgroup $L$ of $G$, then the quasi-centre of $L$ is characteristic in $L$, so the quasi-centre is itself bounded locally normal.  Let $V$ be a finite index subgroup of $U$ such that $\QZ(G) \cap V = 1$.  To show $\QZ(L) \le \QZ(G)$, it is enough to show that $\QZ(L)$ is finite, and hence it is enough to show that $\QZ(L \cap V)$ is finite.  Thus we may assume that $L$ is an infinite subgroup of $V$ such that $L = \QZ(L)$.  By replacing $V$ with $\N_V(L)$, we may assume $L \unlhd V$.

Since $G$ has no infinite abelian bounded locally normal subgroup, the centre of $L$ is finite.  By (\cite{Baer}, Proposition 3 and Theorem 2), it follows that $L$ is the union of its finite normal subgroups, so there exists a minimal finite normal subgroup $M$ of $L$.  Then $M$ is characteristically simple, so either $M$ is abelian or $M$ is perfect.  If $M$ is abelian, then the normal closure $N$ of $M$ in $V$ is a bounded locally normal subgroup of $G$, and $N$ is abelian by Lemma~\ref{lem:finitesubnormal}(i), so $N = 1$; if $M$ is perfect, then $|V:\CC_V(M)|$ is finite by Lemma~\ref{lem:finitesubnormal}(ii), so $M \le \QZ(G) = 1$.  In either case, we have a contradiction to the assumption that $M$ is non-trivial.  This contradiction proves the claim that every bounded locally normal subgroup $L$ of $G$ has quasi-centre contained in $\QZ(G)$.

We next claim that for any bounded locally normal subgroup $L$ of $G$, we have
$${\QC_G(L \CC_G(L)) \le \QZ(G)}.$$
This implies that $L$ is C-stable in view of Lemma~\ref{cstabequiv}.

Let $V$ be a finite index subgroup of $U$ such that $\QZ(G) \cap V = 1$ and $V \le \N_G(L)$, and let $W$ be a subgroup of $G$ commensurate to $U$ such that $L \unlhd W$.  Then $L \cap V$ has finite index in $L$, so $L = \bigcup^n_{i=1}(L \cap V)l_i$ for some finite set of elements $\{l_1,\dots,l_n\}$ of $L$.  Hence
$$ (L\CC_G(L)) \cap W = L\CC_W(L) = \bigcup^n_{i=1}(L \cap V)\CC_W(L)l_i =  \bigcup^n_{i=1}\bigcup^m_{j=1}(L \cap V)\CC_V(L)l_iw_j,$$
where $\{w_1,\dots,w_m\}$ is a finite subset of $W$.  Consequently $\QC_G(L\CC_G(L)) = \QC_G(K)$, where $K = (L \cap V) \CC_V(L)$.

Notice that $\CC_V(K)  \le M \cap \CC_V(M)$, where $M = \CC_V(L)$.  In particular, $\CC_V(K)$ is abelian, and moreover it is bounded and locally normal in $G$ since $\CC_V(K) \unlhd V$. It must thus be contained in $\QZ(G)$ by hypothesis.  Hence $\CC_V(K) \le \QZ(G) \cap V = 1$.  On the other hand $\QZ(K)=1$ by the first part of the proof above. Therefore, by Lemma~\ref{bewcor}(i) we have $\QC_V(K) = \CC_V(K) =1$.  In particular, the group generated by $\QC_G(K)$ and $V$ is a semidirect product $\QC_G(K) \rtimes V$.  Given $g \in \QC_G(K)$, there is a finite index subgroup $Y$ of $V$ such that $YgYg\inv$ is contained in $Y$.  We see that $ygy\inv g\inv \in \QC_G(K) \cap V = 1$ for all $y \in Y$, so $g \in \CC_G(Y) \le \QZ(G)$.  Thus $\QC_G(K) \le \QZ(G)$, as desired, and the claim is proved.

It remains to prove that every locally normal subgroup $H$ satisfies the assertions (i), (ii) and (iii), under the assumption that $\QZ(G)=1$.  

Let $V = \N_U(H)$.  Then $H \cap V$ is bounded locally normal, so
$$\QZ(H) \cap V = \QZ(V) \le \QZ(G).$$
Hence $\QZ(H)$ is normalised by $V$ and has trivial intersection with $V$, so $\QZ(H) \le \QZ(G)$ by Lemma~\ref{lem:QZ:discrete}, which proves (i).

Let $\tilde H = \QC_G(H)$.  Since $\QZ(H)=1$, Lemma~\ref{bewcor}(i) ensures that
$$ \QC_G(H) \cap \N_G(H) = \CC_G(H),$$
in other words $\N_{\tilde H}(H) = \CC_G(H)$.  Now $\N_{\tilde H}(H)$ contains a finite index subgroup $L$ say of $\tilde H \cap U$, so
$$H \le \CC_G(L) \le \QC_G(\tilde H).$$
Moreover the definition $\tilde H = \QC_G(H)$ implies that  $H$ normalises $\tilde H$, and also that $\tilde H$ is locally normal.  Then $\QZ(\tilde H) = 1$ by (i), and we infer  from Lemma~\ref{bewcor}(i) that $H$ centralises $\tilde H$.  Hence $\tilde H = \CC_G(H)$ and (ii) holds.

Let $K$ be a locally normal subgroup of $G$.  If $K \le \CC_G(H)$, then certainly $H \cap K =1$ by property (i).  Conversely, suppose $H \cap K = 1$.  Then there is a finite index subgroup $W$ of $U$ such that $H \cap W$ normalises $K \cap W$, and hence $H \cap W \le \CC_G(K \cap W)$ by Lemma~\ref{finint}.  Hence $H \cap W \le \QC_G(K)$, but then by property (ii) we have $H \cap W \le \CC_G(K)$.  Hence by the same argument $K \le \QC_G(H) = \CC_G(H)$, proving (iii).
\end{proof}

\section{Direct decomposition of locally normal subgroups}
\label{sec:ldlat}

\subsection{Local decomposition lattices}

Let $(G,U)$ be a Hecke pair and let $H$ be a bounded C-stable locally normal subgroup of $G$.  In this section, we will obtain Boolean algebras corresponding to the decomposition of finite index subgroups of $H$ into direct factors.

\begin{lem}\label{dirqf}Let $(G,U)$ be a Hecke pair, let $H$ be a bounded locally normal subgroup of $G$ with trivial quasi-centre, let $K$ be a subgroup of $H$ of finite index and let $L$ be a direct factor of $K$.  Then $L$ is C-stable in $(H,H)$.  Moreover $\CC_K(L)$ is the unique direct complement of $L$ in $K$ and $L = \CC^2_K(L)$.

If $H$ is C-stable in $G$, then $L$ is also C-stable in $G$.\end{lem}

\begin{proof}
Let $K = L \times M$.  Then $L\CC_H(L)$ is a finite index subgroup of $H$, so 
$${\QC_H(L\CC_H(L)) = \QZ(H) = 1}.$$
Thus $L$ is C-stable as a subgroup of $H$ by Lemma~\ref{cstabequiv}.  In particular, $L$ has trivial centre; we see that $\CC_K(L) = M$, and similarly $\CC_K(M)=L$.  In addition $L\CC_G(L)$ contains $K$ and also contains $\CC_G(K)$, so $K\CC_G(K) \le L\CC_G(L)$.

If $H$ is C-stable in $G$, then $K$ is C-stable in $G$ by Lemma~\ref{lem:CstableBasic}(i), and thus by Lemma~\ref{cstabequiv},
$$\QC_V(L\CC_G(L)) \le \QC_V(K\CC_G(K)) = 1$$
for every subgroup $V$ of $G$ commensurate with $U$.  In this case $L$ is C-stable in $G$.
\end{proof}

\begin{lem}\label{dirfac}Let $G$ be a group such that $\Z(G)=1$, let $K$ be a direct factor of $G$, and let $H$ be a subgroup of $G$ such that $H =\CC^2_G(H)$.  Then $H = (H \cap K) \times \CC_H(K)$.\end{lem}

\begin{proof}We have $G = K \times \CC_G(K)$; let $\pi_1$ and $\pi_2$ be the projections onto $K$ and $\CC_G(K)$ respectively associated with this decomposition.  Then $H \le \pi_1(H) \times \pi_2(H)$.  Moreover $\pi_1(H)$ and $\pi_2(H)$ are centralised by $\CC_G(H)$, so contained in $H$.  Thus $H = \pi_1(H) \times \pi_2(H)$, which implies $\pi_1(H) = H \cap K$ and $\pi_2(H) = \CC_H(K)$.\end{proof}

\begin{defn}Given a group $H$ and a subgroup $K$ of $H$, say $K$ is an \defbold{almost direct factor} of $H$ if there is a finite index subgroup $L$ of $H$ such that $K$ is a direct factor of $L$.

Let $H$ be a bounded locally normal subgroup of $G$, such that the quasi-centre of $H$ is  finite.  Define the \defbold{local decomposition lattice} $\ldlat(G,U;H)$\index{LD(G,U;K)@$\ldlat(G,U;K)$}\index{local decomposition lattice!at a subgroup} of $G$ at $H$ to be the subset of $\lnorm(G,U)$ consisting of elements $[K]$ where $K$ is locally normal in $G$ and $K$ is an almost direct factor of $H$.\end{defn}

In the case that $G$ is a \tdlc group and $U$ is an open compact subgroup of $G$, we will write $\ldlat(G;H)$ to mean $\ldlat(G,U;H)$.

\begin{lem}Let $(G,U)$ be a Hecke pair, let $H$ be a bounded locally normal subgroup of $G$, and let $K$ be a locally normal subgroup of $G$ that is commensurate with $H$.  Then $\ldlat(G,U;H) = \ldlat(G,U;K)$ as subsets of $\lnorm(G,U)$.\end{lem}

\begin{proof}Without loss of generality, we may assume that $K$ is a finite index subgroup of $H$.

It is clear from the definition that $\ldlat(G,U;K) \subseteq \ldlat(G,U;H)$.  Conversely, suppose the finite index subgroup $L$ of $H$ decomposes as $L_1 \times L_2$ with $L_1$ and $L_2$ locally normal in $G$.  Then $L_i \cap K$ is locally normal, and has finite index in $L_i$ for $i \in \{1,2\}$.  Hence $(L_1 \cap K) \times (L_2 \cap K)$ has finite index in $K$, so $[L_1] = [L_1 \cap K]$ and $[L_2] = [L_2 \cap K]$ are contained in $\ldlat(G,U;K)$.\end{proof}

Consequently, it makes sense to write $\ldlat(G,U;\alpha)$  where $\alpha$ is any element of $\lnorm(G,U)$ that has a representative with trivial quasi-centre, and this lattice admits an action of $G_\alpha$ by conjugation.  The case $\alpha = [U]$ is of particular interest, and we will define $\ldlat(G,U):=\ldlat(G,U;[U])$.

As the name suggests, the local decomposition lattice is a sublattice of $\lnorm(G,U)$, but more is true:

\begin{thm}\label{dfaclat}Let $(G,U)$ be a Hecke pair and let $\alpha \in \lnorm(G,U)$ have a bounded locally normal representative $H$ with trivial quasi-centre.  Then $\ldlat(G,U;\alpha)$ is a sublattice of $\lnorm(G,U)$, and internally it is a Boolean algebra (relative to the maximum $\alpha$), with complementation map $\bot_\alpha: \beta \mapsto [\QC_H(\beta)]$.  If $(G,U)$ is \tdlc and $\alpha$ has a closed representative, then $\ldlat(G;\alpha)$ is a sublattice of $\lnorm(G)$.\end{thm}

We will prove the following separately, as it will be reused later.

\begin{lem}\label{boolflip}Let $\mcM$ be a meet-semilattice with a least element $0$.  Suppose $\bot \colon \mcM \mapsto \mcM$ is an order-reversing involution such that for all $\alpha,\beta \in \mcM$, $\alpha \wedge \beta = 0$ if and only if $\beta \le \alpha^\bot$.  Then $\mcM$ is a Boolean algebra, with minimum $0$ and maximum $0^\bot$, complementation given by $\bot$ and join operation given by $\alpha \vee \beta = (\alpha^\bot \wedge \beta^\bot)^\bot$.\end{lem}

\begin{proof}Let $\gamma \in \mcM$.  The existence of an order-reversing involution ensures that $\mcM$ is a lattice, with maximum and join as given.  For all $\alpha \in \mcM$ we have $\alpha \vee \alpha^\bot = (\alpha^\bot \wedge \alpha)^\bot = 0^\bot$ and clearly $\alpha \wedge \alpha^\bot = 0$, so $\alpha^\bot$ is a complement of $\alpha$ relative to the maximum element $0^\bot$.

It remains to show that $\mcM$ is distributive.  We first show
\begin{equation}\label{eq:distrib}\alpha \wedge \beta \le \gamma \Leftrightarrow \alpha \le \beta^\bot \vee \gamma,\end{equation}
for all $\alpha, \beta, \gamma \in \mcM$.  Using the properties of $\bot$ in the hypothesis, we have $\alpha \wedge \beta \le \gamma$ if and only if $\alpha \wedge \beta \wedge \gamma^\bot = 0$, while $\alpha \le \beta^\bot \vee \gamma$ if and only if $\delta = 0$ where $\delta = \alpha \wedge (\beta^\bot \vee \gamma)^\bot$.  Now $\delta = \alpha \wedge \beta \wedge \gamma^\bot$, so (\ref{eq:distrib}) is proved.

Now let $\alpha,\beta,\gamma,\delta \in \mcM$.  By (\ref{eq:distrib}) we have
\[ (\alpha \vee \delta) \wedge \beta \le \gamma \Leftrightarrow \alpha \vee \delta \le \beta^\bot \vee \gamma \Leftrightarrow (\alpha \le \beta^\bot \vee \gamma \; \text{and} \; \delta \le \beta^\bot \vee \gamma).\]
Using (\ref{eq:distrib}) again we have ${\alpha \le \beta^\bot \vee \gamma \Leftrightarrow \alpha \wedge \beta \le \gamma}$ and ${\delta \le \beta^\bot \vee \gamma \Leftrightarrow \delta \wedge \beta \le \gamma}$.  Finally, the conjunction of the inequalities $\alpha \wedge \beta \le \gamma$ and $\delta \wedge \beta \le \gamma$ is given by the inequality $(\alpha \wedge \beta) \vee (\delta \wedge \beta) \le \gamma$, proving the distributive law.\end{proof}

\begin{proof}[Proof of Theorem~\ref{dfaclat}]Fix a bounded locally normal representative $H$ of $\alpha$, such that $H$ has trivial quasi-centre.  Then every finite index subgroup of $H$ has trivial quasi-centre.

Let $K$ and $L$ be almost direct factors of $H$ such that $K$ and $L$ are locally normal in $G$.  Then $K\CC_H(K)$ and $L\CC_H(L)$ both have finite index in $H$, and $K \cap \CC_H(K) = L \cap \CC_H(L) = 1$.  We may replace $K$ with $\CC^2_H(K)$ and so assume $K = \CC^2_H(K)$, since this substitution does not affect the commensurability class of $K$ or its status as an almost direct factor of $H$ that is locally normal in $G$.  Let $M = L\CC_H(L)$ and let $K_2 = K \cap M$.  Then $K_2 = \CC_M(\CC_H(K))$.  Since $K$ is C-stable in $H$ by Lemma~\ref{dirqf}, in fact $\CC_H(K) = \CC_H(K_2)$ by Corollary~\ref{goodqc:lnorm}; in turn, $\CC_H(K)$ has the same centraliser in $H$ as $\CC_M(K)$ for the same reason.  So in fact $K_2 = \CC^2_M(K_2)$.  It follows by Lemma~\ref{dirfac} that $K_2 \cap L$ is a direct factor of $K_2$.  In turn, $K_2$ is a direct factor of $K_2\CC_H(K)$, which has finite index in $H$.  Thus $K_2 \cap L$ is an almost direct factor of $H$.

By a similar argument, there is a finite index subgroup $R$ of $\CC_H(K)$ that factorises as ${(R \cap L) \times \CC_R(L)}$, so there is a finite index subgroup of $H$ that factorises as 
$${K_2 \times (R \cap L) \times \CC_R(L)}.$$  
The first two factors are contained in $K_2L$, while the last factor centralises $K_2L$; additionally $K_2L \cap \CC_R(L) \le \QZ(H) = 1$.  It follows that $K_2L$ is an almost direct factor of $H$.  Finally, note that $K_2$ and $L$ are both bounded and locally normal, and this property is inherited by both their intersection and their product.  We conclude that ${[K_2 \cap L], [K_2L] \in \ldlat(G,U;\alpha)}$, so $\ldlat(G, U ;\alpha)$ is a sublattice of $\lnorm(G, U)$.

By Lemma~\ref{lnorm:trivialqz}, $[\QC_H(K)]$ has a representative $\CC_H(K)$; this also represents an element of $\ldlat(G,U;\alpha)$, so $\bot_\alpha$ preserves $\ldlat(G,U;\alpha)$.  One has $[K] \wedge \beta = 0$ for $\beta \in \ldlat(G,U;\alpha)$ whenever $\beta \le [K]^{\bot_\alpha}$, since $\QZ(K)=1$, and conversely $[K] \wedge \beta = 0$ implies by Lemma~\ref{finint} that some representative of $\beta$ centralises a representative of $[K]$, so in this case $\beta \le [K]^{\bot_\alpha}$.  It is also clear that $\bot_\alpha$ is an order-reversing involution on $\ldlat(G,U;\alpha)$.  Thus $\ldlat(G,U;\alpha)$ is a Boolean algebra by Lemma~\ref{boolflip}.

Now suppose $(G,U)$ is \tdlc and that $\alpha$ has a closed representative.  Then by Proposition~\ref{prop:closed_rep}, $\alpha$ has a compact locally normal representative $H_2$, and moreover from the proof, $H_2$ can be taken to be the closure of a finite index subgroup of $H$.  We see that $H \cap H_2$ has trivial quasi-centre.  Since $H \cap H_2$ is a dense subgroup of $H_2$ of finite index, we have $\QC_G(H_2) = \QC_G(H \cap H_2)$.  In particular, $H_2$ has trivial quasi-centre.  We now follow the same argument as before, with $H_2$ in place of $H$, and see that every element of $\lnorm(G;\alpha)$ has a representative of the form $L = \CC^2_{H_2}(K)$, where $K$ is locally normal.  Since centralisers are closed, it follows that $L$ is compact and locally normal in $G$, so $\lnorm(G;\alpha) \subseteq \lnorm(G)$; since $\lnorm(G;\alpha)$ is a sublattice of $\lnorm(G,U)$, it is also a sublattice of $\lnorm(G)$.
\end{proof}

\begin{proof}[Proof of Theorem~\ref{intro:boolean} (i)] 
Since $G$ has trivial quasi-centre,  the quasi-centraliser $\QC_G(U)$ of any open compact subgroup of $G$ is trivial.  In particular, $U$ is C-stable in $G$, so by Theorem~\ref{dfaclat}, the lattice $\ldlat(G;[U]) = \ldlat(G)$ is a Boolean algebra.
\end{proof}

In the case of a topological group $G$ and a subgroup $H$, there naturally arises the question of how direct factors of $H$ relate to direct factors of $\overline{H}$.  Put another way, one can consider the \defbold{quasi-factors} of a topological group, meaning the direct factors of its dense subgroups.  If $\overline{H}$ is compact and centreless, there is good control of this situation.  As a result, we obtain a relationship between decomposition lattices of \tdlc groups and their dense subgroups.

\begin{lem}\label{lem:compact:quasifactor}Let $G$ be a compact topological group such that $\Z(G)=1$ and let $D$ be a dense subgroup of $G$.  Let $K$ be a direct factor of $D$.  Then
$$ G = \CC_G(K) \times \CC^2_G(K),$$
and $K = \CC^2_G(K) \cap D$.

In particular, the map $K \mapsto \CC^2_G(K)$ defines an injective homomorphism from $\mcL_D$ to $\mcL_G$, where $\mcL_D$ is the lattice of direct factors of $D$ and $\mcL_G$ is the lattice of direct factors of $G$.
\end{lem}

\begin{proof}Since centralisers are closed, we have $\CC_G(D) = 1$.

Let $D = K \times L$.  Then $\Z(L) \le \CC_G(D) = 1$, so $\CC_D(L) = K$.  Similarly, $\CC_D(K) = L$.  Thus $D$ is contained in the set $E = \CC_G(K)\CC_G(L)$.  We see that both $\CC_G(K)$ and $\CC_G(L)$ are closed in $G$, hence compact, so $E$ is compact, hence closed.  Since $E$ is also dense, it follows that $E = G$.  Each of the subgroups $\CC_G(K)$ and $\CC_G(L)$ is normalised by $D$, and hence by $G$, since the normaliser of a closed subgroup is closed.  Moreover, $\CC_G(K) \cap \CC_G(L) = \CC_G(D) = 1$.  Hence
$$ G  = \CC_G(K) \times \CC_G(L).$$
We see that $\Z(\CC_G(K)) \le \Z(G) = 1$, so in fact $\CC_G(L) = \CC^2_G(K)$.  It follows that
$$\CC^2_G(K) \cap D = \CC_D(L) = K.$$

The remaining conclusions are clear.
\end{proof} 

\begin{prop}\label{prop:ldlat:dense}Let $G$ be a \tdlc group such that $\QZ(G)=1$ and let $D$ be a dense subgroup of $G$.  Let $U$ be an open compact subgroup of $G$, and let $R$ be a commensurated subgroup of $D$ that is also a dense subgroup of $U$.  Let $H$ be a bounded locally normal subgroup of the Hecke pair $(D,R)$. Then $\QZ(D)=1$; moreover the following assertions hold.

\begin{enumerate}[(i)]
\item If $\overline{H}$ has trivial quasi-centre, then $H$ has trivial quasi-centre. Moreover the map ${[K] \mapsto [\CC^2_{\overline{H}}(K)]}$ defines an injective order-preserving map from $\ldlat(D,R;H)$ to $\ldlat(G,U,\overline{H})$.
\item If $H$ has trivial quasi-centre, then $\ldlat(G,U;H)$ is embedded in $\ldlat(D,R;H)$ as a sublattice in a natural way.  If $H$ is closed in $G$, then this embedding is an isomorphism.
\end{enumerate}
\end{prop}

\begin{proof}We have $\QC_G(U) \le \QZ(G) = 1$, so $\QC_G(R) = 1$ by Lemma~\ref{lem:quasicent:dense}.  In particular, every finite index subgroup of $R$ has trivial centraliser in $D$.  Thus $\QZ(D)=1$.

Since $H$ is locally normal in $(D,R)$, without loss of generality we may assume $H \le R$.

(i) We have $\QC_G(H) = \QC_G(\overline{H})$ by Lemma~\ref{lem:quasicent:dense}.  Suppose $\QZ(\overline{H}) = 1$.  Then $\QZ(H)=1$.  

Let $K$ be an almost direct factor of $H$ that is locally normal in $(D,R)$.  Then $\N_G(K)$ contains a finite index subgroup of $R$, so $\N_G(\overline{K})$ contains a finite index subgroup of $\overline{R} = U$; thus $\overline{K}$ is compact and locally normal in $(G,U)$.  We see from Lemma~\ref{lem:compact:quasifactor} that the map $\phi \colon [K] \mapsto [\CC^2_{\overline{H}}(K)]$ defines an order-preserving map from $\ldlat(D,R;H)$ to $\ldlat(G,U,\overline{H})$.  Moreover, we have $K = M \cap \CC^2_{\overline{H}}(K)$, where $M$ is any finite index subgroup of $H$ of which $K$ is a direct factor, so $\phi$ is injective.

(ii) If $K$ is an almost direct factor of $H$ that is locally normal in $(G,U)$, then it is also locally normal in $(D,R)$, since $R \le U$.  We can therefore regard $\ldlat(G,U;H)$ as a subset of $\ldlat(D,R;H)$, consisting of those elements $\alpha \in \ldlat(D,R;H)$ such that there exists a representative $K$ of $\alpha$ such that $\N_G(K)$ contains a finite index subgroup of $U$.  Moreover, $\ldlat(G,U;H)$ forms a sublattice of $\ldlat(D,R;H)$, since the meet and join operations for $\ldlat(G,U;H)$ and $\ldlat(D,R;H)$ are essentially the same (in both cases, the meet is the class of the intersection of representatives, and the join is the class of the product of suitably-chosen representatives).

If $H$ is closed in $G$, then every almost direct factor $K$ of $H$ has finite index in a closed almost direct factor $\CC^2_H(K)$, and we see that $\CC^2_H(K)$ is locally normal in $(G,U)$ if and only if it is locally normal in $(D,R)$, in other words $\ldlat(D,R;H) = \ldlat(G,U;H)$.\end{proof}

It is clear that the relation ``$A$ is a direct factor of $B$'' is transitive.  A similar transitivity applies to local decomposition lattices:

\begin{lem}\label{ldtrans}Let $(G,U)$ be a Hecke pair and let $\alpha, \beta \in \lnorm(G,U)$ have bounded locally normal representatives with trivial quasi-centre.  Then $\ldlat(G,U;\alpha) \subseteq \ldlat(G,U;\beta)$ if and only if $\alpha \in \ldlat(G,U;\beta)$.\end{lem}

\begin{proof} Suppose $\alpha \in \ldlat(G,U;\beta)$.  Then there is a representative $H$ of $\alpha$ and a representative $K$ of $\beta$ such that $H$ is a direct factor of $K$.  It follows that all almost direct factors of $H$ are also almost direct factors of $K$.\end{proof}

\subsection{Locally finitely decomposable \tdlc groups}

The following are natural properties to consider in the context of local decomposition lattices:

\begin{defn}Let $H$ be a group.  Say $H$ is \defbold{finitely decomposable}\index{finitely decomposable} if
$$H = H_1 \times \dots \times H_n$$
for some $n \ge 1$, such that each $H_i$ is \defbold{indecomposable}\index{indecomposable}, i.e. all of whose direct product decompositions are trivial.  Given a \tdlc group $G$, say $G$ is \defbold{locally finitely decomposable}\index{locally finitely decomposable}\index{finitely decomposable!locally} if every open compact subgroup of $G$ is finitely decomposable.\end{defn}

\begin{prop}\label{infdecomp}Let $G$ be a group such that $\Z(G)=1$.
\begin{enumerate}[(i)]
\item Suppose $G = G_1 \times \dots \times G_n$ for some $n \ge 1$, such that each $G_i$ is indecomposable.  Then $\{G_1,\dots,G_n\}$ is precisely the set of indecomposable direct factors of $G$ and every direct factor of $G$ is a product of some subset of $\{G_1,\dots,G_n\}$.  In particular, $\Aut(G)$ acts on $\{G_1,\dots,G_n\}$ and there are exactly $2^n$ direct factors of $G$.
\item Suppose that $G$ is compact and not finitely decomposable.  Then $G$ is topologically isomorphic to an infinite Cartesian product of closed subgroups.  In particular, $G$ has uncountably many closed direct factors up to commensurability.
\end{enumerate}
\end{prop}

\begin{proof}For any direct factor $H$ of $G$, then $\CC_G(H)$ is the complement of $H$ and thus also a direct factor, so $H = \CC_G(\CC_G(H))$.  Suppose $G = G_1 \times \dots \times G_n$ for some $n \ge 1$, such that each $G_i$ is indecomposable.  Let $K$ be a direct factor of $G$.  Then $G = K\CC_G(K)$, and $K \cap \CC_G(K) \le \Z(G) = 1$.  Let $I = \{ i \mid G_i \not\le \CC_G(K)\}$.  Then $K \le \prod_{i \in I} G_i$, since $K$ is centralised by the remaining factors.  Given $i \in I$, then $K \cap G_i$ is a direct factor of both $K$ and $G_i$, by Lemma~\ref{dirfac}, so $K=G_i$ by the fact that $K$ and $G_i$ are indecomposable.  Thus $K = \prod_{i \in I} G_i$.

Suppose $G$ is compact and not finitely decomposable.  We shall define a family $\{G_{i,j}\}_{1 \le j \le i}$ of non-trivial subgroups as follows, so that for all $n$, $G = G_{n,1} \times \dots \times G_{n,n}$ and $G_{n,n}$ is not finitely decomposable.

Set $G_{1,1} = G$.  Suppose we have $G = G_{n,1} \times \dots \times G_{n,n}$ for some $n$.  Choose some non-trivial direct decomposition $G_{n,n} = H_n \times K_n$, ensuring that $K_n$ is not finitely decomposable (this is always possible, as by induction $G_{n,n}$ is not finitely decomposable).  Now set $G_{n+1, n} = H_n$, $G_{n+1,n+1} = K_n$ and $G_{n+1,j} = G_{n,j}$ for $1 \le j < n$.

Let $K = \bigcap^\infty_{j=1} G_{j,j}$.  We claim now that
\[ G \cong K \times \prod^\infty_{j=1} G_{j+1,j}.\]
The decomposition $G = G_{n,1} \times \dots \times G_{n+1,n+1}$ gives a continuous epimorphism $G \rightarrow \prod^n_{j=1} G_{n+1,j}$ with kernel $G_{n+1,n+1}$, and we note from the construction that $\prod^n_{j=1} G_{n+1,j} =\prod^n_{j=1} G_{j+1,j}$.  These epimorphisms are compatible with each other as $n$ varies so that they generate an inverse system, with a limit homomorphism $\phi: G \rightarrow \prod^\infty_{j=1} G_{j+1,j}$.  We see that $\ker\phi$ is precisely $K$.  In addition $\phi(\CC_G(K))$ contains $\bigoplus^\infty_{j=1} G_{j+1,j}$, since $G_{j+1,j}$ centralises $G_{n+1,n+1}$ for all $1 \le j \le n$, so $\phi(\CC_G(K)) = \prod^\infty_{j=1} G_{j+1,j}$.  In particular $G = K\CC_G(K)$: this implies $K \cap \CC_G(K) \le \Z(G) = 1$, so in fact $G = K \times \CC_G(K)$.  The restriction of $\phi$ to $\CC_G(K)$ is consequently bijective, giving the desired isomorphism.  We now obtain uncountably many direct factors $G_J = \prod_{j \in J} G_{j+1,j}$ by letting $J$ range over all subsets of $\bN$.  For two such direct factors $\prod_{j \in J_1} G_{j+1,j}$ and $\prod_{j \in J_2} G_{j+1,j}$ to be commensurate, the symmetric difference of the sets $J_1$ and $J_2$ must be finite.  This defines an equivalence relation on the power set of $\bN$ whose equivalence classes are countable, so there must be uncountably many equivalence classes.\end{proof}

\begin{cor}\label{cor:LocallyFinitelyDec}
Let $G$ be a \tdlc group with $\QZ(G) =1$.
\begin{enumerate}[(i)]
\item $G$ is locally finitely decomposable if and only if $G$ has a basis of identity neighbourhoods consisting of open compact subgroups that are finitely decomposable.
\item Suppose $G$ is first-countable.  Then $G$ is locally finitely decomposable if and only if $\ldlat(G)$ is countable.
\end{enumerate}
\end{cor}

\begin{proof}
(i) The necessity of the condition is obvious. Conversely, if $G$ is not locally finitely decomposable, then $G$ has some open compact subgroup $V$ which is not finitely decomposable. Since $G$ has trivial quasi-centre, the group $V$ must be center-free, and is thus an infinite Cartesian product of closed subgroups by Proposition~\ref{infdecomp}. Now, any open subgroup $W$ of $V$ contains infinitely many direct factors of $V$; those factors are then also direct factors of $W$, and hence $W$ cannot be finitely decomposable. 

(ii) Suppose that $G$ is locally finitely decomposable.  Since $\ldlat(G) = \ldlat(U)$ for any open compact subgroup $U$, we may assume that $G$ is compact.  In this case $G$ has only countably many open subgroups, each with only finitely many direct factors, and thus $\ldlat(G)$ is countable.

Conversely, suppose $\ldlat(G)$ is countable.  Then given an open compact subgroup $U$ of $G$, the direct factors of $U$ fall into countably many commensurability classes.  Thus $U$ is finitely decomposable by Proposition~\ref{infdecomp}.
\end{proof}

\begin{prop}Let $G$ be a \tdlc group and let $H$ be a C-stable compact locally normal subgroup that is locally finitely decomposable.  Then $\ldlat(G;H) = \ldlat(H)$; indeed $\ldlat(G;H)$ contains $\ldlat(K)$ for every almost direct factor $K$ of $H$.\end{prop}

\begin{proof}It is clear that $\ldlat(G;H) \subseteq \ldlat(H)$.  Let $L$ be an open subgroup of $H$, and let $K$ be a direct factor of $L$.  Then $\N_G(L)$ is open in $G$ by Lemma~\ref{lem:LN:basic}(ii), and it preserves the set $X$ of indecomposable direct factors of $L$; since $X$ is finite, there is an open subgroup $R$ of $\N_G(L)$ that normalises every indecomposable direct factor of $L$.  In turn $K$ is a product of indecomposable direct factors of $L$, so it is normalised by $R$.  Hence $K$ is locally normal, so represents an element of $\ldlat(G;H)$.  Hence $\ldlat(G;H) = \ldlat(H)$, and $\ldlat(H)$ contains $\ldlat(H;K)$ by Lemma~\ref{ldtrans}.  In turn, $K$ is locally finitely decomposable, because given an open subgroup $M$ of $K$, the direct factors of $M$ are also direct factors of $M\CC_H(K)$, which is an open subgroup of $H$.  Hence $\ldlat(H;K) = \ldlat(K)$.\end{proof}

\section{The centraliser lattice and weakly decomposable actions}
\label{sec:lcent}

\subsection{The centraliser lattice}\label{clatsect}

It would be useful to combine all the local decomposition lattices of $(G,U)$ (which are not necessarily individually preserved by the action of $G$) into a single Boolean algebra admitting a $G$-action.  There are two difficulties here: one is that the property of having trivial quasi-centre may not be well-behaved under intersections, and the other is that we need a unified notion of complementation.  To avoid the first difficulty and overcome the second, we restrict our attention to locally C-stable Hecke pairs $(G,U)$ in the sense of \S\ref{qstabsec}.  Recall from Theorem~\ref{thm:locallyCstable} that under the assumption that $U$ is residually finite, these are precisely the Hecke pairs with trivial quasi-centre and no non-trivial abelian locally normal subgroups.  The resulting lattice will not be a sublattice of $\lnorm(G,U)$ in general (because the join operation is different), but it arises from the structure lattice in a natural way.

\begin{defn}
Let $(G,U)$ be a locally C-stable Hecke pair.  Define the map 
$$\bot \colon \lnorm(G,U) \rightarrow \lnorm(G,U) : \alpha \mapsto [\QC_G(\alpha)].$$ 
This is well-defined by Theorem~\ref{thm:locallyCstable}, since there exists a finite index subgroup $V$ of $U$ such that $\QZ(V)=1$, and we have 
$$[\QC_G(\alpha)] = [\QC_V(\alpha)] = [\CC_G(K)]$$ 
for every locally normal representative $K$ of $\alpha$ contained in $V$.  The \defbold{centraliser lattice}\index{centraliser lattice}\index{centraliser lattice!of a Hecke pair} $\lcent(G,U)$\index{LC(G,U)@$\lcent(G, U)$} is defined to be the set $\{\alpha^\bot \mid \alpha \in \lnorm(G,U)\}$ together with the map $\bot$ restricted to $\lcent(G,U)$, the partial order inherited from $\lnorm(G,U)$ and the binary operations $\wedge_c$ and $\vee_c$ given by:
\[ \alpha \wedge_c \beta = \alpha \wedge \beta \]
\[ \alpha \vee_c \beta = (\alpha^\bot \wedge \beta^\bot)^\bot.\]
In general we will write $\vee$ instead of $\vee_c$ in contexts where it is clear that we are working inside the centraliser lattice.

Given a \tdlc group $G$, we define $\lcent(G) := \lcent(G,U)$, where $U$ is some open compact subgroup of $G$.  (The choice of $U$ here is inconsequential.)\end{defn}

\begin{thm}\label{centlat}
Let $(G,U)$ be a locally C-stable Hecke pair, such that $U$ is residually finite.
\begin{enumerate}[(i)]
\item The poset $\lcent(G,U)$ is a Boolean algebra and $\bot^2:\lnorm(G,U) \rightarrow \lcent(G,U)$ is an idempotent surjective lattice homomorphism.  If $G$ is \tdlc, then
$$ \lcent(G) = \{\alpha^\bot \mid \alpha \in \lnorm(G)\} \subseteq \lnorm(G).$$
\item Fix $\alpha \in \lnorm(G,U)$ and let $\theta$ be the map on $\lnorm(G,U)$ given by $\beta \mapsto \beta^{\bot^2} \wedge \alpha$.  Then $\theta$ is an idempotent order-preserving map, and $\theta(\lnorm(G,U)) = \theta(\lcent(G,U)) \supseteq \ldlat(G,U;\alpha)$.  Moreover, $\theta\inv(\ldlat(G,U;\alpha)) \cap \lcent(G,U)$ is a subalgebra of $\lcent(G,U)$.
\end{enumerate}
\end{thm}

We need the following basic group theoretic facts.
 
\begin{lem}\label{cclem}Let $G$ be a group such that every normal subgroup has trivial centre, in other words that $G$ has no non-trivial abelian normal subgroups.  Then the following equations are satisfied for any normal subgroups $A$ and $B$ of $G$.
\begin{enumerate}[(i)]
\item $\CC^2_G(A \cap B) = \CC^2_G(A) \cap \CC^2_G(B)$;
\item ${\CC_A(B) = \CC_A(A \cap B)}$;
\item ${\CC_A(\CC_G(B)) = \CC_A(\CC_A(B))}$.
\end{enumerate}\end{lem}

\begin{proof}
(i) Let $K = \CC^2_G(A)$ and let $L = \CC^2_G(B)$. It is clear that
\[ \CC_G(\CC_G(A \cap B)) \le \CC_G(\CC_G(A)\CC_G(B)) = K \cap L,\]
so it suffices to show that $K \cap L$ centralises $\CC_G(A \cap B)$.  For this it is enough to show that $M = \CC_G(A \cap B) \cap K \cap L = 1$, since all subgroups under consideration are normal.  We see that $M \cap A \cap B \le \Z(A \cap B) = 1$, so $M \cap A$ centralises $B$, in other words $M \cap A \le \CC_G(B)$.  Then $M \cap A$ is centralised by $L$, so $M \cap A$ is central in $M$, so $M \cap A = 1$, and hence $M \le \CC_G(A)$.  Since $K$ centralises $\CC_G(A)$ and contains $M$, we see that $M$ is central in $\CC_G(A)$ and hence trivial, as required.

(ii) Notice that $\CC_A(A \cap B)$ is normal in $G$ and has trivial intersection with $A \cap B$ (since $\Z(A \cap B)=1$). Therefore $ [\CC_A(A \cap B), B] \subset  \CC_A(A \cap B) \cap B =1$, whence $ \CC_A(A \cap B) \leq \CC_A(B)$. The reverse inclusion is clear.

(iii) This a special case of (ii), using the normal subgroups $A$ and $\CC_G(B)$ instead of $A$ and $B$.
\end{proof}

\begin{proof}[Proof of Theorem \ref{centlat}]Since we are concerned only with properties of a local nature, we are free to replace $G$ and $U$ with a finite index subgroup $V$ of $U$ (which can be assumed to be open in the \tdlc case) such that $\QZ(V)=1$.  In other words, we may assume that $G$ has trivial quasi-centre.

(i) Let $A$ be a bounded locally normal subgroup of $G$, let $\alpha = [A]$, and let $V$ be a subgroup of $G$ commensurate with $U$ such that $A \unlhd V$.  By Lemma~\ref{goodqc}, $\CC^n_V(A)$ is then a representative of $\alpha^{\bot^n}$.  We deduce $\alpha^{\bot^2} \ge \alpha$ and $\alpha^{\bot^3} = \alpha^\bot$ from Lemma~\ref{elcent}.  In particular, $\bot$ acts as an involution on $\lcent(G,U)$, so $\bot^2(\lnorm(G,U)) = \lcent(G,U)$.  In the \tdlc case, we note that $\CC_U(\CC_V(A))$ is a closed representative of $\alpha^{\bot^2}$, so $\alpha^{\bot^2} \in \lnorm(G)$, and since $\bot^3 = \bot$, it follows that $\bot(\lnorm(G)) = \bot(\lnorm(G,U)) =: \lcent(G)$ and that $\lcent(G)$ is contained in $\lnorm(G)$.

It is clear that $\bot$ is order-reversing (so $\bot^2$ is order-preserving).  Let $\alpha,\beta \in \lnorm(G,U)$ and choose representatives $A$ and $B$ of $\alpha, \beta$ respectively that are normal in a subgroup $V$ that is commensurate with $U$. Then  we have $\CC_V(AB) = \CC_V(A) \cap \CC_V(B)$, so that 
\begin{equation}\label{eq:vee} (\alpha \vee \beta)^\bot = \alpha^\bot \wedge \beta^\bot.\end{equation}
This shows that $\lcent(G,U)$ is a meet-sublattice of $\lnorm(G,U)$.  Since the only abelian normal subgroup of $V$ is the trivial one, it also follows from Lemma~\ref{cclem}(i) that $\bot^2$ is compatible with the meet operation.

By Lemma~\ref{finint}, $\alpha \wedge \beta = 0$ if and only if $B$ centralises $A$, which amounts to saying that $\beta \le \alpha^\bot$.  Hence $\lcent(G,U)$ is a Boolean algebra by Lemma~\ref{boolflip}.

We see that $\bot^2$ is compatible with the join operations in passing from $\lnorm(G,U)$ to $\lcent(G,U)$ as follows, using (\ref{eq:vee}):
\[ \alpha^{\bot^2} \vee_c \beta^{\bot^2} = (\alpha^{\bot^3} \wedge \beta^{\bot^3})^\bot =   (\alpha^{\bot} \wedge \beta^{\bot})^\bot = (\alpha \vee \beta)^{\bot^2}.\]
Hence $\bot^2$ is a lattice homomorphism.

(ii) It is clear that $\theta(\lnorm(G,U)) = \theta(\lcent(G,U))$, since $\bot^2$ is a projection of $\lnorm(G,U)$ to $\lcent(G,U)$.  It is also clear that $\theta$ is order-preserving.  Given $\beta \in \lnorm(G,U)$ we have
\[ \theta^2(\beta) = (\beta^{\bot^2} \wedge \alpha)^{\bot^2} \wedge \alpha =  \beta^{\bot^4} \wedge \alpha^{\bot^2} \wedge \alpha = \beta^{\bot^2} \wedge \alpha,\]
so $\theta^2 = \theta$.

Let $A$ be a bounded locally normal representative of $\alpha$ and let $K$ be a direct factor of $A$.  Choose $V$ commensurate with $U$ such that both $A$ and $K$ are normal in $U$.  Then $K = \CC^2_A(K)$, so $K = \CC_A(\CC_V(K))$ by Lemma \ref{cclem}(iii).  In other words, $[K] = \beta^{\bot^2} \wedge \alpha$ where $\beta = \beta^{\bot^2}$ is the element $[\CC^2_V(K)] = [\QC^2_U(K)]$ of $\lcent(G,U)$.  Thus $\theta(\lcent(G,U))$ contains $\ldlat(G,U;\alpha)$.

Let $\mcL = \theta\inv(\ldlat(G,U;\alpha)) \cap \lcent(G,U)$.  Then $\mcL$ is closed under meets in $\lcent(G,U)$, by the fact that $\bot^2$ is compatible with the meet operation.  Furthermore $\mcL$ is closed under complementation in $\lcent(G,U)$, since the complement of $\beta \wedge \alpha$ in $\ldlat(G,U;\alpha)$ is $(\beta \wedge \alpha)^\bot \wedge \alpha$ by Theorem~\ref{dfaclat}. Using Lemma~\ref{cclem}(ii) we see that
\[ (\beta \wedge \alpha)^\bot \wedge \alpha = [\CC_A(A \cap B)] = [\CC_A(B)] = \beta^\bot \wedge \alpha,\]
where $A$ and $B$ are representatives of $\alpha$ and $\beta$ respectively that normalise each other.
So $\mcL$ is a subalgebra of $\lcent(G,U)$.\end{proof}

Theorem~\ref{intro:boolean} (ii) now follows from Theorem~\ref{centlat} together with Theorem~\ref{thm:locallyCstable}.

Let $(G,U)$ be a locally C-stable Hecke pair such that $\QZ(G)=1$.  Then each element $\alpha$ of the centraliser lattice has a canonical `global' representative, namely $\QC_G(\alpha^\bot)$.  This subgroup is the unique largest locally normal representative of $\alpha$.  If $(G,U)$ is \tdlc, then $\QC_G(\alpha^\bot)$ is closed, but not necessarily compact.

We can consider the set $\mcL = \{\QC_G(\alpha) \mid \alpha \in \lnorm(G,U)\}$, ordered by inclusion, as the \defbold{global centraliser lattice}\index{centraliser lattice!global}\index{global centraliser lattice} of $(G,U)$.  As a Boolean algebra, it is $G$-equivariantly isomorphic to $\lcent(G,U)$.  The elements of $\mcL$ also have a more direct characterisation using Theorem~\ref{thm:locallyCstable}, as follows:

\begin{prop}\label{prop:globalcent}
Let $(G,U)$ be a locally C-stable Hecke pair  such that $U$ is residually finite and $\QZ(G)=1$.  Then the set of subgroups $\mcL := \{\QC_G(\alpha) \mid \alpha \in \lnorm(G,U)\}$ enjoys the following properties:
\begin{enumerate}[(i)]
\item If $H \in \mcL$, then $H = \CC^2_G(H)$ and $H$ is locally normal.   Indeed $H \cap R = \CC^2_R(H)$, where $R$ is any locally normal subgroup of $G$.

\item Let $K$ be a  subgroup of $G$ such that $\CC_G(K)$ is locally normal.  Then $\CC_G(K) \in \mcL$.  If $K$ is locally normal, in fact $\CC_G(K) = \QC_G([K])$ and $\CC^2_G(K) = \QC_G([K]^\bot)$.
\end{enumerate}
\end{prop}

\begin{proof}
Let $\alpha \in \lnorm(G,U)$, let $H = \QC_G(\alpha)$.  Since $\alpha$ is an element of $\lnorm(G,U)$, then $|U:U_\alpha|$ is finite; in turn $U_\alpha$ normalises $H$, since $H$ is defined as a subgroup of $G$ in terms of $\alpha$.  Hence $|U:\N_U(H)|$ is finite, so that $H$ is locally normal.

Let $L = \QC_G(\alpha^\bot)$.  Then $L = \CC_G(H)$ by Theorem~\ref{thm:locallyCstable}.  The same argument shows that $H = \CC_G(L)$, so $H = \CC^2_G(H)$.   
Now let $M = \CC_G(H)$, and suppose $R$ is a locally normal subgroup of $G$.  Then $R \cap M$ is locally normal and hence centreless, since $G$ is locally C-stable and $\QZ(G)=1$.  Then $\CC_R(R \cap M) \ge H \cap R$, since $M$ centralises $H$.  Moreover, we have
$$\CC_R(R \cap M) \cap M = \CC_R(R \cap M) \cap (R \cap M) \le \Z(R \cap M) = 1.$$
Thus $\CC_R(R \cap M)$ and $M$ are locally normal subgroups that intersect trivially, so they commute by Theorem~\ref{thm:locallyCstable}(iii), so that   $\CC_R(R \cap M) = \CC_R(M)$, and hence $\CC_R(R \cap M) \le \CC^2_G(H) = H$.  This proves (i).

Now let $K$ be a locally normal subgroup of $G$.  Then $\CC_G(K) = \QC_G([K])$ by Theorem~\ref{thm:locallyCstable}, so in particular $\CC_G(K) \in \mcL$.  Setting $L = \CC_G(K)$, the same argument shows that $\CC_G(L) = \QC_G(\alpha^\bot)$.  More generally, if $K$ is a subgroup of $G$ such that $\CC_G(K)$ is locally normal, then $\CC_G(K) = \CC^3_G(K)$, and in turn $\CC^3_G(K)$ is the centraliser of a locally normal subgroup $\CC^2_G(K)$ of $G$, so $\CC_G(K) \in \mcL$.  This completes the proof of (ii).
\end{proof}

As a consequence of Proposition~\ref{prop:globalcent} and Lemma~\ref{dirfac}, we obtain information about the decomposability of centralisers in $G$.

\begin{cor}\label{globalcent:dirfac}
Let $(G,U)$ be a locally C-stable Hecke pair  such that $U$ is residually finite and $\QZ(G)=1$.  Let $A$ and $B$ be locally normal subgroups of $G$.  Suppose that $B = K \times L$ for some   subgroups $K$ and $L$.  Then
$$ \CC_B(A) = \CC_K(A) \times \CC_L(A).$$
\end{cor}

\begin{proof}Let $H = \CC_G(A)$ and let $R = \CC_B(A)$.  Then $R = \CC^2_B(H)$ by Proposition~\ref{prop:globalcent}; consequently $R \ge \CC^2_B(R)$, so in fact $R = \CC^2_B(R)$.  Thus by Lemma~\ref{dirfac}, $R \cap K$ and $R \cap L$ are complementary direct factors of $R$, giving the required factorisation.\end{proof}

Similar to Proposition~\ref{prop:ldlat:dense}, there is a relationship between the centraliser lattice of a \tdlc group and that of its dense subgroups, viewed as Hecke pairs.

\begin{prop}\label{prop:lcent:dense}Let $G$ be a locally C-stable \tdlc group such that $\QZ(G)=1$ and let $D$ be a dense subgroup of $G$.  Let $U$ be an open compact subgroup of $G$, and let $R$ be a commensurated subgroup of $D$ that is also a dense subgroup of $U$.
\begin{enumerate}[(i)]
\item The Hecke pair $(D,R)$ is locally C-stable, with $\QZ(D)=1$.
\item Let $\mcL$ be the global centraliser lattice of $(G,U)$ and let $\mcL'$ be the global centraliser lattice of $(D,R)$.  Then there is an injective $D$-equivariant order-preserving map from $\mcL'$ to $\mcL$ given by $K \mapsto \CC^2_G(K)$.
\end{enumerate}
\end{prop}

\begin{proof}
Given a locally normal subgroup $L$ of $D$, then $\overline{L}$ and $\CC_G(L)$ are both closed and locally normal in $G$; in particular, $\CC_G(L) \in \mcL$ by Proposition~\ref{prop:globalcent}.  Assuming $L$ is non-trivial, then since $\overline{L}$ cannot be finite or abelian, neither can $L$.  In addition, we have $\QZ(D)=1$ by Proposition~\ref{prop:ldlat:dense}.  This proves (i) in view of Theorem~\ref{thm:locallyCstable}.

Given $K \in \mcL'$, then $\CC_G(K) \in \mcL$ as above, hence $\CC^2_G(K) \in \mcL$.  Thus $\theta: K \mapsto \CC^2_G(K)$ is a map from $\mcL'$ to $\mcL$.    Clearly $\theta$ is $D$-equivariant and order-preserving.  Since $K = \CC^2_D(K)$, we see that ${K = \theta(K) \cap D}$ as follows:
$$K \le \theta(K) \cap D = \CC^2_G(K) \cap D \le \CC_G(\CC_D(K)) \cap D = \CC^2_D(K) = K.$$
Thus $\theta$ is injective.  This proves (ii).
\end{proof}

\begin{ex}In general, given a \tdlc group $G$, the centraliser lattice of a dense subgroup of $G$ does not determine the centraliser lattice of $G$.  Let $F$ be a non-abelian free group (with the discrete topology), and consider $F$ as a Hecke pair $(F,F)$.  It is clear that $F$ is locally C-stable, so we can define the centraliser lattice $\lcent(F,F)$.  Since $\CC_F(x)$ is cyclic for all $x \in F \smallsetminus \{1\}$, in fact $\lcent(F,F) = \{0, \infty\}$.  Likewise $\lcent(\hat{F}_p) = \{0,\infty\}$, where $\hat{F}_p$ is the pro-$p$ completion of $F$.  However, there are many pro-$p$ groups $P$ with dense free subgroups such that $\lcent(P)$ is large.  For instance, if $P$ is a pro-$p$ branch group, then by results of J. Wilson, $P$ acts faithfully on $\lcent(P)$, and $P$ contains a dense non-abelian free subgroup.  (See \cite{WilNH}.)\end{ex}

The global centraliser lattice is a natural frame of reference for considering quasi-factors of a locally C-stable \tdlc group that has trivial quasi-centre.  In contrast to the compact case considered in Lemma~\ref{lem:compact:quasifactor}, a closed quasi-factor of a non-compact \tdlc group $G$ need not be a direct factor of $G$.  (See \cite{CM}, Appendix II.)

\begin{prop}\label{prop:globalcent:quasifactor}Let $G$ be a locally C-stable \tdlc group with $\QZ(G)=1$.  Let $K$ be a closed normal subgroup of $G$, and suppose $K\CC_G(K)$ is dense in $G$.  Let $\alpha = [K]$.  Then there is a dense subgroup of $G$ of the form $\QC_G(\alpha) \times \QC_G(\alpha^\bot)$, where both direct factors are closed and normal in $G$.  In addition, the following are equivalent:
\begin{enumerate}[(i)]
\item $G = K \times \CC_G(K)$;
\item $G = \QC_G(\alpha) \times \QC_G(\alpha^\bot)$;
\item $\alpha \in \ldlat(G)$;
\item $\alpha^{\bot^2} \in \ldlat(G)$.
\end{enumerate}
\end{prop}

\begin{proof}We have $\QC_G(\alpha) = \CC_G(K)$ and $\QC_G(\alpha^\bot) = \CC^2_G(K)$ by Proposition~\ref{prop:globalcent}; in particular, both $\QC_G(\alpha)$ and $\QC_G(\alpha^\bot)$ are closed.  Since $K$ is normal, both $\QC_G(\alpha)$ and $\QC_G(\alpha^\bot)$ are normal in $G$.  Since $\Z(G)=1$, we see that both $K$ and $\CC_G(K)$ have trivial centre.  This ensures that the subgroup generated by $\QC_G(\alpha)$ and $\QC_G(\alpha^\bot)$ is a direct product.

Let $U$ be an open compact subgroup of $G$ and let $L$ be a direct factor of $G$.  Then $L = \CC^2_G(L)$, so $[L] = [L]^{\bot^2}$.  Moreover, the topology of $G$ is generated as a product topology by the topologies of $L$ and $\CC_G(L)$.  In particular
$$(L \cap U) \times \CC_U(L)$$
is open in $G$, hence has finite index in $U$.  Thus $L\cap U$ is an almost direct factor of $U$, so $\alpha = [L \cap U] \in \ldlat(G)$.  Thus (i) implies (iii) and (ii) implies (iv).

On the other hand, let $L$ be a closed normal subgroup of $G$ such that $L\CC_G(L)$ is dense and $[L] \in \ldlat(G)$.
$$(L \cap U) \times \CC_U(L \cap U)$$
is open in $U$, since it is a closed subgroup of finite index, and moreover $\CC_U(L \cap U) = \CC_U(L)$ by Theorem~\ref{thm:locallyCstable}.  Hence $L\CC_G(L)$ contains an open subgroup of $G$.  Since $L\CC_G(L)$ is also dense, we conclude that $G = L \times \CC_G(L)$.  So (iii) implies (i) and (iv) implies (ii).

It is clear that (i) implies (ii), since $K \le \QC_G(\alpha^\bot)$ and $\CC_G(K) = \QC_G(\alpha)$.  Conversely if (ii) holds, then $K \times \CC_G(K)$ is a closed subgroup of $G = \QC_G(\alpha) \times \QC_G(\alpha^\bot)$, since $K$ is closed in $\QC_G(\alpha^\bot)$ and $\CC_G(K) = \QC_G(\alpha)$.  Since $K \times \CC_G(K)$ is also dense, it follows that (i) holds.  This completes the proof that (i), (ii), (iii) and (iv) are equivalent.
\end{proof}

\begin{cor}Let $G$ be a locally C-stable \tdlc group such that $\lcent(G) =\ldlat(G)$ and such that $\QZ(G)=1$.  Then every closed quasi-factor of $G$ is a direct factor of $G$.\end{cor}

\subsection{Weakly decomposable actions}

Let $\mcA$ be a Boolean algebra.  Then by the Stone representation theorem, $\mcA$ defines a \defbold{profinite space}\index{profinite space} (that is, a compact totally disconnected space), the \defbold{Stone space}\index{Stone space} $\mfS(\mcA)$ of $\mcA$, whose points are the ultrafilters of $\mcA$, with clopen subsets corresponding to elements of $\mcA$.  Conversely, given a profinite space $\mfX$, the set $\mcA(\mfX)$ of clopen subsets of $\mfX$ forms a Boolean algebra.  This correspondence produces a natural isomorphism between $\Aut(\mcA)$ and $\Aut(\mfS(\mcA))$.  In practice we will often find it convenient to abuse notation and treat elements of $\mcA$ as subsets of $\mfS(\mcA)$, identifying $\alpha \in \mcA$ with 
\[\mfS(\alpha) := \{\mfp \in \mfS(\mcA) \mid \alpha \in \mfp\} \in \mcA(\mfS(\mcA)).\]
Given $\alpha \in \mcA$ and $\mfp \in \mfS(\mcA)$, the expressions `$\alpha \in \mfp$' and `$\mfp \in \alpha$' can therefore be taken to be synonymous.  A \defbold{partition}\index{partition} $\mcP$ of $\alpha \in \mcA$ is a finite subset of $\mcA$ such that the join of $\mcP$ is $\alpha$ and the meet of any two distinct elements of $\mcP$ is $0$; a partition of $\mcA$ is just a partition of $\infty$ in $\mcA$.

\begin{defn}
Let $G$ be a topological group acting on a set $\Omega$.  Say that the action of $G$ on $\Omega$ is {\defbold{smooth}}\index{smooth action}\index{action!smooth} if every point stabiliser is open; this implies that for every open compact subgroup $U$ of $G$, the orbits of $U$ on $\Omega$ are all finite.  Note that if the action is smooth, then $G$ is orbit-equivalent to any of its dense subgroups.
\end{defn}

The following fact shows the relevance of smooth actions for  \tdlc groups:

\begin{lem}\label{lem:SmoothActionContinuous}
Let $G$ be a \tdlc group, let $\mfX$ be a profinite space and  $\mcA$ a Boolean algebra.

\begin{enumerate}[(i)]
\item If $G$ acts on $\mfX$ by homeomorphisms, and if the $G$-action is continuous with respect to the topology of uniform convergence, then the corresponding $G$-action on the Boolean algebra of clopen subsets of $\mfX$ is smooth.

\item If $G$ acts on $\mcA$ by automorphisms, and if the $G$-action is smooth, then the corresponding $G$-action on the Stone space $\mfS(\mcA)$ is continuous with respect to the topology of uniform convergence.  In particular, the action map $(g,x) \mapsto gx$ is a continuous map from $G \times \mfS(\mcA)$ to $\mfS(\mcA)$.
\end{enumerate}
\end{lem}

\begin{proof}
Let  $U \in \mcB(G)$ and let $\mfX = \mfX_1 \cup \dots \cup \mfX_n$ be a clopen partition of $\mfS(\mcA)$. 

 If $G$ acts on $\mfX$ by homeomorphisms, and if the $G$-action is continuous with respect to the topology of uniform convergence, then for any net $(g_\alpha)$ converging to the identity in $G$, for all $i$ and each $x \in \mfX_i$, we have $g_\alpha(x) \in \mfX_i$ for all sufficiently large $\alpha$. Therefore $g_\alpha$ eventually belongs to the common stabiliser  $H = \bigcap_{i=1}^n G_{\mfX_i}$, which is thus an open subgroup of $G$. It follows that $H \cap U$ is open in $G$, and so it therefore  $U_{\mfX_i} = U \cap G_{\mfX_i}$. Thus (i) holds. 
 
 Assume now that  $G$ acts smoothly on $\mcA$ by automorphisms and set $\mfX = \mfS(\mcA)$. 
Since the action is smooth, the common stabiliser $V = \bigcap_{i=1}^n U_{\mfX_i}$ is open in $U$. Thus any net   $(g_\alpha)$ converging to the identity in $G$ eventually belongs to $V$ and thus stabilises each $\mfX_i$. Thus $(g_\alpha)$ uniformly converges to the identity in the space of homeomorphisms of $\mfX$. In particular the action of $G$ on $\mfS(\mcA)$ is jointly continuous (see \cite[Chapter 7]{Kelley}). This proves (ii).
\end{proof}

As before, let $G$ be a group acting on a set $\Omega$. Given $X \subseteq \Omega$, the \defbold{rigid stabiliser}\index{rigid stabiliser} of $X$ is the subgroup
\[ \rist_G(X) := \{ g \in G \mid gx = x  \text{ for all } x \in \Omega \smallsetminus X \}.\]

Similarly, if $G$ acts on a Boolean algebra $\mcA$, then given $\alpha \in \mcA$ we define\index{rist@$\rist_G(\alpha)$}
\[ \rist_G(\alpha) := \{ g \in G \mid g\beta = \beta  \text{ for all } \beta \le \alpha^\bot\}.\]
(Equivalently, one could define $\rist_G(\alpha)$ to be the rigid stabiliser of the subset of $\mfS(\mcA)$ corresponding to $\alpha$.)

Notice that if $\mcA$ is a subalgebra of $\lcent(G)$ or of $\ldlat(G)$, then $\QC_G(\alpha) \le \rist_G(\alpha^\bot)$ for all $\alpha \in \mcA$.  

\begin{lem}\label{clofix}
Let $G$ be a \tdlc group acting smoothly on a Boolean algebra $\mcA$ and let $\upsilon \subseteq \mfS(\mcA)$.  Then $\rist_G(\upsilon)$ is closed.
\end{lem}

\begin{proof}It suffices to show that $G_\mfp$ is closed for all $\mfp \in \mfS(\mcA)$. {This is immediate from Lemma~\ref{lem:SmoothActionContinuous}.}\end{proof}
%
%
%

Related to the notion of rigid stabilisers is that of \defbold{(weakly, locally) decomposable} actions.  The centraliser and local decomposition lattices account for all faithful actions of this kind of a \tdlc group (if any such actions exist).

\begin{defn}
Let $(G,U)$ be a Hecke pair such that $G$ acts on a Boolean algebra $\mcA$, with kernel $K$.  Say the action is \defbold{weakly decomposable}\index{weakly decomposable action}\index{action!weakly decomposable} if, for every $\alpha \in \mcA \smallsetminus \{0\}$, the $U$-orbit of $\alpha$ is finite and $\rist_G(\alpha)/K$ is non-trivial.  Note that the fact that $U$ has finite orbits ensures that $UK/K$ is residually finite. Moreover, taking $\alpha = \infty$, we see that $G/K$ must be non-trivial. Say the action is \defbold{locally decomposable}\index{locally decomposable action}\index{action!locally decomposable} if it is weakly decomposable and if, for all $\alpha \in \mcA$, the product $\rist_U(\alpha)\rist_U(\alpha^\bot)$ has finite index in $U$.  Given an action of $G$ by homeomorphisms on a profinite space $\mfX$, we say the action is weakly or locally decomposable if the action of $G$ on the algebra of clopen subsets of $\mfX$ is weakly or locally decomposable respectively.

Say a Hecke pair $(G,U)$ is \defbold{faithful weakly decomposable}\index{faithful weakly decomposable}\index{Hecke pair!faithful weakly decomposable} if $\QZ(G)=1$ and $G$ admits a faithful weakly decomposable action on a Boolean algebra.  (We require that $\QZ(G)=1$ to avoid degenerate situations, such as Hecke pairs of the form $(G,1)$.)
\end{defn}

A property closed related to the weakly decomposable property has been considered by R.~M\"{o}ller and J.~Vonk (\cite{MollerVonk}), referred to in \cite{MollerVonk} as property H.  In particular, the authors give a criterion for obtaining non-discrete topologically simple \tdlc groups, of a kind that we will study further in \cite{CRW-Part2}.  As can easily be seen from the definitions, Property H and weakly decomposable are related as follows:

\begin{prop}Let $G$ be a group of automorphisms of a thick locally finite tree $T$, let $v$ be a vertex of $T$ and let $U = G_v$.  Then $(G,U)$ is a Hecke pair, and the action of $G$ on $T$ has property H in the sense of \cite{MollerVonk} if and only if the action of $G$ on the space of ends of $T$ is weakly decomposable.\end{prop}

Here are some basic properties of weakly decomposable actions. We state them for actions on profinite spaces; of course these can be translated via the Stone correspondence into corresponding statements for actions on Boolean algebras. 

\begin{lem}\label{lem:LatDecompBasic}
Let $(G,U)$ be a Hecke pair such that $G$ acts by homeomorphisms on a profinite space $\mfX$.  Assume that the action is weakly decomposable with kernel $K$. Then:

\begin{enumerate}[(i)]
\item $\mfX$ has no isolated point. 

\item If $\QZ(G/K)=1$, then for every non-empty clopen set $\alpha \subseteq \mfX$, we have $[\rist_G(\alpha)] > [K]$.

\item Suppose $\QZ(G/K)=1$.  Then for any two clopen subsets $\alpha,\beta$, we have $\alpha \neq \beta$ if and only if $[\rist_G(\alpha) ] \neq [\rist_G(\beta)]$. 
\end{enumerate}
\end{lem}

\begin{proof} 
There is no loss of generality in replacing $(G,U)$ by the Hecke pair $(G/K,UK/K)$, so we may assume that the action is faithful.  Since the action is weakly decomposable, it is clear that $|\mfX| > 1$.  

(i) 
If a point $x \in \mfX$ is isolated, its complement is a non-empty clopen set whose rigid stabiliser is trivial, contradicting the assumption that the action is weakly decomposable. 

(ii)
Let $\alpha \subseteq \mfX$ be a non-empty clopen subset. Then $G_\alpha$ contains a finite index subgroup $V$ of $U$, and contains $\rist_G(\alpha)$ as a normal subgroup; by assumption $\rist_G(\alpha)$ is non-trivial.  Since $\QZ(G)=1$, it follows from Lemma~\ref{lem:QZ:discrete} that $\rist_U(\alpha)$ is infinite, so $[\rist_U(\alpha)] > 0$.

(iii) 
Assume that $\alpha \neq \beta$. Upon swapping $\alpha$ and $\beta$, we may assume that $\alpha \not\le \beta$.  Then $\alpha \wedge \beta^\bot \neq 0$ so that $\rist_U(\alpha \wedge \beta^\bot)$ is an infinite bounded locally normal subgroup contained in $\rist_U(\alpha)$.  On the other hand, $\rist_U(\alpha \wedge \beta^\bot) \cap \rist_U(\beta) = 1$.  Hence $[\rist_G(\alpha)] \neq [\rist_G(\beta)]$.
\end{proof}

\subsection{Faithful weakly decomposable groups}

We now study the rigid stabilisers of the $G$-action on the centraliser and local decomposition lattices.  Of particular interest is the case of faithful actions of $G$.

In general, if a group $G$ acts on a Boolean algebra $\mcA$ with an invariant subalgebra $\mcA'$, then $\rist_G(\alpha)$ could be larger when defined in terms of the action on $\mcA'$ than it is when defined in terms of the action on $\mcA$.  However, this ambiguity does not arise in the case of faithful actions on subalgebras of the centraliser lattice: the rigid stabilisers turn out to be exactly the corresponding elements of the global centraliser lattice discussed earlier.

\begin{prop}\label{latrist}
Let $(G,U)$ be a locally C-stable Hecke pair and let $\mcA$ be a $G$-invariant subalgebra of $\lcent(G,U)$ on which $G$ acts faithfully.
Then $\QZ(G)=1$.  Moreover, the following hold. 
\begin{enumerate}[(i)]
\item For each  $\alpha \in \mcA$, we have $\rist_G(\alpha) = \QC^2_G(\alpha) = \CC_G(\QC_G(\alpha)) = \QC_G(\alpha^\bot)$.

\item  If $G$ is a topological group and $U$ is open in $G$, then the action of $G$ on $\mcA$ is smooth.

\item For each $\alpha \in \mcA$, we have  $\alpha = [\rist_G(\alpha)]$.

\item The action of $G$ on $\mcA$ is weakly decomposable, and if $\mcA$ is a subalgebra of $\ldlat(G,U)$ the action is locally decomposable.

\end{enumerate}
\end{prop}

\begin{proof}
Observe that for any Hecke pair $(G,U)$, the quasi-centre of $G$ fixes every element of $\lnorm(G,U)$ by conjugation.  Thus to have a faithful action on a subalgebra of $\lcent(G,U)$, we must have $\QZ(G)=1$.  Note also that $U$ has finite orbits on $\mcA$.

(i) We see that $\QC_G(\alpha^\bot) \le \rist_G(\alpha)$ and $\QC_G(\alpha) \le \rist_G(\alpha^\bot)$.  In particular, $\rist_G(\alpha)$ centralises $\QC_G(\alpha)$, so $\rist_G(\alpha) \le \CC_G(\QC_G(\alpha))$.  By Proposition~\ref{prop:globalcent}, in fact
$$\CC_G(\QC_G(\alpha)) = \QC^2_G(\alpha) = \QC_G(\alpha^\bot).$$
Hence $\rist_G(\alpha) = \CC_G(\QC_G(\alpha)) = \QC^2_G(\alpha)$.  This proves (i).

(ii) For each $\alpha \in \mcA$, then $\rist_G(\alpha)$ is closed by (i), so $G_\alpha = \N_G(\rist_G(\alpha))$ is closed.  Since moreover $|U:U_\alpha|$ is finite, we conclude that $U_\alpha$ is open in $U$, and hence $G_\alpha$ is open in $G$. 

(iii) Follows from (i) and the fact that $\alpha = \alpha^{\bot^2}$, see Theorem~\ref{centlat}(i).  

(iv) It is clear from (iii) that the action is weakly decomposable.  If $\mcA$ is a subalgebra of $\ldlat(G,U)$, note that for all $\alpha \in \mcA$ there are representatives $K$ of $\alpha$ and $L$ of $\alpha^\bot$ such that $KL$ has finite index in $U$.  Consequently $\rist_G(\alpha)\rist_G(\alpha^\bot)$ contains a finite index subgroup of $U$.
\end{proof}

\begin{cor}\label{cor:rist:support}Let $(G,U)$ be a locally C-stable Hecke pair and let $\mcA$ be a $G$-invariant subalgebra of $\lcent(G,U)$ on which $G$ acts faithfully.  Let $K$ be a locally normal subgroup of $G$, and let $\alpha = [K]^{\bot^2}$.  Suppose $\alpha \in \mcA$.  Then the set of points in $\mfS(\mcA)$ moved by $K$ forms an open dense subset of $\alpha$.
\end{cor}

\begin{proof}Let $\upsilon$ be the set of points in $\mfS(\mcA)$ moved by $K$.  We have $K \le \rist_G(\alpha)$ by Proposition~\ref{prop:globalcent} and Proposition~\ref{latrist}(i), which ensures that $\upsilon \subseteq \alpha$.  Suppose $\upsilon$ is not dense in $\alpha$.  Then $\upsilon$ is contained in a proper clopen subset $\beta$ of $\alpha$.  But then $K \le \rist_G(\beta)$, so $[K] \le [\rist_G(\beta)] = \beta < \alpha$, a contradiction.\end{proof}

A Hecke pair $(G,U)$ with trivial quasi-centre that admits a faithful weakly decomposable action is locally C-stable.  Moreover, the faithful locally or weakly decomposable actions are controlled by $\ldlat(G,U)$ and $\lcent(G,U)$ respectively.  

\begin{thm}\label{latdecomp}Let $(G,U)$ be a Hecke pair.
\begin{enumerate}[(i)]
\item Suppose that $G$ has a faithful weakly decomposable action on some Boolean algebra $\mcA$.  Then the following are equivalent:
\begin{enumerate}[(a)]
\item $\QZ(G)=1$;
\item The action of $U$ is weakly decomposable, and every non-trivial normal subgroup of $G$ has non-trivial intersection with $U$;
\item $G$ is locally C-stable, and $\mcA$ is $G$-equivariantly isomorphic to a subalgebra of $\lcent(G,U)$ (indeed, the set of rigid stabilisers of $\mcA$ form a subalgebra of the global centraliser lattice).
\end{enumerate}
\item Every $G$-equivalence class of faithful locally decomposable actions of $G$ (if there are any) occurs as the action of $G$ on a subalgebra of $\ldlat(G,U)$.
\end{enumerate}\end{thm}

We begin the proof of Theorem~\ref{latdecomp} with two lemmas.

\begin{lem}\label{lem:derived}
Let $A,N$ be subgroups of a group $G$. Assume that $N$ is normal, and contains an element $x \in N$ such that $[A, xAx\inv]=1$. Then $N$ contains $[A, A]$. 
\end{lem}

\begin{proof}
Let $a, b, c \in A$. Since $a$ and $b$ commute with $xcx\inv$, we have $[a, b] = [a, b xcx\inv]$. Setting $c = b\inv$, we infer that $[a, b] = [a, [b, x]]$, which belongs to $N$ since $N$ is normal. 
\end{proof}

\begin{lem}\label{wb:vabelian}Let $U$ be a residually finite group acting faithfully on an atomless Boolean algebra $\mcA$.  Let $K$ be a non-trivial normal subgroup of $U$ and let $\alpha \in \mcA \smallsetminus \{0\}$ be such that $g\alpha \wedge \alpha =0$ for some $g \in K$.  
\begin{enumerate}[(i)]
\item If $K$ is abelian, then $\rist_U(\alpha)$ is abelian.
\item If $K$ is finite, then $\rist_U(\alpha)$ is finite-by-abelian (and hence centre-by-finite).
\end{enumerate}
In either case, there exists $\beta \in \mcA$ such that $0 < \beta \le \alpha$ and $\rist_U(\beta)=1$, so the action of $U$ on $\mcA$ is not weakly decomposable.
\end{lem}

\begin{proof}
We see that $K \cap \rist_U(\alpha)$ and $K \cap \rist_U(g\alpha)$ are conjugate subgroups of $K$, and also that $\rist_U(\alpha) \cap \rist_U(g\alpha) = 1$ since $U$ acts faithfully on $\mcA$. Moreover, by Lemma~\ref{lem:derived}, $K \cap \rist_U(\alpha)$ contains $R = [\rist_U(\alpha),\rist_U(\alpha)]$.

Suppose $K$ is abelian.  Then the conjugation action of $K$ on its subgroups is trivial, so $K \cap \rist_U(\alpha) = K \cap \rist_U(g\alpha)$; since $\rist_U(\alpha) \cap \rist_U(g\alpha) = 1$, it follows that $K \cap \rist_U(\alpha) = \rist_K(\alpha) = 1$.  Thus $R = [\rist_U(\alpha),\rist_U(\alpha)]$ is trivial   and $\rist_U(\alpha)$ is abelian.

Suppose instead that $K$ is finite.    In particular, $R = [\rist_U(\alpha),\rist_U(\alpha)]$ is finite, so $\rist_U(\alpha)$ has a finite index subgroup $Z$ such that $Z \cap R = 1$.  We see that $[\rist_U(\alpha),Z] \le Z \cap R = 1$, so $Z$ is central in $\rist_U(\alpha)$ and hence $\rist_U(\alpha)$ is centre-by-finite.

Since $\mcA$ is atomless, there is a strictly descending sequence $(\alpha_i)_{i \in \bN}$ of non-zero elements of $\mcA$ such that $\alpha_0 = \alpha$.  If $\rist_U(\alpha)$ is finite, then there exist $i,j \in \bN$ with $i < j$ such that $\rist_U(\alpha_i) = \rist_U(\alpha_j)$, and thus $\rist_U(\beta)$ is trivial, where $\beta = \alpha_i \wedge \alpha^\bot_j$.  Thus we may suppose that $\rist_U(\alpha)$ is infinite.  Let $Z$ be a central subgroup of $\rist_U(\alpha)$ of finite index.  Since $U$ acts faithfully on $\mcA$, there must be some $\beta \in \mcA \smallsetminus \{0\}$ and $z \in Z$ such that $z\beta \wedge \beta = 0$; it follows that $\beta \le \alpha$, since $\rist_U(\alpha)$ fixes $\alpha^\bot$ pointwise as a subspace of $\mfS(\mcA)$.  Now $\rist_U(\beta)$ and $\rist_U(z\beta)$  are subgroups of $\rist_U(\alpha)$ that are $Z$-conjugate, so $\rist_U(\beta) = \rist_U(z\beta)$, and hence $\rist_U(\beta) = 1$ since $z\beta \wedge \beta = 0$.\end{proof}

\begin{proof}[Proof of Theorem \ref{latdecomp}](i) Let $\mcA$ be a Boolean algebra equipped with a faithful weakly decomposable action of $G$.  Then $\mcA$ is atomless by Lemma~\ref{lem:LatDecompBasic}.  We also see that $U$ is residually finite, since $U$ has finite orbits on $\mcA$.

Suppose $\QZ(G)=1$.  Then given a non-trivial subgroup $K$ of $G$ such that $|U:\N_U(K)|$ is finite, then $K \cap U$ is infinite.  In particular, $\rist_U(\alpha)$ is infinite for every $\alpha \in \mcA \smallsetminus \{0\}$, so the action of $V$ on $\mcA$ is weakly decomposable, for every subgroup $V$ of $G$ commensurate with $U$.  Hence (a) implies (b).  Moreover, by Lemma~\ref{wb:vabelian}, $V$ has no non-trivial finite or abelian normal subgroups.  Hence $G$ is locally C-stable, by Theorem~\ref{thm:locallyCstable}.

Let $\alpha \in \mcA$ and let $\beta = \alpha^\bot$.  Then $\rist_G(\alpha)$ and $\rist_G(\beta)$ commute.  Let $x \in \CC_G(\rist_G(\beta))$.  We claim that $x \in \rist_G(\alpha)$, in other words, the set $\upsilon$ of points moved by $x$ is a subset of $\alpha$.  If not, then $\upsilon \cap \beta$ contains some $\gamma \in \mcA$, since $\upsilon$ is an open subset of $\mfS(\mcA)$.  Indeed there is some $\delta \in \mcA$ such that $0 < \delta \le \gamma$ and $x\delta \not= \delta$.  The fact that the action is weakly decomposable ensures that $\rist_G(\delta)$ and $\rist_G(x\delta)$ are not locally equivalent by Lemma~\ref{lem:LatDecompBasic}(iii), so $x$ does not commensurate $\rist_G(\delta)$, and in particular $x \not\in \QC_G(\rist_G(\delta))$.  Since $\rist_G(\delta) \le \rist_G(\beta)$, we have a contradiction.  This shows that
\begin{equation}\label{eq:ristcent}
\rist_G(\alpha) = \CC_G(\rist_G(\alpha^\bot)).
\end{equation}
By Proposition~\ref{prop:globalcent}, $\rist_G(\alpha)$ is an element of the global centraliser lattice of $G$.  We therefore have a map 
$$\theta: \mcA \rightarrow \lcent(G,U); \; \theta(\alpha) = [\rist_G(\alpha)].$$
By Lemma~\ref{lem:LatDecompBasic}, $\theta$ is injective, and (\ref{eq:ristcent}) shows that $\theta$ is compatible with complementation; $\theta$ is also compatible with meets.  Hence $\theta$ is a homomorphism of Boolean algebras.

From the above argument, we conclude that (a) implies (b) and (c).

In the other direction, it is clear that (c) implies (a).  Suppose instead that (b) holds.  By Lemma~\ref{wb:vabelian}, $U$ has no non-trivial virtually abelian normal subgroups.  This implies that the quasi-centre of $U$ is trivial, in other words, $\QZ(G) \cap U = 1$.  Since every non-trivial normal subgroup of $G$ intersects $U$ non-trivially, we conclude that $\QZ(G)=1$.  Hence (b) implies (a).  This completes the proof of (i).

(ii) If the action of $G$ on $\mcA$ is locally decomposable, then $\rist_U(\alpha) \times \rist_U(\beta)$ has finite index in $U$, so $\rist_U(\alpha)$ is an almost direct factor of $U$ for all $\alpha \in \mcA$.  So in this case $\theta(\mcA)$ is contained in $\ldlat(G,U)$.\end{proof}

\begin{cor}\label{cor:weakdecomp:smooth}Let $G$ be a \tdlc group with $\QZ(G)=1$ and let $U$ be an open compact subgroup of $G$.  Then every faithful weakly decomposable action of $G$ (if there are any) on a profinite space is continuous with respect to the topology of uniform convergence.\end{cor}

\begin{proof}By Theorem~\ref{latdecomp}, $G$ is locally C-stable and by Proposition~\ref{latrist}, the action of $G$ on $\lcent(G)$ is smooth.  The conclusion now follows from Theorem~\ref{latdecomp} and Lemma~\ref{lem:SmoothActionContinuous}.\end{proof}

\begin{proof}[Proof of Theorem~\ref{intro:weakdecomp}](i) This follows from Proposition~\ref{latrist}.

(ii) By Theorem~\ref{latdecomp}~(i), there is a $G$-equivariant embedding $\rho$ of the algebra $\mcA$ of clopen subsets of $\mfX$ into $\lcent(G)$.  This embedding $\rho$ then corresponds, via the contravariant functor provided by the Stone representation theorem, to a $G$-equivariant quotient map $\rho^*$ from $\mfS(\lcent(G))$ to $\mfX$.\end{proof}

Another consequence of Theorem~\ref{latdecomp} is that if $(G,U)$ is locally C-stable and has a faithful action on its centraliser lattice, then both properties are inherited by subgroups with large normalisers.

\begin{prop}\label{prop:latdecomp:hered}Let $(G,U)$ be a faithful weakly decomposable Hecke pair.  Let $H$ be a locally normal subgroup of $G$.  Then the Hecke pair $(H,H \cap U)$ is locally C-stable and $H$ has a faithful weakly decomposable action on a principal ideal of $\lcent(G,U)$; this ideal is $H$-equivariantly isomorphic to a subalgebra of $\lcent(H,H \cap U)$.\end{prop}

\begin{proof}The existence of a faithful weakly decomposable action ensures that $U$ is residually finite.  By Proposition~\ref{latrist} and Theorem~\ref{latdecomp}, $(G,U)$ is locally C-stable and has a faithful weakly decomposable action on $\mcA=\lcent(G,U)$.  Let $\alpha = [H]^{\bot^2}$ and let $\mcI$ be the ideal of $\mcA$ generated by $\alpha$.  By Corollary~\ref{cor:rist:support}, the set of points in $\mfS(\mcA)$ moved by $H$ forms an open dense subset of $\alpha$.  Since $G$ acts faithfully on $\mcA$, it follows that $H$ acts faithfully on $\mcI$.  Since $\mcI$ is a principal ideal, it is a Boolean algebra in its own right.  We have $\QZ(H):= \QZ(H, H \cap U)=1$   by Theorem~\ref{thm:locallyCstable}.  Given $0 < \beta \in \mcI$, then $\rist_G(\beta) = \QC_G(\beta^\bot)$.  Since ${\beta^\bot > \alpha^{\bot^3} = \alpha^\bot}$, we see that $[\rist_G(\beta)] > [\QC_G(H)]$.  If $\rist_H(\beta)=1$, then $0 = [\rist_G(\beta) \cap H] = [\rist_G(\beta)] \wedge [H] = \beta \wedge [H]$. Hence $[H] \le \beta^\bot$ so that $\alpha \le \beta^{\bot^3} = \beta^\bot$. Since $\beta \le \alpha$, we obtain $\beta = 0$, a contradiction.  Thus the action of $H$ on $\mcI$ is weakly decomposable.  Hence by Theorem~\ref{latdecomp}, $H$ is locally C-stable and $\mcI$ is $H$-equivariantly isomorphic to a subalgebra of $\lcent(H,H \cap U)$.\end{proof} 

\begin{cor}\label{latdecomp:hered:subnormal}Let $(G,U)$ be a faithful weakly decomposable Hecke pair and let $H$ be a subgroup of $G$.  Suppose there is a sequence of subgroups
$$ H = H_1 \le H_2 \le \dots \le H_n = G$$
such that
$$|H_{i+1} \cap U: \N_{H_{i+1}}(H_i) \cap U|$$
is finite for $1 \le i < n$.  Then either $H=1$ or $(H,H \cap U)$ is faithful weakly decomposable.  In particular, $\QZ(H)=1$.\end{cor}

\begin{proof}Follows from Proposition~\ref{prop:latdecomp:hered} by induction.\end{proof}

\subsection{Topologically free actions}

For groups acting on topological spaces, there is a natural notion of actions that are in some sense the opposite of weakly decomposable:

\begin{defn}Let $G$ be a group acting faithfully on a topological space $X$.  The action is \defbold{locally faithful}\index{locally faithful}\index{action!locally faithful} at $x \in X$ if for every neighbourhood $O$ of $x$, the pointwise stabiliser of $O$ in $G$ is trivial.  The action is \defbold{topologically free}\index{topologically free}\index{action!topologically free} if it is locally faithful at every point in $X$.\end{defn}

Let $(G,U)$ be a Hecke pair and let $G$ act by homeomorphisms on a profinite space $\mfX$, such that $U$ has finite orbits on the set of clopen subsets of $\mfX$.  Clearly, if the action of $G$ on $\mfX$ is weakly decomposable, then it is not locally faithful at any point.  In general, the action of $(G,U)$ on $\mfX$ may be neither weakly decomposable nor locally faithful at any point, even if the action is minimal.  For example, given a (faithful, minimal) weakly decomposable action of $(G,U)$ on $\mfX$, then ${(G \times \bZ/2\bZ, U)}$ has a (faithful, minimal) action on the disjoint union of two copies of $\mfX$ that is not weakly decomposable, but has no locally faithful point: this is achieved by letting $G$ act identically on each copy of $\mfX$ and $\bZ/2\bZ$ swap the two copies of $\mfX$.  If $G$ admits a  non-inner automorphism $\phi$ of order~$2$, then one could construct a similar example of an action of $(H,U)$, where this time $H = G \rtimes \langle \phi \rangle$ has trivial quasi-centre with respect to $U$.  As a concrete example, let $T$ be a regular rooted tree, let $G$ be the derived subgroup of $\Aut(T)$,  and let $\mfX$ be the set of ends of $T$. Then $(G,G)$ is a Hecke pair; the $G$-action on $\mfX$ is faithful, minimal and weakly decomposable, and the $G$-orbit of any clopen subset is finite. The outer automorphism group $\Out(G)$ contains (and is in fact equal to)  $\Aut(T)/G$, which is an infinite direct product of cyclic groups of order~$2$. In particular $\Out(G)$ contains many involutions, which can be lifted to involutions in $\Aut(G)$.

However, with a stronger restriction on the group structure of $U$, we can ensure the existence of locally faithful points.  This property may be regarded as the antithesis of the faithful weakly decomposable property.  The following definition and proposition are inspired by \cite{Nekra}, which uses similar methods to comprehend the actions of free groups on rooted trees. 

\begin{defn}Let $U$ be group.  Say $U$ has the \defbold{virtually normal intersection property}\index{virtually normal intersection property}\index{VNIP@(VNIP)} (VNIP) if for any two non-trivial subgroups $K,L$ of $U$ such that $\N_U(K)$ and $\N_U(L)$ have finite index in $U$, the intersection $K \cap L$ is infinite.\end{defn}

If $(G,U)$ is locally C-stable, one sees that $U$ has VNIP if and only if $\lcent(G,U)$ is trivial. 

\begin{prop}
Let $(G,U)$ be a Hecke pair, let $\mfX$ be a profinite space on which $G$ acts faithfully by homeomorphisms, let $\mcA$ be the Boolean algebra of clopen subsets of $\mfX$, and suppose that $U$ has finite orbits on $\mcA$.  Suppose that $\QZ(G)=1$ and that $U$ has VNIP.  Then the set of points in $\mfX$ at which $G$ acts locally faithfully is non-empty and closed.  In particular, if the action of $G$ on $\mfX$ is minimal, then it is topologically free.
\end{prop}

\begin{proof}
Let $\mfY$ be the set of points at which $G$ acts locally faithfully.  We see from the definition that $\mfX \smallsetminus \mfY$ is open and hence $\mfY$ is closed, since if $G$ does not act locally faithfully at some $\mfp \in \mfX$, then there is an open neighbourhood $O$ of $\mfp$ such that the pointwise stabiliser of $O$ is non-trivial, and hence $G$ does not act locally faithfully at any point in $O$.

Suppose $\mfY=\emptyset$.  Then for each $\mfp \in \mfX$, we can find $\alpha_\mfp \in \mcA$ such that $\mfp \in \alpha_\mfp$ and $\rist_G(\alpha^\bot_\mfp)$ is non-trivial.  Since $U$ has finite orbits on $\mcA$, the normaliser $V = \N_U(\rist_G(\alpha^\bot_\mfp))$ has finite index in $U$; the fact that $\QZ(G)=1$ then forces $\rist_V(\alpha^\bot_\mfp)$ to be infinite.  Since every point in $\mfX$ is covered by some $\alpha_\mfp$, by compactness we have $\bigvee^n_{i=1}\alpha_{\mfp_i} = \infty$ for some $\mfp_1,\dots,\mfp_n \in \mfX$.  Since $U$ has VNIP, the intersection $K = \bigcap^n_{i=1}\rist_U(\alpha^\bot_{\mfp_i})$ is non-trivial.  But $\bigvee^n_{i=1}\alpha_{\mfp_i} = \infty$, in other words $\bigwedge^n_{i=1}\alpha^\bot_{\mfp_i} = 0$, so $K$ acts trivially on $\mcA$, contradicting the hypothesis that $G$ acts faithfully on $\mcA$.  Hence $\mfY$ must be non-empty.

Clearly $\mfY$ is $G$-invariant, so if the action of $G$ on $\mfX$ is minimal, then $\mfY = \mfX$.
\end{proof}

\section{Some radicals of locally compact groups}\label{radsec}

\subsection{The quasi-hypercentre}

In previous sections, we found it useful to assume that the groups under consideration have trivial quasi-centre.  Of course, there is no reason in general for a \tdlc group to have this property; more interesting, however, is the question of which locally compact groups $G$ have a non-trivial Hausdorff quotient $G/K$ such that $\QZ(G/K)=1$.

To address this issue, we need to introduce the following definitions. 

\begin{defn}
Let $G$ be a locally compact group.  A subgroup $K$ of $G$ is \defbold{quasi-hypercentral}\index{quasi-hypercentral} in $G$ if, whenever $N$ is a closed normal subgroup of $G$ such that $K \not\le N$, then $KN/N$ has non-trivial intersection with $\QZ(G/N)$.

The \defbold{quasi-hypercentre}\index{quasi-hypercentre} of  $G$, denoted by $\QZ^\infty(G)$, is the intersection of all closed normal subgroups $N$ of $G$ such that $\QZ(G/N)=1$.
\end{defn}

By Lemma~\ref{lem:QZ}, we have $\QZ(G) \leq \QZ^\infty(G)$. If moreover $\QZ(G)$ is discrete, then Lemma~\ref{lem:QZ} ensures that $\QZ(G/\QZ(G))=1$ so that $\QZ(G) = \QZ^\infty(G)$ in this case. 

The main properties of the quasi-hypercentre are collected in the following.

\begin{thm}\label{qzinf}
Let $G$ be a locally compact group. Then:
\begin{enumerate}[(i)]
\item $\QZ^\infty(G)$ is a closed characteristic subgroup. 

\item The quotient $G/\QZ^\infty(G)$ has trivial quasi-centre. In particular $\QZ^\infty(G)$ is the unique minimal closed normal subgroup of $G$ with this property. 

\item $\QZ^\infty(G)$ is  the unique largest quasi-hypercentral subgroup of $G$.
\end{enumerate}
\end{thm}

\begin{proof}[Proof of Theorem~\ref{qzinf}]
(i) is clear by the definition of the quasi-hypercentre. 

(ii)  Set $H = \QZ^\infty(G)$ and let $\mcN$ be the set of all closed normal subgroups $N$ of $G$ such that $\QZ(G/N)=1$.  Let $x \in G$, and suppose $xH \in \QZ(G/H)$.  For any  $N \in \mcN$, the group $G/N$ is a quotient of $G/H$, so that  $xN \in \QZ(G/N) $ by Lemma~\ref{lem:QZ}(iii). Since $\QZ(G/N)=1$ by the definition of $\mcN$, we infer that $x \in N$. Thus $x \in \bigcap_{N \in \mcN} N = H$. In other words $G/H$ has trivial quasi-centre, as desired. 

(iii) Since $\QZ(G/H)=1$ by part (ii), it follows that every quasi-hypercentral subgroup of $G$ is contained in $H$.  It remains to show that $H$ itself is quasi-hypercentral. Let $M$ be a closed normal subgroup of $G$ not containing $H$.  Then  $N = H \cap M$ does not contain $H$. Therefore, by the definition of $H$, the quasi-centre of $G/N$ is non-trivial and, hence,  there is some element $x$ of $G$ such that $x \not\in N$ and $xN \in \QZ(G/N)$.  Since $\QZ(G/H)=1$ by part (ii), in fact we have $x \in H \smallsetminus N = H \smallsetminus M$.  Since $N\leq M$, the group $G/M$ is a quotient of $G/N$, and we deduce from Lemma~\ref{lem:QZ}(iii) that $x$ is an element of $H$ whose image in $G/M$ is non-trivial and contained in  $\QZ(G/M)$.  We conclude that $H$ is quasi-hypercentral in $G$, as desired.
\end{proof}

Recall that the quasi-centre $\QZ(G)$ is characteristic but need not be closed. Theorem~\ref{qzinf} ensures that if $\QZ(G)$ is non-trivial, then $\QZ^\infty(G)$ is a non-trivial closed characteristic subgroup (that could be the whole of $G$).

\subsection{Relation with the quasi-hypercentre of subgroups}

We now specialise to \tdlc groups. The goal is to establish the following   stability properties of the quasi-hypercentre by passage to subgroups.

\begin{prop}\label{qzinfstab}
Let $G$ be a \tdlc group and $H$ be a subgroup of $G$.
\begin{enumerate}[(i)]
\item If $H$ is closed and normal in $G$, then $\QZ^\infty(H) \ge \QZ^\infty(G) \cap H$.
\item If $H$ is open in $G$, then $\QZ^\infty(H) = \QZ^\infty(G) \cap H$.
\end{enumerate}
\end{prop}

The proof of that proposition requires some preparations and will be given at the end of the subsection. We shall need additional definitions. 

Recall that given two normal subgroups $K \geq L$ of an abstract group $G$, the \textbf{centraliser of $K$ modulo $L$}\index{centraliser of $K$ modulo $L$} is defined by
$$
\CC_G(K/L) = \{g \in G \; | \; [g, K] \subset L\},
$$
\index{CG(K/L)@$\CC_G(K/L)$}
where we define $[u,g] = ugu^{-1}g^{-1}$.
The centraliser $\CC_G(K/L) $  is a normal subgroup of $G$, since it coincides with the preimage in $G$ of the centraliser $\CC_{G/L}(K/L)$. We shall need to quasify this notion as follows.  (Recall that $\mcB(G)$ denotes the set of open compact subgroups of $G$.)

\begin{defn}\label{cqdef}
Let $G$ be a \tdlc group and let $K$ and $L$ be closed subgroups.  The \defbold{quasi-centraliser  of $K$ modulo $L$}\index{quasi-centraliser!of $K$ modulo $L$} is defined by
$$\QC_G(K/L) = \{g  \in G \; | \; \text{there is } U \in \mcB(G) \text{ such that }
[g, U\cap K] \subset L.\}
$$\index{QCH(KL)@$\QC_H(K/L)$}
\end{defn}

The following shows that $\QC_G(K/L)$ is a normal subgroup of $G$ provided that $L$ is normal in $K$ and that $K$ and $L$ are both `quasi-normal' in $G$:

\begin{lem}\label{cqdeflem}Let $G$ be a \tdlc group and let $K$ and $L$ be non-trivial subgroups of $G$ such that $L \unlhd K$. Assume that $g Kg\inv \in [K]$ and $g Lg\inv \in [L]$ for all $g \in G$.    Then: 
\begin{enumerate}[(i)]
\item $\QC_G(K/L) = \QC_G(K'/L')$ for all $K' \in [K]$ and $L' \in [L]$. 

\item $\QC_G(K/L)$ is a non-trivial normal subgroup of $G$. 
\end{enumerate}
\end{lem}

\begin{proof}
(i) It is clear from the definition that $\QC_G(K/L)$ is invariant under replacing $K$ with a subgroup of $G$ that is locally equivalent to $K$.

Let $M = V \cap L$ for some open compact subgroup $V$ of $G$.  Certainly $\QC_G(K/M) \subseteq \QC_G(K/L)$.  On the other hand, given $g \in \QC_G(K/L)$ and an open compact subgroup $U$ of $G$ such that $[u,g] \in L$ for all $u \in U \cap K$, there is an open subgroup $W$ of $V$ such that $WgWg^{-1}$ is a subset of $V$.  Thus for all $w \in W \cap K$ we have $[w,g] \in V \cap L = M$, so $\QC_G(K/M) = \QC_G(K/L)$.  Hence $ \QC_G(K/L)$ is invariant under replacing $L$ with a subgroup of $G$ that is locally equivalent to $L$.  This proves (i).

(ii) We see that $L \le \QC_G(K/L)$, since $[K,L] \le L$, so $\QC_G(K/L)$ is non-trivial.

Let $g,h \in \QC_G(K/L)$ and let $U$ and $V$ be open compact subgroups of $G$ such that $[u,g] \in L$ for all $u \in U \cap K$ and $[v,h] \in L$ for all $v \in V \cap K$.  Then for $u \in U \cap V$ we have
\[ [u,g^{-1}h] = ug^{-1}hu^{-1}h^{-1}g = g^{-1}[g,u][u,h]g \in g^{-1}Lg.\]
Since $g^{-1}Lg \in [L]$ by hypothesis, we deduce from (i) that $g^{-1}h \in \QC_G(K/g^{-1}Lg) = \QC_G(K/L)$, so $\QC_G(K/L)$ is a subgroup of $G$. By definition, we have 
$$g \QC_G(K/L) g\inv= \QC_G (gKg\inv/gLg\inv)$$ 
for all $g \in G$. Therefore $\QC_G(K/L)$ is normal in $G$ by (i).
\end{proof}

We note the following facts separately as they will be used again later.

\begin{lem}\label{qzradlem}
Let $G$ be a \tdlc group.  Suppose that to every open compact subgroup $U$ we have assigned a closed normal subgroup $\R(U)$, such that the following statements hold:
\begin{enumerate}[(a)]
\item Given two open compact subgroups $U$ and $V$ of $G$, then $\R(U)$ is commensurate with $\R(V)$.  If $V = gUg\inv$ then $\R(V) = g\R(U)g\inv$.
\item Given an open compact subgroup $U$ of $G$, then the quotient $U/\R(U)$ has trivial quasi-centre.
\end{enumerate}
Let $U \in \mcB(G)$ and set $K = \QC_G(U/\R(U))$. Then:
\begin{enumerate}[(i)]
\item $K = \QC_G(V/\R(V))$ for any $V \in \mcB(G)$. 

\item $K \cap V = \R(V)$ for every  $V \in \mcB(G)$.

\item $K$ is a closed normal subgroup of $G$.

\item $\QZ(G/K)=1$.

\end{enumerate}
\end{lem}

\begin{proof}
(i) Condition (a) ensures that the set $\{ \R(U) \mid U \in \mcB(G)\}$ lies in a single commensurability class of subgroups of $G$, so that $\R(U)$ is a compact locally normal subgroup of $G$ such that $[\R(U)]$ does not depend on $U$.  It follows from Lemma~\ref{cqdeflem} that $K$ is normal in $G$ and does not depend on the choice of $U$.  Thus (i) holds.

(ii) By the definition of $K$, we see that $K \cap U$ consists of those elements of $U$ whose image in $U/\R(U)$ is quasi-central.  By condition (b), we thus have $K \cap U = \R(U)$. Hence (ii) follows from (i). 

(iii) We have already noted that $K$ is normal in $G$ as a consequence of Lemma~\ref{cqdeflem}. Since $K \cap U = \R(U)$ is closed in $U$ by (ii), it follows that $K$ is closed in $G$. 

(iv) Let $x \in G$ such that $xK \in \QZ(G/K)$; then $x$ centralises $VK/K$ for some open compact subgroup $V$ of $G$, but this ensures in turn that $x \in \QC_G(V/\R(V)) = K$ by (ii), so $xK$ is trivial.  Hence $\QZ(G/K)=1$.
\end{proof}

\begin{proof}[Proof of Proposition~\ref{qzinfstab}]

(i) We see that the quasi-hypercentre of a \tdlc group is a characteristic subgroup, so $\QZ^\infty(H)$ is normal in $G$.  Let $K = \CC_G(H/\QZ^\infty(H))$ be the centraliser of $H$ modulo $\QZ^\infty(H)$. Since $H$ and $\QZ^\infty(H)$ are closed normal subgroups of $G$, so is $K$. Moreover  $K \cap H = \QZ^\infty(H)$, since $H/\QZ^\infty(H)$ has trivial quasi-centre.  Consequently the quotient $G/K$ of $G$ has a normal subgroup $HK/K \cong H/\QZ^\infty(H)$ whose quasi-centre is trivial by  Theorem~\ref{qzinf}. 
Since 
$$\QZ(G/K) \cap HK/K \le \QZ(HK/K) = 1,$$
we infer that $\QZ(G/K)$ is a normal subgroup of $G/K$ with trivial intersection with $HK/K$. Therefore $HK/K$ and $\QZ(G/K)$ commute. On the other hand the definition of $K$ implies that $HK/K$ has trivial centraliser in $G/K$. It follows that $\QZ(G/K)=1$.  Therefore $K \ge \QZ^\infty(G)$, and thus $\QZ^\infty(H) = K \cap H \ge \QZ^\infty(G) \cap H$.

(ii) Let $Q = \QZ^\infty(G)$.  Note that $H/(Q \cap H) \cong HQ/Q$, and $HQ/Q$ is an open subgroup of $G/Q$; thus $\QZ(HQ/Q) \le \QZ(G/Q) = 1$, so $H/(Q \cap H)$ has trivial quasi-centre.  By Theorem~\ref{qzinf}, we conclude that $Q \cap H \ge \QZ^\infty(H)$.

Combining this with part (i), we see that $\QZ^\infty(O) = \QZ^\infty(G) \cap O$ in the case that $O$ is an open normal subgroup of $G$.  

Let now $V \leq U$ be compact open subgroups of $G$. Let $O$ be the core of $V$ in $U$, so that $O$ is an open normal subgroup of both $V$ and $U$. By the preceding paragraph, we have  
$$ \QZ^\infty(U) \cap O = \QZ^\infty(O) =\QZ^\infty(V) \cap O.$$ 
In particular $ \QZ^\infty(U)$ and $\QZ^\infty(V)$ are commensurate. This implies that for all $U, V \in \mcB(G)$, the groups $ \QZ^\infty(U)$ and $\QZ^\infty(V)$ are commensurate. 

Define $\R(U) := \QZ^\infty(U)$ for every profinite group $U$. We have just checked that condition (a) from Lemma~\ref{qzradlem} is satisfied, and so is condition (b) by Theorem~\ref{qzinf}. 
Hence, given $V \in \mcB(G)$, the group $K = \QC_G(V/\QZ^\infty(V))$ is a closed normal subgroup of $G$ that does not depend on the choice of $V$.  Moreover Lemma~\ref{qzradlem} ensures that $\QZ(G/K) = 1$, so $K \ge Q$. Also $K \cap V = \QZ^\infty(V) \leq Q \cap V$, where the latter containment was observed in the first paragraph above. It follows  that $Q$ contains an open subgroup of $K$, so $K/Q$ is a discrete normal subgroup of $G/Q$; since $\QZ(G/Q)=1$, this forces $K=Q$ by Lemma~\ref{lem:QZ}.

This shows that $\QZ^\infty(G) =  \QC_G(V/\QZ^\infty(V))$ for any open compact subgroup $V$ of $G$. In particular, if $H$ is any open subgroup of $G$, we may replace $V$ by $V \cap H$ and deduce that 
$$
\QZ^\infty(G) \cap H = \QC_G(V/\QZ^\infty(V)) \cap H =\QC_H(V/\QZ^\infty(V))   = \QZ^\infty(H).
$$
This confirms (ii). 
\end{proof}

\begin{rem}
We have shown that the local equivalence class of the quasi-hyper\-centre is unchanged by passing to an open subgroup.  In particular for any open subgroup $H$ of $G$ and any $V \in \mcB(H)$, the group $\QZ^\infty(H)$ is locally equivalent to $\QZ^\infty(V)$, so we have
$$\QZ^\infty(G)= \QC_G(V/\QZ^\infty(V)) = \QC_G(H/\QZ^\infty(H)).$$
\end{rem}

\begin{rem}In general, the quasi-hypercentre of an arbitrary closed subgroup of a \tdlc group bears no relation to that of the group itself.  For instance, any profinite group $H$ can be embedded as a closed subgroup of the Cartesian product $G$ of its finite continuous images, in which case $\QZ^\infty(G) = G$ since $\QZ(G)$ is dense in $G$.  On the other hand $H$ can be embedded in a semidirect product $K \rtimes H$ (where $K$ is an infinite Cartesian product of copies of the free profinite group on $2$ generators, for instance, and $H$ permutes the terms of this product) of such a form that $\QZ^\infty(K \rtimes H) = \QZ(K \rtimes H) = 1$.\end{rem}

The following  question is natural. 

\begin{que}
If $H$ is a closed normal subgroup of $G$ contained in $\QZ^\infty(G)$, does it follow that $H$ is quasi-hypercentral in $G$?
\end{que}

We can at least say that $H$ is quasi-hypercentral in itself; in other words, given any proper closed normal subgroup $K$ of $H$, then $\QZ(H/K)$ is non-trivial.

\subsection{Other radicals}

We have seen that locally compact groups admit a unique largest quotient with trivial quasi-centre.  In a similar manner, for some classes of profinite group there is a unique largest quotient that has no non-trivial compact locally normal subgroups in the given class.  In particular, a motivation for eliminating virtually abelian locally normal subgroups in this way is the centraliser lattice described in Section~\ref{clatsect}.

\begin{defn}Let $\mcC$ be a class of profinite groups.  Say $\mcC$ is \defbold{stable}\index{stable class of profinite groups} if the following conditions hold:
\begin{enumerate}[(a)]
\item $\mcC$ contains all abelian profinite groups, as well as all finite simple groups;

\item Given a group $U \in \mcC$ and a closed normal subgroup $K$ of $U$, then $K \in \mcC$ and $U/K \in \mcC$;

\item Given a profinite group $U$ that is a (possibly non-direct) product of finitely many closed normal $\mcC$-subgroups, then   $U \in \mcC$;

\item $\mcC$ is stable under isomorphisms, i.e. it contains the whole isomorphism class of each of its elements.
\end{enumerate}

Given a profinite group $U$, write $\mcC(U)$ for the closed subgroup generated by all normal $\mcC$-subgroups of $U$.

Given a stable class $\mcC$, a \tdlc group $G$ is \defbold{$\mcC$-semisimple}\index{C-semisimple@$\mcC$-semisimple} if $\QZ(G)=1$ and the only locally normal subgroup of $G$ belonging to $\mcC$ is the identity subgroup.  For a profinite group $U$, a subgroup $K$ of $U$  is called \defbold{$\mcC$-regular} in $U$ if 
for every closed normal subgroup $L$ of $U$ not containing  $K$, the image of $K$ in $U/L$ contains a non-trivial locally normal $\mcC$-subgroup of $U/L$.  For a \tdlc group $G$ and closed subgroup $H$, say $H$ is \defbold{$\mcC$-regular}\index{C-regular@$\mcC$-regular} in $G$ if $U \cap H$ is $\mcC$-regular in $U$ for all open compact subgroups $U$ of $G$.

Write $\VA$\index{A@$\VA$}\index{A-regular@$\VA$-regular}\index{A-semisimple@$\VA$-semisimple} for the smallest stable class; note that by Fitting's Theorem, all groups in this class are virtually nilpotent.  However, the class of $\VA$-regular \tdlc groups is larger than just the locally $\VA$ groups; for example, all \tdlc groups with virtually soluble open compact subgroups are $\VA$-regular.  In a recent paper (\cite[Theorem 1.11]{Wes}), P. Wesolek has shown that every second-countable $\VA$-regular \tdlc group can be constructed from profinite groups and countable discrete groups by a (possibly transfinite) sequence of extensions and countable increasing unions. On the other hand, we shall see in Proposition~\ref{cstab} below that a \tdlc group is $\VA$-semisimple if and only if it is locally C-stable with trivial quasi-centre.
\end{defn}

Another example of a stable class (within the class of profinite groups) is the class of profinite groups that are topologically generated by pronilpotent normal subgroups and quasisimple subnormal subgroups; for this class $\mcC$, then $\mcC(U)$ is the \defbold{generalised pro-Fitting subgroup}\index{generalised pro-Fitting subgroup} of $U$.  The regular and semisimple classes associated to this class are discussed in \cite{ReiF}.

Here is the main theorem for this section.

\begin{thm}\label{radtdlc}
Let $G$ be a \tdlc group and let $\mcC$ be a stable class of topological groups.  Then $G$ has a uniquely defined closed normal subgroup $\R_\mcC(G)$\index{RC(G)@$\R_\mcC(G)$}\index{RA(G)@$\R_\VA(G)$}, the \defbold{$\mcC$-regular radical}\index{C-regular@$\mcC$-regular!radical}\index{radical!$\mcC$-regular} of $G$, that is characterised by either of the following properties:
\begin{enumerate}[(i)]
\item $\R_\mcC(G)$ is the unique largest  locally normal subgroup of $G$ that is $\mcC$-regular in $G$;
\item $G/\R_\mcC(G)$ is $\mcC$-semisimple, and given any closed normal subgroup $N$ of $G$ such that $G/N$ is $\mcC$-semisimple, then $N \ge \R_\mcC(G)$.\end{enumerate}
\end{thm}

The proof will be given at the end of the section. The following classical fact is well known. 

\begin{lem}[Dietzmann's Lemma]\index{Dietzmann's Lemma}\label{lem:Dietzmann} 
Let $G$ be a group and $\Sigma$ be a finite subset consisting of torsion elements. If $\Sigma$ is invariant under conjugation, then $\langle \Sigma \rangle$ is finite. 
\end{lem}

\begin{proof}
See \cite{Dietzmann} or \cite[p. 154]{Kurosh}. 
\end{proof}

For the remainder of this section, $\mcC$ is a stable class of profinite groups.  We need to check that the local and global definitions of $\mcC$-regularity are not in conflict; this will be ensured by part (v) of the following lemma.

\begin{lem}\label{ssnorm}Let $U$ be a profinite group.
\begin{enumerate}[(i)]
\item Suppose that $\mcC(U)=1$.  Then $U$ is $\mcC$-semisimple.

\item Let $\mcN$ be the set of all closed normal subgroups $N$ of $U$ such that $U/N$ is $\mcC$-semisimple, and define $K = \R(U)$ to be the intersection of $\mcN$.  Then $U/K$ is $\mcC$-semisimple.

\item Let $N$ be a non-trivial closed normal subgroup of $U$ such that $N$ does not contain any non-trivial closed normal $\mcC$-subgroups of $U$.  Then $U/\CC_U(N)$ is a non-trivial $\mcC$-semisimple image of $U$.

\item Let $V \ge K \ge L$ be closed normal subgroups of $U$ and suppose $K$ is $\mcC$-regular in $U$.  Then $L$ is $\mcC$-regular in $U$ and $K$ is $\mcC$-regular in $V$. 

\item Let $K$ be a closed normal subgroup of $U$, let $V$ be an open subgroup of $U$ and let $L \leq K$ be a closed normal subgroup of $U$ that has finite index in $K$.  Then the following are equivalent: (a) $K$ is $\mcC$-regular in $U$; (b) $L$ is $\mcC$-regular in $U$; (c) $K \cap V$ is $\mcC$-regular in $V$.\end{enumerate}\end{lem}

\begin{proof}(i) Let $H$ be a non-trivial finite locally normal subgroup of $U$.  Then by Dietzmann's Lemma, $H$ is contained in a finite non-trivial normal subgroup $K$ of $U$; there is then a characteristic subgroup $M$ of $K$ that is a direct power of a finite simple group, so that $M$ is a normal $\mcC$-subgroup of $U$.  Thus the condition $\mcC(U)=1$ ensures that every non-trivial locally normal subgroup of $U$ is infinite.

Let $H$ be an infinite locally normal $\mcC$-subgroup of $U$, and let $V$ be a closed normal subgroup of $U$ of finite index that is contained in $N_U(H)$.  By intersecting with $V$, we may assume $H \le V$.  Let $K$ be the smallest closed normal subgroup of $U$ containing $H$.  Then $K$ is topologically generated by the set of conjugates of $H$ in $U$, which is a finite set of normal subgroups of $K$ (since $K \le V$), so $K$ is a normal $\mcC$-subgroup of $U$, contradicting the assumption that $\mcC(U)=1$.

It remains to show that $\QZ(U)=1$.  If not, then there is some open normal subgroup $V$ of $U$ such that $\CC_U(V) > 1$.   Then $\CC_U(V)$ is normal in $U$, hence infinite, so that  $\Z(V) = \CC_U(V) \cap V$ is infinite as well.  But then $\Z(V)$ is a non-trivial abelian normal subgroup of $U$, so $1 < \Z(V) \le \mcC(U)$, a contradiction.

(ii) We may assume $K=1$. By (i) it suffices to show that $\mcC(U)=1$.  Let $H$ be a non-trivial closed normal subgroup of $U$ such that $H \in \mcC$.  Then there is some $N \in \mcN$ such that $H \not\le N$.  But then the image of $H$ in $U/N$ is a non-trivial normal $\mcC$-subgroup, contradicting the hypothesis that $U/N$ is $\mcC$-semisimple.

(iii) Let $K = \CC_U(N)$.  We must have $\Z(N)=1$, since otherwise $\Z(N)$ would be a non-trivial normal $\mcC$-subgroup of $U$; thus $K \cap N = 1$.  As a consequence, $U/K$ has a normal subgroup $NK/K$ isomorphic to $N$, and in particular $U/K \not= 1$; moreover, $\CC_{U/K}(NK/K)=1$.  Let $H/K$ be a non-trivial normal $\mcC$-subgroup of $U/K$.  Then $H/K \cap NK/K$ is non-trivial since $NK/K$ has trivial centraliser.  Since $H > K$, we then have $L > 1$ where $L = H \cap N$.  Now $L$ is isomorphic to $(H \cap NK)/K$, so $L$ is a normal $\mcC$-subgroup of $U$ contained in $N$, a contradiction.  Thus no such subgroup $H/K$ of $U/K$ exists, so $\mcC(U/K)=1$; by part (i), it follows that $U/K$ is $\mcC$-semisimple.

(iv) If $L$ is not $\mcC$-regular in $U$, then there is some closed normal subgroup $M$ of $U$ not containing $L$ such that $LM/M$ does not contain any non-trivial closed normal $\mcC$-subgroups of $U/M$.  It follows by part (iii) that there is a closed normal subgroup $R$ of $U$ containing $M$ but not  $L$ such that $U/R$ is $\mcC$-semisimple.  But then every $\mcC$-regular closed normal subgroup of $U$ must be contained in $R$, so $K$ is not $\mcC$-regular in $U$. This confirms that $L$ is $\mcC$-regular in $U$ as soon as $K$ is so. 

Let $N$ be a closed normal $\mcC$-subgroup of $V$ that does not contain $K$, and let $R$ be the core of $N$ in $U$.  Then $KR/R$ contains a non-trivial locally normal $\mcC$-subgroup $S/R$ of $U/R$.  Since $S > R$, there is some $u \in U$ such that $S \not \leq  uNu^{-1}$.  Since $S \leq KR \unlhd U$, we infer that $u\inv S  u \leq KR \unlhd V$. Thus $u^{-1}SuN/N$ is a non-trivial locally normal $\mcC$-subgroup of $V/N$ that is contained in $KN/N$. This confirms that $K$ is $\mcC$-regular in $V$. 

(v) We have already seen in (iv) that (a) implies (b).  In the other direction, suppose that (b) holds and let $M$ be a closed normal subgroup of $U$ that does not contain $K$.  If $KM/M$ is finite, then it certainly contains a non-trivial normal $\mcC$-subgroup of $U/M$.  If $KM/M$ is infinite, then $LM/M$ is also infinite and so contains a non-trivial normal $\mcC$-subgroup of $U/M$.  Thus (a) and (b) are equivalent.

Given the equivalence of (a) and (b), from now on we may assume $K \le V$.  Suppose $K$ is $\mcC$-regular in one of $U$ and $V$.  Then $K$ is $\mcC$-regular in the core $W$ of $V$ in $U$ by (iv).  Let $Y$ be either $U$ or $V$.  Given a closed normal $\mcC$-subgroup $R$ of $Y$ not containing $K$, then $R \cap W$ is a normal $\mcC$-subgroup of $W$ not containing $K$. It follows that $K(R \cap W)/(R \cap W)$ contains a non-trivial locally normal $\mcC$-subgroup $S/(R \cap W)$ of $W/(R \cap W)$, so $KR/R$ contains the non-trivial locally normal $\mcC$-subgroup $SR/R$ of $W/R$; note that $SR/R$ is also locally normal in $Y/R$.  Thus $K$ is $\mcC$-regular in $Y$, so (a) and (c) are equivalent.\end{proof}

 \begin{proof}[Proof of Theorem~\ref{radtdlc}]Assume for the moment that $G$ is profinite.
 
Let $\mcN$ be the set of all closed normal subgroups of $G$ such that $G/N$ is $\mcC$-semisimple, and let $K$ be the intersection of $\mcN$.  Then $K$ satisfies property (ii) of the theorem by Lemma~\ref{ssnorm} (ii).  It is clear from the construction that $K$ contains every subgroup of $G$ that is $\mcC$-regular in $G$.

Let $N$ be a closed normal subgroup of $G$ not containing $K$.  If $KN/N$ is abelian, then it is a locally normal $\mcC$-subgroup of $G/N$.  If $KN/N$ is non-abelian, let $L = \CC_G(KN/N)$ be the centraliser of $KN$ modulo $N$. Note that $L$ does not contain $K$, so by the definition of $K$, the quotient $G/L$ cannot be $\mcC$-semisimple.  It then follows by Lemma~\ref{ssnorm} (iii) that $KN/N$ contains a non-trivial closed normal $\mcC$-subgroup of $G/N$.  Thus $K$ is $\mcC$-regular in $G$.

We have proved the theorem in the profinite case.  Now let $G$ be an arbitrary \tdlc group.  We have already obtained a $\mcC$-regular radical for the compact subgroups of $G$.  Moreover, the $\mcC$-regular radicals of the open compact subgroups satisfy the hypotheses of Lemma~\ref{qzradlem}: this is clear from Lemma~\ref{ssnorm} (v) and the fact that the $\mcC$-regular radical is the unique largest $\mcC$-regular normal subgroup.  Let $U$ be an open compact subgroup of $G$ and define $R = \R_\mcC(G) := \QC_G(U/\R_\mcC(U))$.  Then by Lemma~\ref{qzradlem}, $R$ is a closed normal subgroup of $G$ that does not depend on the choice of $U$, such that $\QZ(G/R)=1$ and $R \cap U = \R_\mcC(U)$.

The open compact subgroups of $G/R$ are of the form $UR/R$ for $U$ an open compact subgroup of $G$; $UR/R$ is isomorphic to $U/U \cap R = U/\R_\mcC(U)$, so $UR/R$ is $\mcC$-semisimple.  Hence $G/R$ is $\mcC$-semisimple; in particular $\R_\mcC(G/R)=1$.

Let $N$ be a closed normal subgroup of $G$ such that $G/N$ is $\mcC$-semisimple.  Then by the profinite case of the theorem, we have $N \ge \R_\mcC(U)$ for every open compact subgroup $U$ of $G$; this ensures that $RN/N$ is a discrete normal subgroup of $G/N$.  But $G/N$ has trivial quasi-centre, so all discrete normal subgroups are trivial.  Thus $N \ge R$.  We have now shown that property (ii) holds, which is clearly a characterisation of the $\mcC$-regular radical.

Given an open compact subgroup $V$ of $G$, then the $V \cap R$ is $\mcC$-regular in $V$ by the profinite case of the theorem.  Hence $R$ is $\mcC$-regular in $G$.  The fact that $G/R$ is $\mcC$-semisimple then ensures that $R$ is the largest $\mcC$-regular  locally normal subgroup of $G$:  given $K \le G$ such that $K \not\le R$ and $K$ is locally normal in $G$, then $KR/R$ is non-discrete (since $\QZ(G/R)=1$), so $KR/R \cap UR/R$ is non-trivial for every open compact subgroup $U$ of $G$, but clearly $KR/R$ does not contain any non-trivial locally normal $\mcC$-subgroup of $UR/R$.
\end{proof}

Theorem~\ref{intro:radical} now follows immediately from Theorems \ref{qzinf} and \ref{radtdlc}.

Theorem~\ref{radtdlc} also has the following consequence:

\begin{cor}Let $G$ be a \tdlc group.  Then at least one of the following holds:
\begin{enumerate}[(i)]
\item $G$ is $\VA$-regular;
\item there is a closed characteristic normal subgroup $R$ of $G$ such that $G/R$ is non-discrete and $\VA$-semisimple.\end{enumerate}\end{cor}

In particular, consider a topologically characteristically simple \tdlc group $G$.  Then $G$ is either $\VA$-regular or $\VA$-semisimple.

As in the case of the quasi-hypercentre (see Proposition~\ref{qzinfstab}), the $\mcC$-regular radical enjoys   stability properties with respect to certain closed subgroups.

\begin{prop}\label{radtdlcstab}Let $G$ be a \tdlc group and let $\mcC$ be a stable class of profinite groups.
\begin{enumerate}[(i)]
\item Let $H$ be a closed normal subgroup of $G$.  Then $\R_\mcC(H) \ge \R_\mcC(G) \cap H$.
\item Let $H$ be an open subgroup of $G$.  Then $\R_\mcC(H) = \R_\mcC(G) \cap H$.
\end{enumerate}
\end{prop}

\begin{proof}(i) Let $U$ be an open compact subgroup of $G$.  Then $\R_\mcC(U)$ is $\mcC$-regular in $U$, so $\R_\mcC(U) \cap H$ is $\mcC$-regular in $U \cap H$ by Lemma~\ref{ssnorm} (iv).  Thus $\R_\mcC(U) \cap H \le \R_\mcC(U \cap H)$.  At the global level, we have $\QC_H(G/(\R_\mcC(U) \cap H)) = \QC_H(G/\R_\mcC(U))$ since $H$ is normal in $G$, so $[g,h] \in H$ whenever $g \in G$ and $h \in H$.  Hence we have
$$
\begin{array}{rcl}
 \R_\mcC(H) & = & \QC_H(G/\R_\mcC(U \cap H)) \\
 &\ge & \QC_H(G/(\R_\mcC(U) \cap H)) \\
 &= & \QC_H(G/\R_\mcC(U))\\
 & =& \R_\mcC(G) \cap H.
 \end{array}
 $$

(ii) We can obtain both $\R_\mcC(H)$ and $\R_\mcC(G) \cap H$ as $\QC_H(U/\R_\mcC(U))$, where $U$ is an open compact subgroup of $G$ that is contained in $H$.\end{proof}

We finish this section by linking it with our previous discussion of locally C-stable \tdlc groups.

Note the following variant of Lemma~\ref{bewlem}, giving a condition under which a locally normal subgroup has trivial centraliser in $G$.

\begin{lem}\label{bewlarge}
Let $(G,U)$ be a Hecke pair such that $\QZ(G) = 1$ and let $L$ be a locally normal subgroup of $G$ such that $\CC_U(L)$ is finite.  Then:
\begin{enumerate}[(i)]
\item $\CC_G(L) = 1$.
\item Let $H$ and $K$ be subgroups of $G$ containing $L$, such that $H \cap K$ contains a finite index subgroup $V$ of $U$.  Suppose $\theta: H \rightarrow K$ is an isomorphism such that $\theta(x) = x$ for all $x \in L$ and such that $\theta(V)$ is commensurate with $V$.  Then $H = K$ and $\theta = \mathrm{id}_H$.
\end{enumerate}
\end{lem}

\begin{proof}Observe that $\CC_U(L)$ is finite and normalised by a finite index subgroup of $U$, so in fact $\CC_U(L) \le \QZ(U) = 1$.

We now prove (ii). 
We may assume $V \le \N_G(L)$.  Choose a finite index subgroup $W$ of $V$ such that $W\theta(W)$ is contained in $V$.  Let $g \in W$ and let $x \in L$.  Then
\[ gxg^{-1} = \theta(gxg^{-1}) = \theta(g)\theta(x)\theta(g^{-1}) = \theta(g)x\theta(g^{-1}),\]
so $g^{-1}\theta(g)$ centralises $x$. This shows that  $g^{-1}\theta(g) \in \CC_V(L) = 1$.  Thus $g = \theta(g)$, that is, $\theta$ fixes $W$ pointwise.

We may now apply Lemma~\ref{bewlem}, since $\CC_G(W \cap g^{-1}Wg) \le \QZ(G) =1$ for all $g \in G$.  Thus $H = K$ and $\theta = \mathrm{id}_H$. Part (ii) follows.

By applying (ii) to the automorphisms of $G$ induced by conjugation by elements of $\CC_G(L)$, we infer that $\CC_G(L) \le \Z(G) = 1$, thereby confirming (i).
\end{proof}

\begin{prop}\label{cstab}
Let $G$ be a \tdlc group such that $\QZ(G)=1$.  The following are equivalent:
\begin{enumerate}[(i)]
\item $G$ is locally C-stable;
\item For every compact locally normal subgroup $H$ of $G$ and open compact subgroup $U$ of $G$ we have $\CC_G(H\CC_U(H))=1$;
\item $G$ has no non-trivial abelian locally normal subgroups;
\item $G$ is $\VA$-semisimple.
\end{enumerate}
\end{prop}

\begin{proof}Suppose $G$ is locally C-stable.  Then (ii) follows  from Lemma~\ref{bewlarge}(i).

Suppose (ii) holds and let $H$ be an abelian locally normal subgroup of $G$.  Let $U$ be any open compact subgroup of $G$ and let $K = H \cap U$.  Then $K \le \CC_G(K\CC_U(K)) = 1$, so $H$ is discrete.  Since $\QZ(G)=1$, we conclude that $H$ is trivial.

Suppose (iii) holds.  If $G$ is not $\VA$-semisimple, then there is an open compact subgroup $U$ of $G$ such that $U$ has a non-trivial normal $\VA$-subgroup $H$, by Lemma~\ref{ssnorm}(i).  But we see by hypothesis that $U$ has no non-trivial finite or abelian locally normal subgroups, so by the construction of $\VA$, there is no way to obtain a normal $\VA$-subgroup.  Thus $G$ is $\VA$-semisimple.

Now suppose that $G$ is $\VA$-semisimple.  In particular, $\QZ(G)=1$ and there are no non-trivial compact abelian locally normal subgroups.  Then $G$ is locally C-stable by Theorem~\ref{thm:locallyCstable}.
\end{proof}

The following illustration of the above discussion also provides a justification for the choice of terminology. \index{p-adic Lie group@$p$-adic Lie group}

\begin{prop}\label{prop:Lieradical}
Let $G$ be a $p$-adic Lie group with Lie algebra $\mathfrak g = \mathrm{Lie}(G)$. Let also $R = \R_{[A]}(G)$ be the the $[A]$-regular radical of $G$. 

Then $\mathrm{Lie}(R)$ coincides with   the solvable radical  $\mathfrak r$ of $\mathfrak g$, and $R$ is the largest closed normal subgroup of $G$ with Lie algebra $\mathfrak r$. It coincides with the kernel of the $G$ action on $\mathfrak g/ \mathfrak r$. 

In particular $R$ is locally soluble, and $G$ is $\VA$-semisimple if and only if its Lie algebra  $\mathfrak g$ is semisimple.
\end{prop}

\begin{proof}
{By Proposition~\ref{prop:p-adicLieCorrespondence}, every ideal of $\mathfrak g$ is the Lie algebra of a compact locally normal subgroup of $G$.} In particular $G$ has a compact locally normal subgroup $K$ whose Lie algebra is the soluble radical $\mathfrak r$ of $\mathfrak g$. By \cite[Ch.~III, \S9, Cor. to Prop.~6]{Bbki}, the group $K$ has a soluble open subgroup, say $K_0$. By Lemma~\ref{lem:LN:basic} there is a compact open subgroup $U$ of $G$ containing $K_0$ as a normal subgroup. 

By Theorem~\ref{radtdlc} and Proposition~\ref{radtdlcstab}(ii), we have $K_0 = \mathrm{Lie}(K_0) \le R_{[A]}(U) = R \cap U$.

Since $U/K_0$ is a compact $p$-adic Lie group whose Lie algebra $\mathfrak g/ \mathfrak r$ is a direct product of simple (non-abelian) Lie algebras, it follows from  Proposition~\ref{prop:p-adicLieCorrespondence} that $U$ is virtually a direct product of hereditarily just finite profinite groups that are not virtually abelian. In particular $U$ has trivial quasi-centre (see Proposition~5.1 in \cite{BEW}) and no non-trivial abelian locally normal subgroup. Therefore $U$ is $[A]$-semisimple by Proposition~\ref{cstab}, so that $K_0 = \mathrm{Lie}(K_0) \le R_{[A]}(U) = R \cap U$. In particular $K_0$ is an open subgroup of $R$ that is soluble, and we have  $\mathfrak r =  \mathrm{Lie}(K_0)=  \mathrm{Lie}(R)$.  

Any closed normal subgroup of $G$ with Lie algebra $\mathfrak r$ is locally solvable by \cite[Ch.~III, \S9, Cor. to Prop.~6]{Bbki}, hence $[A]$-regular, and thus contained in $R$ by Theorem~\ref{radtdlc}. Thus $R$ is the largest such subgroup. It follows that $R$ is the kernel of the $G$-action on $\mathfrak g/ \mathfrak r$ (a short and complete argument showing this can be found in the proof of Lemma~6.4 in \cite{CCLTV}).

The remaining assertions are now clear.
\end{proof}

The $\VA$-regular radical has been used recently to analyse the global structure of \tdlc group admitting an open subgroup which has a continuous faithful finite-dimensional representation over a local field: by \cite[Theorem~R.1]{CStu}, the $\VA$-regular radical of such a group is the largest closed normal subgroup that is locally soluble (i.e. that has a soluble open subgroup). 

\subsection{Canonical completions of Hecke pairs}

The radical theory developed so far in this section is specific to \tdlc groups, as the topology plays an important role.  However, results about \tdlc groups naturally have consequences for Hecke pairs in general, due to the existence of a universal `completion' for Hecke pairs that generalises the well-known profinite completion of a group.  A construction is given in \cite{Bel} and a proof of universality in \cite{ReiPL}.

\begin{thm}\label{tdlctop:mainthm}Let $(G,U)$ be a Hecke pair, such that $G$ is a topological group (possibly discrete) and $U$ is closed in $G$.  Then there is an essentially unique group homomorphism $\theta: G \rightarrow \hat{G}_U$, the \defbold{localised profinite completion!of a Hecke pair of $G$ at $U$}\index{localised profinite completion}, where $\hat{G}_U$\index{GhatU@$\hat G_U$} is a topological group, such that the following properties hold:
\begin{enumerate}[(i)]
\item The image of $\theta$ is dense, the restriction $\theta|_U$ of $\theta$ to $U$ is continuous, and $\overline{\theta(U)}$ is an open profinite subgroup of $\hat{G}_U$.
\item Let $R$ be a topological group and let $\phi: G \rightarrow R$ be a group homomorphism.  Suppose that $\phi|_U$ is continuous and $\overline{\phi(U)}$ is profinite.  Then there is a unique continuous homomorphism $\psi: \hat{G}_U \rightarrow R$ such that $\phi = \psi\theta$.\end{enumerate}\end{thm}

By composing with the appropriate quotient map, we also have canonical homomorphisms with related universal properties for \tdlc groups with trivial quasi-centre, respectively locally C-stable \tdlc groups with trivial quasi-centre, that follow immediately from the results we have so far.

\begin{cor}Let $(G,U)$ be a Hecke pair, such that $G$ is a topological group (possibly discrete) and $U$ is closed in $G$.  Let
$$\theta_1: G \rightarrow \hat{G}_U/\QZ^\infty(\hat{G}_U) \text{ and } \theta_2: G \rightarrow \hat{G}_U/\R_\VA(\hat{G}_U)$$
be the canonical homomorphisms obtained from Theorem~\ref{tdlctop:mainthm}, Theorem~\ref{qzinf} and Theorem~\ref{radtdlc}.  Set $G_1 =\hat{G}_U/\QZ^\infty(\hat{G}_U)$ and $G_2 = \hat{G}_U/\R_\VA(\hat{G}_U)$. Then the homomorphisms $\theta_i$ are characterised by the following properties:
\begin{enumerate}[(i)]
\item The image of $\theta_i$ is dense, the restriction $\theta_i|_U$ of $\theta_i$ to $U$ is continuous, and $\overline{\theta_i(U)}$ is an open profinite subgroup of $G_i$.  Moreover, $\QZ(G_1)=\QZ(G_2) = 1$, and $G_2$ is locally C-stable.
\item Let $R$ be a \tdlc group and let $\phi: G \rightarrow R$ be a group homomorphism.  Suppose that $\phi|_U$ is continuous and $\overline{\phi(U)}$ is profinite.  If $\QZ(R)=1$, then there is a unique continuous homomorphism $\psi_1: G_1 \rightarrow R$ such that $\phi = \psi_1\theta_1$.  If in addition $R$ is locally C-stable, then there is a unique continuous homomorphism $\psi_2: G_2 \rightarrow R$ such that $\phi = \psi_2\theta_2$.
\end{enumerate}
\end{cor}

Let $(G,U)$ be a Hecke pair, equipped with the discrete topology, such that $U$ is residually finite.  Two questions arise at this point:

\begin{que}\label{que:qzembed}If $(G,U)$ has trivial quasi-centre, is $\theta_1: G \rightarrow \hat{G}_U/\QZ^\infty(\hat{G}_U)$ necessarily injective?  Equivalently, does $(G,U)$ necessarily admit a dense embedding into a \tdlc group with trivial quasi-centre?\end{que}

\begin{que}\label{que:cstabembed}If $(G,U)$ is locally C-stable with trivial quasi-centre, is $\theta_2: G \rightarrow \hat{G}_U/\R_\VA(\hat{G}_U)$ necessarily injective?  Equivalently, does $(G,U)$ necessarily admit a dense embedding into a locally C-stable \tdlc group with trivial quasi-centre?\end{que}

Both questions are of a local nature: if $(G,U)$ has trivial quasi-centre, then by Lemma~\ref{lem:QZ:discrete}, every non-trivial normal subgroup of $G$ has infinite intersection with $U$.  Thus $\theta_i$ is injective for $i\in \{1,2\}$ if and only if $\theta_i|_U$ is injective.  The questions can thus be restated as follows:

\begin{que}Let $U$ be a residually finite group with trivial quasi-centre.  Can $U$ be densely embedded in a profinite group $P$, such that $P$ has trivial quasi-centre?\end{que}

\begin{que}Let $U$ be a residually finite group with no non-trivial virtually abelian normal subgroups.  Can $U$ be densely embedded in a profinite group $P$, such that $P$ has no non-trivial virtually abelian normal subgroups?\end{que}

\section{Restrictions on the group topology}

\subsection{The structure lattice and refinements of the group topology}\label{subsec:refine}

In this section, we investigate the structural significance of fixed points in the action of a \tdlc group $G$ on its structure lattice.  Of course, if $G$ has a closed normal subgroup $K$ that is neither discrete nor open, then $[K]$ is a non-trivial fixed point in $\lnorm(G)$.  Especially interesting, however, is the question of whether a compactly generated, topologically simple \tdlc group can have non-trivial fixed points in $\lnorm(G)$. We  shall return to this question in the second paper \cite{CRW-Part2}, elaborating on the  results of the present chapter.  Most results we shall obtain here  are   for \tdlc groups that are first-countable or $\sigma$-compact; it is known (see \cite{KK}) that a compactly generated (or even $\sigma$-compactly generated) \tdlc group without non-trivial compact normal subgroups is second-countable (and hence also first-countable and $\sigma$-compact).

Our first goal is to prove Theorem~\ref{intro:refine}, giving a general connection between the structure lattice and group topologies for first-countable \tdlc groups.

\begin{defn}
Let $G$ be a \tdlc group and let $\mu$ be a commensurability class (or subset of a commensurability class) of compact subgroups of $G$.  The \defbold{localisation}  of $G$ at $\mu$ is the \tdlc group $H$, unique up to isomorphism, that admits a continuous embedding $\phi:  H \rightarrow G$, such that for some (any) $K \in \mu$, $\phi(H) = \Comm_G(K)$ and $\phi^{-1}(K)$ is an open compact subgroup of $H$. This group  topology on $H$ is called the \textbf{localised topology}\index{localised topology} at $\mu$; it is   denoted by $\mathcal T_{(\mu)}$\index{Tmu@$\mathcal T_{(\mu)}$}. 
\end{defn}

Thus, as an abstract group, the localisation of $G$ at $\mu$ coincides with $\Comm_G(K)$ so that the localisation of $G$ at $\mu$ is nothing but $(\Comm_G(K), \mathcal T_{(\mu)})$. 


Starting with a \tdlc group $G$, one sees that the \tdlc groups with the same group structure as $G$ but a finer topology are precisely the groups of the form $(G, \mcT_{(\mu)})$, where $\mu$ is a $G$-invariant commensurability class of compact subgroups of $G$.  If $G$ is first-countable, we can in fact classify these topological refinements of $G$ using the structure lattice: this is the purpose of Theorem~\ref{intro:refine}.

First we recall a result of Caprace--Monod \cite{CM}.  Recall that a \tdlc group $G$ is \defbold{residually discrete}\index{residually discrete} if the intersection of all open normal subgroups of $G$ is trivial.

\begin{prop}[\cite{CM} Corollary 4.1]\label{capmon}Let $G$ be a compactly generated \tdlc group.  Then the following are equivalent:
\begin{enumerate}[(i)]
\item $G$ is residually discrete;
\item $G$ has a base of neighbourhoods of the identity consisting of subsets that are invariant under conjugation;
\item $G$ has a base of neighbourhoods of the identity consisting of open compact normal subgroups.\end{enumerate}\end{prop}

\begin{cor}\label{resuni}Let $G$ be a residually discrete locally compact group and let $x \in G$.  Then the open subgroups of $G$ normalised by $x$ form a base of neighbourhoods of the identity.\end{cor}

\begin{lem}\label{profcomm}
Let $U$ be a first-countable profinite group and let $K$ be a closed subgroup of $U$.  Then $U$ commensurates $K$ if and only if $U$ normalises an open subgroup of $K$.
\end{lem}

\begin{proof}Certainly, if $U$ normalises an open subgroup of $K$ then $U$ commensurates $K$.  Conversely, suppose $U$ commensurates $K$.  Let $\mu$ be the commensurability class of $K$.  Then $U_{(\mu)}$ is residually discrete, so by Corollary~\ref{resuni}
$$U = \bigcup \{ \N_U(L) \mid L \le_o K\}.$$
Now $K$ is first-countable, so it has countably many open subgroups; thus we have expressed $U$ as a union of countably many subgroups, each of which is closed for the given  topology on $U$.  By the Baire Category Theorem, at least one of these, say $\N_U(L)$, must be open in $U$.  Thus $L$ has finitely many conjugates in $U$, and we have $U = \Comm_U(L)$ since $L$ is commensurate with $U$.  Hence $\bigcap_{u \in U}uLu^{-1}$ is a closed subgroup of $L$ of finite index and thus an open subgroup of $K$.\end{proof}

\begin{proof}[Proof of Theorem~\ref{intro:refine}]Let $\mcT$ be the given topology of $G$ and let $\mcT'$ be a refinement of $\mcT$ compatible with the group structure.  By van Dantzig's theorem, under $\mcT'$ there is a subgroup $K$ of $G$ that is open and compact; it follows that $K$ is compact under $\mcT$.  Moreover, the set of such subgroups forms a commensurability class $\mu$ that is invariant under conjugation in $G$.  In particular, given a subgroup $U$ of $G$ that is open and compact under $\mcT$, then by Lemma~\ref{profcomm}, $U$ normalises a subgroup $L \in \mu$.  Thus $\mu$ corresponds to an element $\alpha = [L]$ of $\lnorm(G)$; we have $G = G_\alpha$ since $\mu$ is $G$-invariant, and see that the topology $\mcT'$ is exactly the topology  $\mcT_{(\alpha)}$.

Conversely, let $\alpha \in \lnorm(G)$ that is fixed by $G$.  Then $G$ admits a locally compact topological refinement by declaring those representatives of $\alpha$ that are closed in $\mcT$ to be open, in other words, the group $(G, \mcT_{(\alpha)})$, which by definition has the same underlying group structure as $G$.\end{proof}

Note that none of the group topologies in Theorem~\ref{intro:refine} are $\sigma$-compact, except possibly the original topology of $G$, since a closed subgroup of a profinite group is either open or of uncountable index.

\subsection{A group topology derived from the centraliser lattice}

Let $\mcA$ be a Boolean algebra, and equip $\Aut(\mcA)$ with the coarsest group topology such that the stabiliser of $\alpha$ is open, for all $\alpha \in \mcA$.  Then $\Aut(\mcA)$ is a totally disconnected group.  Given a group $G$ acting faithfully on $\mcA$, write $\overline{G}_{\mcA}$ for the closure of $G$ in $\Aut(\mcA)$, equipped with the subspace topology.

Of particular interest is the case in which an infinite commensurated subgroup $U$ of $G$ (perhaps $G$ itself) has finite orbits.  In this case, we see that $U$ is residually finite and $\overline{U}_\mcA$ is profinite, so the group $\langle G, \overline{U}_\mcA \rangle$ can be equipped with a \tdlc topology that extends the topology of $\overline{U}_\mcA$.

\begin{thm}\label{thm:lc_completion}Let $(G,U)$ be a Hecke pair with trivial quasi-centre and let $\mcA$ be a Boolean algebra on which $G$ has a faithful, weakly decomposable action.  Let $E$ be the subgroup of $\Aut(\mcA)$ generated by $G$ and $\overline{U}_\mcA$.  Then $E$ commensurates $\overline{U}_\mcA$, so that $E$ can be equipped with a \tdlc group topology by extending the topology of $\overline{U}_\mcA$ (so that $\overline{U}_\mcA$ is embedded as an open subgroup).  Moreover, the following holds:

\begin{enumerate}[(i)]
\item As a \tdlc group, $E$ is locally C-stable and faithful weakly decomposable.
\item As a locally compact topological group, $E$ is characterised as follows:

Given a locally compact group $L$ and an injective homomorphism $\theta: G \rightarrow L$ with dense image such that $\overline{\theta(U)}$ is an open subgroup of $L$, then there exists a unique continuous surjective homomorphism $\phi: L \rightarrow E$ such that the composition $\phi \theta$ is the embedding of $G$ in $E$.
\item If $|G:U|$ is countable, the following characterisation also applies:

Given a topological group $L$ that is a Baire space and an injective homomorphism $\theta: G \rightarrow L$ such that $L = \langle \theta(G), \overline{\theta(U)} \rangle$, then there exists a unique continuous homomorphism $\phi: L \rightarrow E$ such that the composition $\phi \theta$ is the embedding of $G$ in $E$.
\end{enumerate}
\end{thm}

\begin{proof}Let $V = \overline{U}_\mcA$.  Given $g \in G$, we see that $gUg\inv$ contains a finite index subgroup of $U$, which ensures that $gVg\inv$ contains an open subgroup of $V$.  Thus $E$ can be equipped with a \tdlc group topology by extending the topology of $\overline{U}_\mcA$.

Certainly $E$ acts faithfully on $\mcA$, and the action of $E$ on $\mcA$ is weakly decomposable with respect to $V$.  Given $x \in \QZ(E)$, then $x$ centralises a finite index subgroup of $U$, and hence $x\rist_U(\alpha)x\inv$ is commensurate with $\rist_U(\alpha)$, for all $\alpha \in \mcA$.  This implies that $x$ acts trivially on $\mcA$, so $x = 1$.  Thus $E$ is locally C-stable by Theorem~\ref{latdecomp}.

Let us suppose we have a homomorphism $\theta: G \rightarrow L$ as in either (ii) or (iii), with $(G,U)$ and $\mcA$ satisfying the necessary hypotheses.  We will prove (ii) and (iii) together, except where it is necessary to divide into cases.

Given $\alpha \in \mcA$, let $R_\alpha = \theta(\rist_G(\alpha))$ and let $K_\alpha = \CC_L(R_{\alpha^\bot})$.  Then $K_\alpha$ is a closed subgroup of $L$, for all $\alpha \in \mcA$.  Moreover, we have $R_\alpha \le K_\alpha$, but $K_\alpha \cap R_{\alpha^\bot}$ is trivial.  This ensures that $K_\alpha \cap \theta(G) = R_\alpha$; in particular, given $\alpha,\beta \in \mcA$ such that $\alpha \neq \beta$, then $R_\alpha \neq R_\beta$ since the action of $G$ is weakly decomposable, so $K_\alpha \neq K_\beta$.  Now set $N_\alpha = \N_L(K_\alpha)$.  Then $N_\alpha$ is closed in $L$, since $K_\alpha$ is closed.  Moreover, $N_\alpha \cap \theta(G) = \N_{\theta(G)}(R_\alpha) = \theta(\N_G(\rist_G(\alpha)))$.  We see that $N_\alpha \cap \theta(U)$ has finite index in $\theta(U)$, so $N_\alpha \cap P$ has finite index in $P$, where $P = \overline{\theta(U)}$.

In case (ii), it follows that $N_\alpha \cap P$ is open in $P$, and hence $N_\alpha$ is open in $L$, since $P$ is open in $L$.  In case (iii), $N_\alpha$ is a closed subgroup of $L$ of countable index, and hence open by the Baire category theorem.  In either case, we conclude that $N_\alpha$ is open in $L$, as is any finite intersection of subgroups $N_\alpha$.  Therefore the set $\{ K_\alpha \mid \alpha \in \mcA\}$ ordered by inclusion forms a Boolean algebra $\mcA'$ on which $L$ acts smoothly, so that $\theta(U)$ is orbit-equivalent to $P$ on the finite subsets of $\mcA'$.  The natural map $K_\alpha \mapsto \alpha$ induces a continuous homomorphism $\psi: L \rightarrow \Aut(\mcA)$ such that $\psi\theta$ is the given action of $G$ on $\mcA$.  In case (ii), $\psi(P)$ has a dense subgroup in common with $\overline{U}_\mcA$, and both groups are compact, so $\psi(P) = \overline{U}_\mcA$.  In case (iii), $U_\mcA$ is dense in $\psi(P)$, so $\psi(P) \le \overline{U}_\mcA$.  In either case, $\psi(L)$ is generated by $\psi\theta(G)$ and $\psi(P)$, so $\psi(L) \le E$, and hence $\psi$ restricts to a continuous homomorphism $\phi: L \rightarrow E$ with the required properties (in particular, $\phi$ is surjective in case (ii)).  Moreover, in both cases the required properties specify $\phi$ uniquely, since the condition on $\phi \theta$ ensures that $\phi$ is uniquely specified on a dense subset of $L$.
\end{proof}

We see from Theorem~\ref{thm:lc_completion}(i) that the embedding $G \rightarrow E$ is not dependent on the choice of $\mcA$, only on $G$ and the commensurability class of $U$.  Thus if $(G,U)$ is any faithful weakly decomposable Hecke pair, we may define $E$ to be the \defbold{centraliser lattice completion}\index{centraliser lattice completion} ($\lcent$-completion) of $(G,U)$, denoted $\hat{G}_{\lcent}$.  For example, if $G$ is a branch group, regarded as a Hecke pair $(G,G)$, then $\hat{G}_{\lcent}$ is the \defbold{congruence completion}\index{congruence completion} of $G$ in the sense of \cite{BSZ}.  

Theorem~\ref{thm:lc_completion} is of interest even if $(G,U)$ is \tdlc, as it has implications for the topologies that are compatible with the group structure of $G$.  In this case, $G = \hat{G}_{\lcent}$ as topological groups.  Notice that if $G$ is second-countable, then $|G:U|$ is countable.

In particular, we can prove a slightly more general form of Theorem~\ref{intro:rigid}.

\begin{cor}\label{refpropcent}Let $G$ be a \tdlc group that is faithful weakly decomposable.  Let $L$ be a topological group and let $\theta:L \rightarrow G$ be an abstract group isomorphism, such that one of the following conditions holds:
\begin{enumerate}[(i)]
\item $L$ is a \tdlc group, and there exists an open compact subgroup $U$ of $L$ such that $\theta(U)$ is open and compact in $G$.
\item $G$ is second-countable and the topology of $L$ is a Baire space.
\end{enumerate}

Then $\theta$ is continuous.

In particular, if $G$ is second-countable, then the topology of $G$ is the unique $\sigma$-compact locally compact topology compatible with the group structure (in particular, it is preserved by every automorphism of $G$ as an abstract group), and the set of all locally compact group topologies of $G$ is in natural bijection with the set of fixed points of $G$ acting on $\lnorm(G)$ by conjugation.\end{cor}

\begin{proof}By applying Theorem~\ref{thm:lc_completion} to the group isomorphism $\theta^{-1}: G \rightarrow L$, we see that there exists a continuous group homomorphism $\phi$ such that $\phi\theta^{-1} = \mathrm{id}_G$.  Clearly $\phi = \theta$.

In particular, if $G$ is second-countable, then the topology of $G$ is the coarsest group topology on $G$ that forms a Baire space, so in particular it is the coarsest locally compact Hausdorff group topology on $G$.  The remaining assertions follow from Theorem~\ref{intro:refine}.\end{proof}

\subsection{Subgroups abstractly generated by compact locally normal subgroups}

Let $G$ be a locally compact group. Recall that the group $\Aut(G)$ of homeomorphic automorphisms carries a natural Hausdorff group topology, called the \textbf{Braconnier topology}\index{Braconnnier topology} (see Appendix~I in \cite{CM}). Basic identity neighbourhoods in $\Aut(G)$ are provided by the sets
$$
\mathfrak A(K, U) = \big\{ \alpha \in \Aut(G) \; | \; \forall x \in K, \ \alpha(x) \in Ux \text{ and } \alpha\inv(x) \in Ux\big\},
$$
where $K$ runs over compact subsets of $G$ and $U$ runs over identity neighbourhoods of $G$. In particular, for each compact open subgroup $U < G$, the stabiliser $\{ \alpha \in \Aut(G) \; | \; \alpha(U) = U\}$ is an identity neighbourhood in $\Aut(G)$.

In what follows, automorphism groups of locally compact groups will always be implicitly endowed with the Braconnier topology.

\begin{lem}\label{lem:Brac:ContinuityCriterion}
Let $G$ be a \tdlc group, let $I$ be a directed set, and let $(U_n)_{n \in I}$ be a net of open compact subgroups in $G$ forming a base of identity neighbourhoods, such that $U_i < U_j$ whenever $i > j$.
Let $(\alpha_n)_{n \in I}$ be a net in $\Aut(G)$ such that $\lim_{n \in I} \alpha_n(x) = x$ for all $x \in G$. 

If for each $m \in I$, there is some $n_0 \in I$ such that $\alpha_n(U_m) = U_m$ for all $n \geq n_0$,  then $(\alpha_n)_{n \in I}$  converges  to the identity with respect to the Braconnier topology. 
\end{lem}

\begin{proof}
Let $K \subset G$ be a compact subset and $U < G$ an identity neighbourhood. Choose $U_m \subset U$. We can find a finite subset $\{x_1, \dots, x_k\} \subset K$ such that $K \subset \bigcup_{i=1}^k U_m x_i$. Since $(\alpha_n)$ converges pointwise to the identity, so does the sequence $(\alpha_n\inv)$. Hence we may find $n_0$ such that $\alpha_n(x_i) \in U_m x_i$ and $\alpha_n\inv(x_i) \in U_m x_i$ for all $n \geq n_0$. 

Let now $x \in K$. Choose $i$ such that $U_m x = U_m x_i$. Thus we have 
$$
\alpha_n(x) \in \alpha_n(U_m x) =  \alpha_n(U_m x_i) = U_m \alpha_n(x_i) = U_m x_i = U_m x.
$$
Similarly, we get $\alpha_n\inv(x) \in U_m x$. 
This shows that $\alpha_n$ belongs to $\mathfrak A(K, U_m) \subset \mathfrak A(K, U)$ for all $n \geq n_0$. Thus $(\alpha_n)$ converges to the identity with respect to the Braconnier topology. 
\end{proof}

\begin{defn}Let $G$ be a \tdlc group and let $\alpha \in \lnorm(G)$.  Define the \defbold{compact locally normal closure}\index{compact locally normal closure} $\LN(G,\alpha)$\index{LN(G,Valpha)@$\LN(G,\alpha)$} of $\alpha$ in $G$ to be the group generated by all compact locally normal representatives of $\alpha$.\end{defn}

We remark that $\LN(G,\alpha)$ is a (not necessarily closed) subgroup of $G$ with open normaliser, and if $\alpha$ is fixed by $G$ then $\LN(G,\alpha)$ is normal in $G$.  If $G$ is locally C-stable with trivial quasi-centre and $\alpha \in \lcent(G)$, then it follows from Proposition~\ref{prop:globalcent} that $\LN(G,\alpha)$ is an open subgroup of $\QC_G(\alpha^\bot)$, which in turn is a closed subgroup of $G_\alpha$, and the normaliser of $\LN(G,\alpha)$ is exactly $G_\alpha$.

\begin{prop}\label{prop:ContinuousBrac}
Let $G$ be a \tdlc group and let $\alpha$ be a fixed point of $G$ in $\lnorm(G)$.  Let $D= \LN(G,\alpha)$. 
Then the natural homomorphism $G \to \Aut(D, \mcT_{(\alpha)})$ induced by the $G$-action by conjugation is continuous.
\end{prop}

\begin{proof}
Since $G$ commensurates the open compact subgroups of $(D, \mcT_{(\alpha)})$, it follows that the homomorphism  $\varphi \colon G \to \Aut(D, \mcT_{(\alpha)})$ induced by the conjugation action indeed takes its values in the group of homeomorphic automorphisms of $(D, \mcT_{(\alpha)})$.  Let $(g_n)_{n \in I} \subset G$ be a net converging to the identity in $G$. 

Let us first show that $\varphi(g_n)$ converges pointwise to the trivial automorphism. Let $x \in D$. Let $L_1, \dots, L_k$ be {compact locally normal representatives of $\alpha$} such that $x \in L_1 L_2 \dots L_k$. Then the group $\bigcap_{i=1}^k \N_G(L_i)$ is open in $G$, and therefore contains $g_n$ for all sufficiently large $n$. In particular, for such a large $n$, we have $\varphi(g_n).x = g_n x g_n\inv \in L_1 L_2 \dots L_k$. The latter product being compact with respect to the localised topology $\mcT_{(\alpha)}$, it follows that the net of commutators $(g_n x g_n\inv x\inv)_{n \in I}$ is bounded in $(D, \mcT_{(\alpha)})$.  Therefore, all we need to show is that $1$ is its only accumulation point.

Suppose for a contradiction that some subnet $([g_{f(n)} , x])_{n \in I'}$ converges  to $z \neq 1$. Choose a compact locally normal representative $L$ of $\alpha$ which is sufficiently small so that $z \not \in L$. We have  $[g_{f(n)} , x] \in Lz$ for all large $n$. In particular $[g_{f(n)}, x]  \in Oz$, where $O$ is any open compact subgroup of $G$ containing $L$. Since $z \not\in L$, we may choose $O$ small enough so that it does not contain $z$. Then $O$ and $Oz$ are disjoint so that  $[g_{f(n)}, x]  \not \in O$ for all large $n$.  Since $O$ is an open subgroup of $G$, we deduce that $([g_{f(n)}, x])_{n \in I'}$ does not converge to the identity in $G$, which is absurd. 

Let now $L$ be a compact locally normal representative of $\alpha$ and let $(U_n)_{n \in I}$ be as in Lemma~\ref{lem:Brac:ContinuityCriterion}; we may assume $U_n \le \N_G(L)$ for all $n \in I$. Then the groups $L_m = L \cap U_m$ form a basis of identity neighbourhoods in $(D, \mcT_{(\alpha)})$. Clearly, for each $m \geq 0$, we have $g_n \in \N_G(L_m)$ for all sufficiently large $n$. Lemma~\ref{lem:Brac:ContinuityCriterion} now implies that $\varphi(g_n) $ converges to the identity with respect to the Braconnier topology.
\end{proof}

\begin{cor}\label{cor:Brac:ContinuousBis}
Let $G$ be a \tdlc group, and {let $\alpha$ be a fixed point of $G$ in $\lnorm(G)$.  Let $D$ be an abstract normal subgroup of $G$ generated by compact locally normal representatives of $\alpha$, and let $H = \Gamma.D$ where $\Gamma$ is a discrete normal subgroup of $G$.}  
Then the natural homomorphism $G \to \Aut(H, \mcT_{(\alpha)})$ induced by the $G$-action by conjugation is continuous.
\end{cor}

\begin{proof}
Note that $\Gamma$ is a discrete normal subgroup of $(H, \mcT_{(\alpha)})$, and we have $H  = \Gamma. D $.  Note also that $D $ is an open subgroup of $\LN(G,\alpha)$ with respect to the localised topology $\mcT_{(\alpha)}$, so the Braconnier topology on $\Aut(D, \mcT_{(\alpha)})$ is compatible with the Braconnier topology on $\Aut(\LN(G,\alpha), \mcT_{(\alpha)})$.

Let $(g_n)$ be a net tending to the identity in $G$; let $K \subset H$ be $\mcT_{(\alpha)}$-compact and $V \leq   H$ be a $\mcT_{(\alpha)}$-compact open subgroup. We need to show that $[g_n, K] \cup [g_n\inv, K] $ is contained in $V$ for all sufficiently large $n$. 
Since $H = \Gamma. D $, it suffices to deal with the case when $K = \gamma C$ with $\gamma \in \Gamma$ and $C \subset D$ is $\mcT_{(\alpha)}$-compact. For $n$ large enough, we have $[g_n, \gamma] = 1$ since $\Gamma$ is contained in the quasi-center of $G$. Therefore $[g_n, \gamma C] = \gamma [g_n, C] \gamma\inv$. 
Observe that $W = \gamma\inv V \gamma$  is a compact open subgroup of $(H, \mcT_{(\alpha)})$. By Proposition~\ref{prop:ContinuousBrac}, we have $[g_n, C] \subset W$ for all sufficiently large $n$. Hence $[g_n, \gamma C] \subset V$ for all large $n$. The argument with $g_n\inv$ instead of $g_n$ is similar. 
\end{proof}

In general, a compact subgroup of a \tdlc group that is commensurate with a locally normal subgroup need not itself be locally normal.  For example, if $G$ has trivial quasi-centre, then the only finite locally normal subgroup is the trivial one, but $G$ may still have non-trivial finite subgroups.  Proposition~\ref{prop:ContinuousBrac} leads to an interesting criterion for compact subgroups to be locally normal.

\begin{prop}\label{prop:Brac:commensurate}
Let $G$ be a \tdlc group and let $\alpha \in \lnorm(G)$.  Let $L$ be a compact subgroup of $G$ such that $[L] = \alpha$.  Then $L$ is locally normal in $G$ if and only if $L \le \LN(G,\alpha)$.
\end{prop}

\begin{proof}
If $L$ is locally normal, then $L \le \LN(G,\alpha)$ by definition. 

Assume conversely that $L \le \LN(G,\alpha)$. Since $G_\alpha$ is open, it follows that a subgroup of $G$ contained in $G_\alpha$ is locally normal in $G$ if and only if it is locally normal in $G_\alpha$. Moreover, every locally normal representative of $\alpha$ is contained in $G_\alpha$. Thus we have $\LN(G,\alpha) = \LN(G_\alpha, \alpha)$, and there is no loss of generality in assuming that $G = G_\alpha$.  

Let $D$ be the group $\LN(G,\alpha)$ equipped with the localised topology $\mcT_{(\alpha)}$.  Note that $L$ is compact with respect to  $\mcT_{(\alpha)}$.  Thus  $L$ is an open compact subgroup of $D$.   In particular the collection $\mathfrak A(L, L)$ of automorphisms of $D$ which stabilise $L$ is an identity neighbourhood with respect to the Braconnier topology on $\Aut(D)$. Since the natural map $G \to \Aut(D)$ is continuous by Proposition~\ref{prop:ContinuousBrac}, there must be an open subgroup  of $G$ that normalises $L$.
\end{proof}

If $U$ is an open compact subgroup of $G$ that fixes $\alpha$, there is a useful description of $\LN(G,\alpha) \cap U$.

\begin{prop}
Let $G$ be a \tdlc group, let $\alpha \in \lnorm(G)$ and let $U$ be an open compact subgroup of $G$ such that $U$ fixes $\alpha$.  Let $L$ be a closed subgroup of $U$ such that $[L] \le \alpha$.  Then $L \le \LN(G,\alpha)$ if and only if there is a closed normal subgroup $M$ of $U$ such that $L \le M$ and $[M] = \alpha$.
\end{prop}

\begin{proof}
By Lemma~\ref{profcomm}, there is a normal subgroup $N$ of $U$ such that $[N] = \alpha$; by definition, $N \le \LN(G,\alpha)$.

Let $L$ be a closed subgroup of $U$ such that $[L] \le \alpha$.  Suppose $L \le \LN(G,\alpha)$.  By replacing $L$ with $LN$ we may assume $N \le L$.  By Proposition~\ref{prop:Brac:commensurate}, $L$ is locally normal in $G$.  In particular, $L/N$ is a finite locally normal subgroup of $U/N$, so $\N_{U/N}(L/N)$ has finite index in $U/N$, and hence $\CC_{U/N}(L/N)$ has finite index in $U/N$.  Now by Dietzmann's Lemma, it follows that the normal closure $M/N$ of $L/N$ in $U/N$ is finite and hence closed, so $M$ is a closed normal subgroup of $U$ such that $[M] = \alpha$.
\end{proof}

\begin{cor}\label{lem:normal_representative}
Let $U$ be a profinite group and let $L$ be a closed subgroup.  Assume that $L$ is commensurated and locally normal in $U$. Then the abstract normal closure of $L$ in $U$ is closed, and contains $L$ as a finite index subgroup.  
\end{cor}

\begin{proof}
Let $N$ be the abstract normal closure of $L$. Since $\alpha = [L]$ is fixed by $G$ by hypothesis, we have $N \leq \LN(G,\alpha)$. The previous proposition therefore implies that $\alpha$ has a representative $M$ that is normal in $U$ and contains $L$. Thus $M$ is closed, the index of $L$ in $M$ is finite and we have $L \leq N \leq M$. Therefore $N$ is closed and contains $L$ as a finite index subgroup. 
\end{proof}

\addcontentsline{toc}{section}{\protect\numberline{}Index}
\printindex

\end{document}